%% file: ConstructionsMain.tex
\documentclass[a4paper, american, 11pt, reqno]{amsart}

\usepackage[latin1]{inputenc}
\usepackage{mathtools}
\usepackage{mathabx}
\usepackage{amsmath}
\usepackage{amssymb}
\usepackage{hyperref}
\usepackage{extarrows}
\usepackage{url}
\usepackage{babel}
\usepackage{pgf,tikz,pgfplots}
\pgfplotsset{compat=1.16}

\usepackage[notocbasic]{nomencl}
\makenomenclature

\renewcommand{\nompreamble}{Below are some of the symbols most widely used in the body of the article. We organize them according either to the section in which they are first introduced, or the one in which they feature most prominently.}

\usepackage{ifthen}
  \renewcommand{\nomgroup}[1]{%
  \vspace{10pt}
  \item[\bfseries
  \ifthenelse{\equal{#1}{A}}{\ref{SectionIntroduction} Introduction}{%
  \ifthenelse{\equal{#1}{B}}{\ref{SectionPatchwork} Viro's patchworking method}{%
  \ifthenelse{\equal{#1}{C}}{\ref{SectionConstructionMethod} The construction method}{%
  \ifthenelse{\equal{#1}{D}}{\ref{SectionComputingAsymptoticBettiNumbers} Computing asymptotic Betti numbers}{%
  \ifthenelse{\equal{#1}{E}}{\ref{SectionAsymptoticallyLargeBettiNumbers} Asymptotically large Betti numbers}{}}}}}%
  ]}
 
\usetikzlibrary{decorations.pathmorphing}

\usepackage[a4paper]{geometry}
\usepackage{graphicx}

\usepackage{tikz, float} \usetikzlibrary {positioning}
\usetikzlibrary{shapes.misc}
\usetikzlibrary{arrows}

\usetikzlibrary{calc}
\usepackage{amsfonts}
\input xy
\xyoption{all}

\newcommand{\Z}{{\mathbb Z}}
\newcommand{\NN}{{\mathbb N}}
\newcommand{\PP}{{\mathbb P}}

\newcommand{\C}{{\mathbb C}}

\newcommand{\R}{{\mathbb R}}

\DeclareMathOperator{\Conv}{Conv}

\newtheorem{thm}{Theorem}[section]

\newtheorem{proposition}[thm]{Proposition}
\newtheorem{lemma}[thm]{Lemma}

\newtheorem{remark}[thm]{Remark}

{\theoremstyle{definition}
}
{\theoremstyle{definition}

}

 \numberwithin{equation}{section}

\tikzset{%
  add/.style args={#1 and #2}{to path={%
 ($(\tikztostart)!-#1!(\tikztotarget)$)--($(\tikztotarget)!-#2!(\tikztostart)$)%
  \tikztonodes}}
} 

\newcommand{\comment}[1]{}

\begin{document}
\title[]{Patchworking real algebraic hypersurfaces with asymptotically large Betti numbers}

\author[Charles Arnal]{Charles Arnal}

\address{Charles Arnal, Univ. Paris 6, IMJ-PRG, France.}
\email{charles.a.arnal@gmail.com} 

\maketitle

\begin{abstract}
In this article, we describe a recursive method for constructing a family of real projective algebraic hypersurfaces in ambient dimension $n$ from families of such hypersurfaces in ambient dimensions $k=1,\ldots,n-1$. The asymptotic Betti numbers of real parts of the resulting family can then be described in terms of the asymptotic Betti numbers of the real parts of the families used as ingredients.

The algorithm is based on Viro's Patchwork \cite{ViroPatchworking} and inspired by I. Itenberg's and O. Viro's construction of asymptotically maximal families in arbitrary dimension \cite{IV}.

Using it, we prove that for any $n$ and $i=0,\ldots,n-1$, there is a family of asymptotically maximal real projective algebraic hypersurfaces $\{Y^n_d\}_d$ in $\R \PP ^n$ (where $d$ denotes the degree of $Y^n_d$) such that the $i$-th Betti numbers $b_i(\R Y^n_d)$ are asymptotically strictly greater than the $(i,n-1-i)$-th Hodge numbers $h^{i,n-1-i}(\C Y^n _d)$.
We also build families of real projective algebraic hypersurfaces whose real parts have asymptotic (in the degree $d$) Betti numbers that are asymptotically (in the ambient dimension $n$) very large.

\end{abstract}

\tableofcontents

\newcommand\blfootnote[1]{%
  \begingroup
  \renewcommand\thefootnote{}\footnote{#1}%
  \addtocounter{footnote}{-1}%
  \endgroup
}
\blfootnote{The author was supported by the DIM Math Innov de la r\'egion Ile-de-France.

\begin{center}
\includegraphics[width = 10mm]{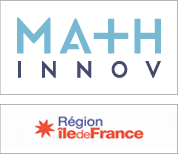} 
\end{center} 
}


\nomenclature[A,1]{\(a_i^n\)}{Leading coefficient of $h^{i,n-1-i}(\C X)$ seen as a polynomial in $d$, where $\C X$ is the complex part of a smooth real projective algebraic hypersurface of degree $d$ and dimension $n-1$.}
\nomenclature[A,3]{\(b_i(M)\)}{The $i$-th Betti number $  \dim_{\Z _2}H_i(M,\Z_2)$ of topological space $M$.}
\nomenclature[A,4]{\(b_*(M)\)}{The sum $\sum_{i\in\mathbb{N}}b_i(M) $ of all Betti numbers of topological space $M$.}
\nomenclature[A,5]{\(\Delta (P)\)}{Newton polytope of real Laurent polynomial $P$.}
\nomenclature[A,6]{\(H_i(M)\)}{\sloppy The $i$-th homology group $H_i(M,\Z_2)$ with coefficients in $\Z_2 = \Z/2\Z$ of topological space $M$.}
\nomenclature[A,71]{\( h^{p,q}(\C X)\)}{The $(p,q)$-th Hodge number of smooth complex projective variety $\C X$.}
\nomenclature[A,72]{\(P^k_d\) and  \(Q^n_d\)}{Throughout the text, $P^k_d$ usually denotes polynomials used as ingredients for the Cooking Theorem, while $Q^n_d$ usually denotes the polynomials obtained by applying it, with $n$ being the number of variables and $d$ the degree.}
\nomenclature[A,8]{\(\R X \)}{The real part of real algebraic variety $X$.}
\nomenclature[A,90]{\( \C X\)}{The complex part of real algebraic variety $X$.}
\nomenclature[A,911]{\( V_{(K^*)^n}(P)\)}{Zero locus of real polynomial $P$ in the algebraic torus $(K^*)^n$, where $K\in\{\R,\C\}$.}
\nomenclature[A,912]{\( V_{KY}(P)\)}{Closure of $V_{(K^*)^n}(P)$ in the toric variety $KY$ defined via the inclusion $(K^*)^n\subset KY$, where $K\in\{\R,\C\}$.}
\nomenclature[A,92]{\( f \overset{n}{\geq} g\)}{Indicates that $f(d)\leq g(d)+\mathcal{O}(d^{n-1})$ for $f,g:\NN \longrightarrow \NN$.}
\nomenclature[A,93]{\( f \overset{n}{=} g\)}{Indicates that $f \overset{n}{\geq} g$ and $f \overset{n}{\leq} g$.}

\nomenclature[B,1]{\(\mathring{A}\)}{The relative interior of $A\subset \R^n$.}
\nomenclature[B,2]{\(K\Delta\)}{The toric variety associated with the full-dimensional polytope with integer vertices $\Delta \subset \R^n$, for $K\in \{\R,\C\}$.}
\nomenclature[B,3]{\(P^\Gamma\)}{Truncation $P(X)=\sum_{\lambda = ( \lambda_1,\ldots,\lambda_n) \in \Gamma \cap \Lambda } c_\lambda X_1^{\lambda_1},\ldots, X_n^{\lambda_n}$ of real Laurent polynomial $P(X)=\sum_{\lambda = ( \lambda_1,\ldots,\lambda_n) \in \Lambda } c_\lambda X_1^{\lambda_1},\ldots, X_n^{\lambda_n} $ for $\Gamma \subset \Z^n$.}
\nomenclature[B,4]{\(U^n_K\)}{$U^n_k=(U^1_K)^n$ for $K\in\{\R,\C\}$,   where $U^1_\C\subset\C$ is the unit circle and $U^1_\R = \{1,-1\}$. }
\nomenclature[B,5]{\(Chart_{\Delta \times U^n_K}(P)\)}{Chart of $P$  in $\Delta \times U^n_K$ - see Section \ref{SectionPatchwork} for details.}

\nomenclature[C,1]{\(e_i\)}{The $i$-th vector of the standard basis of $\R^n$.}
\nomenclature[C,2]{\(S^n_{d}\)}{The simplex $\{(x_1,\ldots,x_n) \in \R^n \text{ s.t. } x_i\geq 0 \; \text{for } i=1,\ldots,n \text{ and } \allowbreak \sum_{i=1}^n x_i\leq d \}$.}
\nomenclature[C,3]{\(S^n_{d,m}\)}{The ($n-1$)-dimensional simplex $S^n_{d,m} = S^n_d\cap\{x_n=m\}$.}
\nomenclature[C,5]{\(\tilde{\Phi}(f)\)}{Given a finite set $\Lambda \subset \Z^n$ and a function $f:\Lambda  \longrightarrow \R$, we define  $\tilde{\Phi}(f):\Conv(\Lambda) \longrightarrow \R$ as the function whose graph is the lower convex hull in $\R^{n+1}$ of the graph of $f$.}
\nomenclature[C,4]{\(R^n_{d,m,i}\)}{For $m=1,\ldots,d-n$ odd and $i=1,\ldots,n-2$, we have $R^{n}_{d,m,i}=\{(x_1,\ldots,x_i,0,\ldots,0,m) \in \R^{n} \text{ s.t. }  x_j\geq 1 \; \text{for } j=1,\ldots,i, \allowbreak \; \sum_{j=1}^{n-1} x_j  \leq d-m-1 \} $. For $m=0,\ldots,d-n$ even and $i=1,\ldots,n-2$, we have $R^{n}_{d,m,i}=\{(0,\ldots,0,x_{n-1-i},\ldots,x_{n-1},m) \in \R^{n} \text{ s.t. }  x_j\geq 1 \; \text{for } j= n-i-1,\ldots,n-1, \; \sum_{j=1}^{n-1} x_j= d-m \},$ and for $m=0,\ldots,d-n-1$ and $i=n-1$, $R^{n}_{d,m,n-1}=\{(x_1,\ldots,x_{n-1},m) \in \R^{n} \text{ s.t. }  x_j\geq 1 \; \text{for } j=1,\ldots,n-1, \; \sum_{j=1}^{n-1} x_j\leq d-m-1 \} $. }
\nomenclature[C,0]{\(A\star B\)}{The topological join of topological spaces $A$ and $B$.}

\nomenclature[D]{\(\tilde{H}_i(M)\)}{The $i$-th reduced homology group $\tilde{H}_i(M,\Z_2)$ with coefficients in $\Z_2 = \Z/2\Z$ of topological space $M$.}
\nomenclature[D]{Cycles and axes}{\sloppy For a given closed submanifold  $M\subset\R^n$, the classes $\beta_1,\ldots,\beta_r $ in $ H_{n-1-i}(\R^n\backslash M)$ (respectively, in $\ker (H_0 (\R^n \backslash M) \longrightarrow H_0(\R^n))$ if $ n-1-i=0$) are axes to the cycles $\alpha_1,\ldots,\alpha_r$ in $ H_{i}(M)$ if their linking numbers verify $l(\alpha_s,\beta_t) = \delta_{s,t}\in\Z_2$.}

\nomenclature[E]{\(D^2a_p^n\)}{The second order central finite difference $D^2a^n_p = a^n_{p+1} -2a^n_p +a^n_{p-1}$ of the coefficients $a^n_p$.}
\nomenclature[E]{\(t^n_i\)}{Given some coefficients $x^n_i\geq 0$ such that a family $\{P^n_d\}_{d\in\Z}$ of real Laurent polynomials in $n$ variables verifies $b_i(V_{\R\PP ^n}(P^n_d))  \overset{n}{\geq} x_i^n\cdot d^n$, we sometimes consider the difference $t^n_i := x^n_i -a^n_i$.  }


\input{Introduction.tex}
\subsection*{Acknowledgements}
The author is very grateful to Fr\'ed\'eric Bihan, Ilia Itenberg and Oleg Viro for helpful discussions, as well as to Mark Gross for his support and advice and to the anonymous referee for his suggestions.
\input{Patchwork}
\input{ConstructionMethod}
\input{ComputingAsymptoticBetti}

\input{AsymptoticallyLargeBetti}

\input{SomeExplicitComputations}

\input{Conclusion}

\pagebreak
\printnomenclature

\bibliographystyle{alpha}
\bibliography{biblio}

\end{document}

%% file: Introduction.tex
\section{Introduction}\label{SectionIntroduction}

Our goal in this article is to describe a new construction method for real projective algebraic hypersurfaces, and use it to further our understanding of the possible asymptotic behavior of the Betti numbers of families of real projective algebraic hypersurfaces in the spirit of Hilbert's 16-th problem.

Many constraints on the topology of real algebraic varieties are known, such as the famous Smith-Thom inequality
\begin{equation}\label{SmithThom}
    \sum_i \dim_{\Z _2}H_i(\R X) \leq \sum_i \dim_{\Z _2}H_i(\C X) ,
\end{equation}
where $\R X$ (respectively, $\C X$) denotes the real part (respectively, complex part) of a real algebraic variety $X$, and homology is taken with coefficients in $\Z_2$ (this is assumed from now on, unless specified otherwise). If inequality (\ref{SmithThom}) is an equality, $X$ is said to be \textit{Smith-Thom maximal}, or simply \textit{maximal}.

Nonetheless, very little is currently known for varieties of dimension higher than $3$.

Besides searching for new constraints, one can try to build "exotic" or "extreme" examples in order to further our understanding of topology of real algebraic varieties. O. Viro's Patchwork, described in full details in \cite{ViroPatchworking}, has proved a powerful tool in that regard; it allows one to build varieties with prescribed topology by patchworking, \textit{i.e.} gluing together simpler varieties. We make extensive use of the Patchwork in this article, and relevant definitions and results are stated in Section \ref{SectionPatchwork}.

The patchworking method was for example used by Viro himself to  classify, up to isotopy, the non-singular curves of degree $7$ in the real projective plane in \cite{Viro80}, as well as to disprove the famous \textit{Ragsdale conjecture} (from \cite{Rags}) regarding real algebraic curves (also in \cite{Viro80}), and later in \cite{Itenberg93} by I. Itenberg to show that this same conjecture was asymptotically wrong as well. It was also one of the starting points of tropical geometry.

More recently, the Patchwork has been successfully used to build families of real algebraic toric hypersurfaces with interesting asymptotic properties, in the following sense.

In ambient dimension $n$, the dimension of the total homology of the complex part of a smooth real projective algebraic hypersurface $X$ of degree $d$ satisfies
\begin{equation}\label{ComplexPartTotalHomology}
\sum_i \dim_{\Z _2}H_i(\C X)= \frac{(d-1)^{n+1}-(-1)^{n+1}}{d}+n+(-1)^{n+1}.
\end{equation}
In particular, it is a polynomial of degree $n$ in $d$, with $1$ as its leading coefficient (see \cite{DanilovKhovansky}).
Moreover, for $i=0,\ldots,n-1$, the $(i,n-1-i)$-th Hodge number $h^{i,n-1-i}(\C X)$ is also a polynomial of degree $n$ in $d$ (the same for any such hypersurface). Denote its leading coefficient by $a_i^n$. 

If $f,g:\NN \longrightarrow \NN$ are such that $f(d)\leq g(d)+\mathcal{O}(d^{n-1})$, using the usual convention for the $\mathcal{O}$ notation, we write $f\overset{n}{\leq} g$. If both $f\overset{n}{\leq}g$ and $f\overset{n}{\geq}g$, we say that $f\overset{n}{=}g$. We naturally extend this notation to the case where $f$ and $g$ are both defined on the same infinite subset of $\NN$.
Using that notation, we already know from (\ref{SmithThom}) and (\ref{ComplexPartTotalHomology}) that 
$$\max_{X^n_d \in \mathcal{X}^n_d}\left[ \sum _i \dim_{\Z _2}  H_i(\R X^n_d) \right]\overset{n}{\leq}d^n,$$
where $\mathcal{X}_d^n$ is the set of all smooth real algebraic hypersurfaces of degree $d$ in $\PP^n$.

It is then natural to ask what the maximal possible value $B_i^n(d)$ of the $i$-th Betti number $b_i(\R X^n_d):=\dim_{\Z _2}(H_i(\R X^n_d))$ of the real part of a smooth real algebraic hypersurface $X^n_d\subset \PP^n$ of degree $d$ is. The asymptotic value in the degree $d$ of $B^n_i(d)$ is also of interest. Using Viro's Patchwork, F. Bihan showed in \cite{BihanAsymptotic} (in dimension $3$, but the proof easily generalizes to arbitrary dimension)
that there exists $\beta_i^n \in\R_{>0}$ such that $B_i^n(d)= \beta_i^n\cdot d^n +o(d^n)$. The same question can be asked about linear combinations of Betti numbers, or under additional conditions.

There is a principle of sorts that suggests that for a real projective algebraic hypersurface $X$ in ambient dimension $n$, we should have for all $i=0,\ldots,n-1$
\begin{equation} \label{EquationPrinciple}
b_i(\R X)   \leq \sum_p h^{i,p}(\C X) = h^{i,n-1-i}(\C X)+ 1 -\delta_{i,\frac{n-1}{2}} ,
\end{equation}
\sloppy where $\delta_{i,\frac{n-1}{2}}$ is $1$ if $i=\frac{n-1}{2}$ and $0$ otherwise, and the equality $\sum_p h^{i,p}(\C X) = h^{i,n-1-i}(\C X)+ 1-\delta_{i,\frac{n-1}{2}}$ is a consequence of Lefschetz's hyperplane theorem.
\sloppy Inequality (\ref{EquationPrinciple}) does not hold in general; however, it does in many reasonable cases, as illustrated below.

As a correction to the Ragsdale conjecture mentioned above and that he had disproved, Viro conjectured in \cite{Viro80} that  any smooth compact real algebraic surface $X$ whose complex part is simply connected verifies 
\begin{equation}\label{ViroConjecture}
 b_1(\R X)   \leq h^{1,1}(\C X). 
\end{equation}
Though it turned out to be wrong, Itenberg showed in \cite{Itenberg1997} that Inequality (\ref{ViroConjecture}) does hold for any smooth real algebraic surface $X$ in $\PP^2$ obtained using the combinatorial Patchwork and a so-called \textit{maximal triangulation} (this is NOT synonymous with $X$ being Smith-Thom maximal).

Under stricter conditions but in arbitrary dimension, A. Renaudineau and K. Shaw proved in \cite{RS} and \cite{ARS} that any smooth real algebraic hypersurface $X$ in $\PP^n$ obtained using the combinatorial Patchwork and a  \textit{primitive triangulation}  verifies Inequality (\ref{EquationPrinciple}).
Itenberg and Viro had previously shown in \cite{IV} an asymptotic version of that inequality, \textit{i.e.} that any smooth real algebraic hypersurface $X$ of degree $d$ in $\PP^n$ obtained using the combinatorial Patchwork and a  primitive triangulation satisfies $b_i(\R X)   \leq \sum_p h^{i,p}(\C X) +\mathcal{O}(d^{n-1})$, where the term $\mathcal{O}(d^{n-1})$ is independent from the choice of $X$.

A family of real algebraic hypersurfaces $\{X^n_d\}_{d\in\NN}$ in $\PP^n$ is \textit{asymptotically maximal} if $\sum_i b_i(\R X^n_d) \overset{n}{=} \sum_i b_i(\C X^n_d)  (  \overset{n}{=} d^n)$.
We also say that a family of real algebraic hypersurfaces $\{Y^n_d\}_{d\in\NN}$ in $\PP^n$ is \textit{asymptotically standard} (this is, ironically, non-standard terminology) if it verifies $b_i(\R Y_n^d)  \overset{n}{=} a_i^n\cdot d^n$ for all $i=1,\ldots,n-1$ (which implies asymptotic maximality) - in particular, such a family asymptotically obeys the principle enounced above in Inequality (\ref{EquationPrinciple}).

Asymptotically standard families are in a sense the baseline examples of asymptotically maximal families, since they are the easiest to build and the "least singular".
It is natural to compare the asymptotic Betti numbers of any asymptotic family of real projective hypersurfaces to the asymptotically standard case.

In \cite{IV}, Itenberg and Viro constructed for any $n$ an asymptotically standard family of real algebraic hypersurfaces $\{X^n_d\}_{d\in\NN}$ in $\PP^n$.
B. Bertrand achieved similar results with general toric varieties, as well as complete intersections, in \cite{BertrandAsymptotic}. In \cite{BihanAsymptotic}, Bihan gave good lower bounds on the values of $\beta_i^n$ for $n=3$, which E. Brugall\'e further improved in \cite{Brugalle2006} using the same method. Renaudineau also worked on related problems in his thesis \cite{Renaudineauthesis}.
All of these results made use of the patchworking method.

In the same spirit, we develop a construction technique based on Viro's method and inspired by \cite{IV} such that, given for each $k=1,\ldots,n-1$ a family of projective smooth real algebraic hypersurfaces in $\PP^k$, which we call "ingredients", we can use them to "cook" (construct) a family $\{Y^n_{d}\}_{d\in \NN}$ of smooth real algebraic hypersurfaces in $\PP^n$ such that the asymptotic Betti numbers of $\{\R Y^n_{d}\}_{d\in \NN}$ can be computed from those of the real parts of the hypersurfaces used as ingredients.

A real Laurent polynomial $P\in \R[X_0^\pm, \ldots, X_n^\pm]$ gives rise, \textit{via} its zero locus, to a real algebraic hypersurface with complex points in the algebraic torus $(\C^*)^n$ and real points in $(\R^*)^n$, which we denote as $V_{(\C^*)^n}(P)$ and $V_{(\R^*)^n}(P)$, respectively.

For any toric variety $Y$ (see W. Fulton's book \cite{FultonToric} for more on toric varieties, or Viro's article \cite{ViroPatchworking} for an exposition focused on the questions with which we concern ourselves in this text), we define $V_{\C Y}(P)$ (respectively, $V_{\R Y}(P)$) as the closure of $V_{(\C^*)^n}(P) \subset (\C^*)^n \subset \C Y$ (respectively, $V_{(\R^*)^n}(P)\subset (\R^*)^n \subset \R Y$) in $\C Y$ (respectively, $\R Y$) in the Zariski topology.
The sets $V_{\C Y}(P)$ and $V_{\R Y}(P)$ can be seen as the complex and real points of the same real algebraic object $V_{Y}(P)$, the \textit{real algebraic hypersurface in $Y$}  associated to $P$.

Given a real Laurent polynomial $P(X)=\sum_{\lambda = ( \lambda_1,\ldots,\lambda_n) \in \Lambda } c_\lambda X_1^{\lambda_1},\ldots, X_n^{\lambda_n} $, where $\Lambda$ is a finite subset of $M\cong \Z ^n$ and $c_\lambda \in \R^*$ for all $\lambda \in \Lambda$, we call the convex hull in $M_\R \cong \R^n$ of $\Lambda$ the \textit{Newton polytope} of $P$, and denote it by $\Delta(P)$. Let $\Gamma \subset  \R^n$. We define the \textit{truncation} $P^\Gamma $ as the polynomial $P(X)=\sum_{\lambda = ( \lambda_1,\ldots,\lambda_n) \in \Gamma \cap \Lambda } c_\lambda X_1^{\lambda_1},\ldots, X_n^{\lambda_n} $.

The real Laurent polynomial $P$ is \textit{completely nondegenerate} over $K$ if $V_{(K^*)^n}(P^\Gamma)$ is a nonsingular hypersurface for any face $\Gamma$ of its Newton polytope $\Delta(P)$ (including $\Delta(P)$ itself).

\sloppy We have the following "cooking" theorem, where we let $S_d^k := \{(x_1,\ldots,x_k) \in \R^k |\sum_{i=1}^k x_i\leq d, \; x_i\geq 0 \; \text{for } i=1,\ldots,n \}$ be the simplex of side $d$ and dimension $k$:

\begin{thm}[Cooking Theorem]\label{TheoremConstructionsMainTheoremConstructions}
Let $n\geq 2$. For $k=1,\ldots,n-1$, let $\{P^k_{d}\}_{d\in \NN}$ be a family of completely nondegenerate real Laurent polynomials in $k$ variables, such that $P^k_d$ is of degree $d$ and that the Newton polytope $\Delta(P^k_d)$ of $P^k_d$ is $S^k_d$.
Suppose additionally that for $k=1,\ldots,n-1$ and $i=0,\ldots,k-1$,
 $$b_i(V_{\R\PP ^k}(P^k_d))\overset{k}{\geq} x_i^k \cdot d^k$$
for some $x_i^k \in \R_{\geq 0}$.
Then there exists a family $\{Q^n_d\}_{d\in\NN}$ of completely nondegenerate real Laurent polynomials in $n$ variables such that $\Delta(Q^n_d)=S^n_d$ and such that for $i=0,\ldots,n-1$ 
\begin{equation}\label{FormulaConstructionsMainFormula}
b_i(V_{\R\PP ^n}(Q^n_d))   \overset{n}{\geq}   \frac{1}{n}\left(x_i^{n-1}+x_{i-1}^{n-1} +\sum_{k=1}^{n-2}\sum_{j=0}^{i-1} x_j^k \cdot x_{i-1-j}^{n-1-k}\right) \cdot d^n,
\end{equation}
where $x_j^k$ is set to be $0$ for $j\not \in \{0,\ldots,k-1\}$.

Moreover, if the families $\{P^k_{d}\}_{d\in \NN}$ were obtained using a combinatorial patchworking for all $k$, then the family $\{Q^n_{d}\}_{d\in \NN}$ can also be obtained by combinatorial patchworking.

If each family $\{P^k_{d}\}_{d\in \NN}$ (for $k=1,\ldots,n-1$) is such that the associated family of projective hypersurfaces is asymptotically maximal, then the family of projective hypersurfaces associated to $\{Q^n_{d}\}_{d\in \NN}$ is also asymptotically maximal.
\end{thm}

\begin{remark}
In light of Lemma \ref{LemmaConstructionsEquivalenceOuvertFerme} and Remark \ref{RemarkConstructionsEquivalenceOuvertFerme} below, the expressions $b_i(V_{\R\PP ^k}(P^k_d))$ for $k=1,\ldots, n-1$ and $b_i(V_{\R\PP ^n}(Q^n_d))$ can be indifferently (and independently) replaced in the statement by $b_i(V_{(\R^*)^k }(P^k_d))$ and $b_i(V_{(\R^*)^n }(Q^n_d))$ respectively.

For the same reasons, one can see that the polynomials $P^k_d$ do not actually need to be completely nondegenerate; we only need the associated hypersurfaces $V_{(\C^*)^k}(P^k_d)$ in the complex torus to be smooth.
\end{remark}

One can get varying, and potentially interesting, families of real projective algebraic hypersurfaces in high ambient dimension by starting with various low-dimensional families of hypersurfaces and applying the Cooking Theorem \ref{TheoremConstructionsMainTheoremConstructions} recursively: each application yields a new family in some dimension $n$, which can then serve as an ingredient for higher dimensional constructions.
One advantage of that method is that each new family in ambient dimension $N$ with "good" asymptotic Betti numbers obtained using other means can potentially automatically give rise, through repeated applications of the Cooking Theorem, to new interesting families in all dimensions greater than $N$.

In particular, we make use of already existing families of projective smooth real algebraic hypersurfaces in $\PP^3$ designed using Bihan's results from \cite{BihanAsymptotic} by Brugall\'e in \cite{Brugalle2006} to build asymptotic counter-examples to Inequality (\ref{EquationPrinciple}), \textit{i.e.}  families of real projective algebraic hypersurfaces such that the $i$-th Betti numbers (for a given $i$) of their real parts are asymptotically much larger than the sum over $p$ of the ($p,i$)-th Hodge numbers of their complex parts - in other words, asymptotically non-standard families.
We prove the two following theorems.

\begin{thm}\label{TheoremConstructionsApplicationThm}
For any $n\geq 3$ and any $i=0,\ldots,n-1$, there exists  $b^n_i>a^n_i$ and a family $\{Q^n_d\}_{d\in\NN}$ of completely nondegenerate real Laurent polynomials in $n$ variables such that $\Delta(Q^n_d)=S^n_d$, that the associated family of real projective hypersurfaces is asymptotically maximal and that 
\begin{equation*}
b_i(V_{\R\PP^n }(Q^n_d))   \overset{n}{\geq}  b^n_i \cdot d^n.   
\end{equation*}
\end{thm}
Hence $b_i(V_{\R\PP^n }(Q^n_d))$ grows asymptotically strictly faster than the corresponding Hodge number $h^{i,n-1-i}(V_{\C\PP^n }(Q^n_d))$, though $b^n_i$ cannot be expected to be particularly large compared to $a^n_i$ (see the end of Subsection \ref{SubsectionConstructionsFirstConstruction} for more details on that).
As far as the author is aware, this had not yet been achieved.

The second theorem, which we prove using probabilistic methods, allows us to find asymptotic (in the degree $d$) results that are asymptotically (as the ambient dimension $n$ goes to infinity) much better.

\begin{thm}\label{TheoremConstructionsBonneAsymptotique}

Let $N\geq 1$. For $k=1,\ldots,N$, let $\{P^k_{d}\}_{d\in \NN}$ be a family of completely nondegenerate real Laurent polynomials in $k$ variables such  that the Newton polytope $\Delta(P^k_d)$ of $P^k_d$ is $S^k_d$.
Suppose additionally that for $k=1,\ldots,N$ and $i=0,\ldots,k-1$,
 $$b_i(V_{\R\PP ^k}(P^k_d))\overset{k}{=} x_i^k \cdot d^k$$
for some $x_i^k \in \R_{\geq 0}$ such that $\sum_{i=0}^{k-1}x_i^k =1$ (in particular, the family of projective hypersurfaces associated to each family $\{P^k_d\}_{d\in \NN}$ is asymptotically maximal). Set also $x_i^k$  to be $0$ for $i\not \in \{0,\ldots,k-1\}$.

Define
$$\sigma ^2 :=\frac{2}{(N+1)(N+2)} \left(\frac{1}{4} +\sum_{k=1}^N\sum_{i=0}^{k-1}x^k_i\left(i- \frac{k-1}{2}\right)^2 \right). $$

Then for every $n\geq N+1$ and any $i\in \Z$, there exist $x_i^n\in \R_{\geq 0}$  and a family $\{Q^n_d\}_{d\in\NN}$ of completely nondegenerate real Laurent polynomials in $n$ variables such that $\Delta(Q^n_d)=S^n_d$, that for $i\in \Z$ 
\begin{equation*}
b_i(V_{\R\PP ^n}(Q^n_d))   \overset{n}{=}   x_i^n \cdot d^n
\end{equation*}
and such that for any $m\in  \Z$ we have
\begin{equation}
x^n_m = \frac{1}{\sigma\sqrt{2\pi}}\frac{1}{\sqrt{n}} \exp\left(-\frac{\left(\frac{n-1}{2} - m\right)^2}{2n\sigma^2}\right) +o\left(n^{-\frac{1}{2}}\right),
\end{equation}
where the $o(1)$ error term is uniform in $m$.
The family of projective hypersurfaces associated to each family $\{Q^n_d\}_{d\in \NN}$ is also asymptotically maximal.
\end{thm}

As it is known (see Formula (\ref{FormulaConstructionsPolya})) that 
\begin{equation}\label{FormulaConstructionsPolyaIntroduction}
a^{n}_{\left \lfloor {\frac{n-1}{2} + x\sqrt{n}} \right \rfloor}= \sqrt{\frac{6}{\pi(n+1)}} \exp\left(-6x^2\right) + \mathcal{O}\left({n^{-\frac{3}{2}}}\right),
\end{equation}
this direct corollary of the theorem clearly shows its usefulness.
\begin{thm}\label{CorollaryConstructionsBonneAsymptotique}
For any $n\geq 3$, there exists families $\{F^{+,n}_{d}\}_{d\in \NN}$ and $\{F^{-,n}_{d}\}_{d\in \NN}$ of completely nondegenerate real Laurent polynomials in $n$ variables such that the Newton polytope $\Delta(F^{\pm,n}_{d})$ is $S^n_d$, as well as reals $c^n_i,d^n_i\in \R_{\geq 0}$ for any $i\in\Z$, such that for $i=0,\ldots,n-1$, we have
$$b_i(V_{\R\PP ^n}(F^{+,n}_{d}))\overset{n}{=} c^n_i\cdot d^n$$
and
$$b_i(V_{\R\PP ^n}(F^{-,n}_{d}))\overset{n}{=} d^n_i \cdot d^n$$
and such that we have, for all $x\in \R$, that 
$$ c^n_{\left \lfloor {\frac{n-1}{2} + x\sqrt{n}} \right \rfloor} = \frac{2}{\sqrt{\pi}}\frac{1}{\sqrt{n}}\exp\left(-4x^2\right) + o\left({n^{-\frac{1}{2}}}\right)$$
and 
$$ d^n_{\left \lfloor {\frac{n-1}{2} + x\sqrt{n}} \right \rfloor} = \frac{\sqrt{20}}{\sqrt{3\pi}}\frac{1}{\sqrt{n}}\exp\left(\frac{-20x^2}{3}\right) + o\left({n^{-\frac{1}{2}}}\right),$$
where the error terms $o\left({n^{-\frac{1}{2}}}\right)$ are uniform in $x$.
The family of projective hypersurfaces associated to each family $\{F^{\pm,n}_{d}\}_{d\in \NN}$ is also asymptotically maximal.

\end{thm}

\begin{remark}
Equation (\ref{FormulaConstructionsPolyaIntroduction}) tells us that for large $n$, the coefficients $\{a^n_i\}_{i\in\Z}$ are well approximated by a discrete Gaussian of sorts, while the coefficients $\{c^n_i\}_{i\in\Z}$ (respectively  $\{d^n_i\}_{i\in\Z}$) from  Theorem \ref{CorollaryConstructionsBonneAsymptotique} correspond to a more spread out Gaussian (respectively, to a Gaussian that is more concentrated around $\frac{n-1}{2}$).

In particular, for $x=0$, compare with Formula (\ref{FormulaConstructionsPolyaIntroduction}) and see that for $n$ odd,
$$\frac{d^n_{\frac{n-1}{2}}}{a^n_{\frac{n-1}{2}}}=\frac{\sqrt{10}}{3} +o(1),$$
which is strictly greater than $1$ for $n$ large enough.
\end{remark}

These results are currently the only known "counterexamples" in general dimension to the principle presented above, which suggested that real projective algebraic hypersurfaces should be expected to  verify
\begin{equation*}
\dim_{\Z_2} H_i(\R X;\Z_2)  \leq  h^{i,n-1-i}(\C X)+1 -\delta_{i,\frac{n-1}{2}}.
\end{equation*}

This article is organized as follows: a brief exposition of Viro's patchworking method is to be found in Section \ref{SectionPatchwork}, including Viro's Main Patchwork Theorem, some elementary results on toric varieties, and a few useful related definitions (convex subdivisions, charts, primitive triangulations, etc.).

The construction method is detailed in Section \ref{SectionConstructionMethod}. More specifically, we first define in Subsection \ref{SubsectionConstructionsConvexSubdivision} a specific convex triangulation of the simplex  $S_d^{n}$. Using the ingredients $\{P^k_d\}$ of the Cooking Theorem, we also define in Subsection \ref{SubsectionConstructionsCoefficientsChoice} a certain family $ \{\tilde{Q}^n_d \}_{d\in\NN}$ of real Laurent polynomials. We then apply in Subsection \ref{SubsectionConstructionsDefinitionPnd} Viro's patchworking method to $\{\tilde{Q}^{n}_d\}_{d\in\NN}$ and the convex triangulation we devised in Subsection \ref{SubsectionConstructionsConvexSubdivision} and obtain  a family $\{Q^n_d\}_{d\in\NN}$ of real Laurent polynomials.

We prove the Cooking Theorem \ref{TheoremConstructionsMainTheoremConstructions} in  Section \ref{SectionComputingAsymptoticBettiNumbers}. We start by introducing in Subsection \ref{SubsectionConstructionsPreliminaries} the key notions of linking numbers and axes associated to a collection of homological cycles - the use of axes allows us to prove that certain collections of cycles are independent in the homology of the real part of the hypersurfaces that we consider, and therefore that its Betti numbers are large enough. We also show that enough ``convenient" cycles and axes can be found in the hypersurfaces used as ingredients for the Cooking Theorem. In Subsections \ref{SubsectionConstructionsCyclesSuspension} and \ref{SubsectionConstructionsCyclesJoin}, we explain how the cycles and axes that were found in the ingredients are joined and suspended in the hypersurfaces corresponding to $\{Q^n_d\}_{d\in\NN}$ to give rise to new cycles and axes. Finally, we count those cycles and axes and prove that there are asymptotically enough of them for the Cooking Theorem \ref{TheoremConstructionsMainTheoremConstructions} to be true. 

In Subsection \ref{SubsectionConstructionsHodgeNumbersProperties}, we make some useful observations regarding Hodge numbers, their asymptotic behaviour and their relations to more combinatorial objects, such as Eulerian numbers and hypercube slices. We then introduce in Subsection \ref{SubsectionConstructionsNotationsAndKnownResults} the families of projective real algebraic
hypersurfaces constructed by Itenberg and Viro in \cite{IV} and by Brugall\'e in \cite{Brugalle2006}. They are the main ingredients of our first family of asymptotic constructions using the Cooking Theorem - we describe it and prove that it satisfies Theorem \ref{TheoremConstructionsApplicationThm} in Subsection \ref{SubsectionConstructionsFirstConstruction}. 
We recursively build in Subsection \ref{SubsectionConstructionsSecondConstruction} our second family of asymptotic constructions. The Betti numbers of the hypersurfaces thus obtained are described by a recursive formula - we associate them with some ad hoc probability distributions and use probabilistic and analytical techniques (including a variant of the Local Limit Theorem) to understand their asymptotics and prove Theorem \ref{TheoremConstructionsBonneAsymptotique}, of which Theorem \ref{CorollaryConstructionsBonneAsymptotique} is a simple corollary.

Explicit approximations of the largest $c_i^n$ (using the notations of Theorem \ref{TheoremConstructionsApplicationThm}) that we were able to get for small $n$ are given in Section \ref{SectionExplicitComputations}. Finally, some closing observations are made in Section \ref{SectionConclusion}. An index of notation can also be found before the bibliography.


%% file: Patchwork.tex
\section{Viro's patchworking method}\label{SectionPatchwork}

We give a quick description of Viro's method so as to have the necessary definitions, mostly paraphrasing Viro's own exposition in \cite{ViroPatchworking}. All omitted details can be found there, or in Itenberg's, G. Mikhalkin's and E. Shustin's notes on Tropical Algebraic Geometry \cite{TextBookTropicalGeometry}. See \cite{FultonToric} for more on toric varieties.

In what follows, $K$ can be either $\R$ or $\C$. Let $U^1_\C \subset \C$ be the unit circle, $U^1_\R$ be $\{ 1,-1 \}$ and define $U^n_K:=(U^1_K)^n$ for $n\in \NN$. We also use the notation $z^\lambda:=z_1^{\lambda_1}\ldots z_n^{\lambda_n}$, where $z=(z_1\ldots,z_n)\in\C$ and $\lambda=(\lambda_1,\ldots,\lambda_n)\in \Z$.

Let $\Delta \subset \R^n$ be a full-dimensional polytope with integer vertices; we denote as $K\Delta$ the toric variety to which the normal fan $\Sigma$ of $\Delta$ gives rise.
The inclusion $\Delta \subset \R^n$ determines an embedding $(K^*)^n\subset K\Delta$, and there is a natural action of the torus $(K^*)^n\subset K\Delta$ on itself by multiplication, which can be naturally extended to an action $S: K\Delta \times (K^*)^n   \longrightarrow K\Delta$ on $K\Delta$.
Moreover, the variety $K\Delta$ is stratified along the closures of the orbits of the action of the algebraic torus, and those strata are in  bijection with the faces of $\Delta$.
Let $\Gamma$ be a face of $\Delta$, and denote by $K\Gamma$ the corresponding stratum of $K \Delta$. It can be shown that $K\Gamma$ is isomorphic to the toric variety to which $\Gamma$, seen as a full-dimensional polytope in the vector space that it spans, gives rise, which justifies the notation.

There is a stratified (along its intersections with the strata $K\Gamma$) subspace of $K\Delta$, denoted as $\R_+ \Delta$, which corresponds to the points in $K \Delta$ with real nonnegative coordinates (this can be given a precise meaning with the definition of $K\Delta$).
We can see $K\Delta$ as a quotient of $\R_+ \Delta \times U^n_K$ via the map $S: \R_+ \Delta \times U^n_K \longrightarrow K\Delta$, where $S$ is the restriction to $\R_+ \Delta \times U^n_K$ of the action $S: K\Delta \times (K^*)^n \longrightarrow K\Delta$ mentioned above.

As stratified topological spaces, $\R_+ \Delta$ and the polytope $\Delta$ are homeomorphic, for example via the Atiyah moment map $M :\R_+ \Delta \longrightarrow \Delta$ (see \cite{AtiyahMomentMap}).
If $x\in (\R_+)^n\subset \R_+\Delta$ and $w_1,\ldots, w_k \in \Z^n$ are such that their convex hull is $\Delta$, then
\begin{equation*}\label{MomentMap}
M(x)=\frac{\sum_{i=1}^k|x^{w_i}|w_i}{\sum_{i=1}^k|x^{w_i}|}.
\end{equation*}

Thus we have the following map
\begin{equation*}
  \Phi :  \Delta \times U^n_K  \xlongrightarrow{M^{-1}\times id}  \R_+ \Delta \times U^n_K \xlongrightarrow{S} K\Delta
\end{equation*}
through which $K\Delta$ is seen as a quotient of $\Delta \times U^n_K $ (when $K=\R$, this quotient can in fact be quite nicely described in terms of an appropriate gluing of the faces of $\R_+ \Delta \times U^n_\R $ - see \cite{ViroPatchworking}). It restricts to a homeomorphism from $  \mathring{\Delta} \times U^n_K$ to $(K^*)^n \subset K\Delta$.

Let $P$ be as above a real Laurent polynomial in $n$ variables and $\Delta \subset \R^n$ be a full-dimensional polytope with integer vertices.
The stratified topological pair $(\Delta \times U^n_K ,v)$, where $v=\Phi^{-1}(V_{K\Delta}(P))$, is called a \textit{chart} of $P$. A slightly different definition exists, where $M$ can be replaced in the definition of $\Phi$ by any "nice enough" homeomorphism.

When there is no possible confusion, we sometimes refer to $v$ itself as the chart (as opposed to the pair $(\Delta \times U^n_K ,v)$), and we denote it as $Chart_{\Delta \times U^n_K}(P)$.
By extension, we also write $Chart_{\mathring{\Delta} \times U^n_K}(P) :=Chart_{\Delta \times U^n_K}(P)\cap (\mathring{\Delta} \times U^n_K)$, where $\mathring{\Delta}$ is the relative interior of $\Delta$ in $\R^n$.

A \textit{convex subdivision} (or \textit{regular subdivision}) $T$ of an $n$-dimensional convex polytope $\Delta\subset \R^n$ with integer vertices is a finite family $\{\Delta_i\}_{i\in I}$ of $n$-dimensional convex polytopes with integer vertices such that:
\begin{itemize}
    \item $\bigcup_{i\in I} \Delta_i = \Delta$.
    \item For all indices $i, j \in I$, the intersection $\Delta_i \cap \Delta_j$ is either empty or a face of both $\Delta_i$ and $\Delta_j$.
    \item There is a piecewise linear convex function $\mu:\Delta \longrightarrow \R$ such that the domains of linearity of $\mu$ are exactly the polytopes $\Delta_i$.
\end{itemize}

If each polytope $\Delta_i$ in the subdivision is a simplex, we say that $T$ is a \textit{convex triangulation}. 
Moreover, if each $\Delta_i$ is a simplex of minimum Euclidean volume $\frac{1}{n!}$, we call $T$ a \textit{convex primitive triangulation}.

All is set to state Viro's Main Patchwork Theorem:
let $\Delta\subset \R^n$ be a convex polytope with integer vertices and let $Q_1, \ldots, Q_s$ be completely nondegenerate real Laurent polynomials in $n$ variables such that $\{\Delta(Q_1), \ldots, \Delta(Q_s)\}$ is a convex subdivision of $\Delta$. Suppose moreover that $Q_i^{\Delta(Q_i) \cap \Delta(Q_j)}=Q_j^{\Delta(Q_i) \cap \Delta(Q_j)}$ for all $i,j \in \{ 1,\ldots, s\}$.
This means that there exists a unique real Laurent polynomial $P(z)=\sum_{\lambda \in \Delta \cap \Z^n } c_\lambda z^\lambda$  such that $P^{\Delta(Q_i)}=Q_i$ for all $i$.
Let $\mu:\Delta \longrightarrow \R$ be a piecewise linear function certifying the convexity of the subdivision, and consider the family of real Laurent polynomials $\{P_t\}_{t\in \R_{>0}}$, where $P_t(z):=\sum_{\lambda \in \Delta \cap \Z^n } c_\lambda t^{\mu(\lambda)} z^\lambda$. Let $(\Delta(Q_i)\times U_K^n,v_i)$ be the chart of $Q_i$, and $(\Delta \times U_K^n, v_t)$ be the chart of $P_t$.
\begin{thm}[Main Patchwork Theorem]\label{TheoremPatchworkMainPatchworkTHM}
The union $ \bigcup_{i=1}^s v_i$ is a submanifold of $\Delta \times U_K^n$, smooth in $\mathring{\Delta}\times U_K^n$ and with boundary in  $\partial \Delta \times U_K^n$ (and corners in the boundary for $n\geq 3 $).

Moreover, for any $t>0$ small enough, $\bigcup_{i=1}^s v_i$ is isotopic to $v_t$ in $\Delta \times U_K^n$ by an isotopy which leaves $\Gamma \times U_K^n$ invariant for each face $\Gamma \subset \Delta$.
\end{thm}

We say that such a polynomial $P_t$ (with $t>0$ small enough) has been obtained by patchworking the polynomials $Q_1, \ldots, Q_s$, or that it is a patchwork of them.

In particular, Theorem \ref{TheoremPatchworkMainPatchworkTHM} allows us to recover the topology of the pairs $((K^*)^n,V_{(K^*)^n}(P_t))$ and $(K_\Delta, V_{K\Delta}(P_t))$ from that of the pairs $(K_{\Delta(Q_i)},V_{K_{\Delta(Q_i)}}(Q_i))$.

Note that though the function $\mu$ plays an important role in the definition of the family of polynomials $\{P_t\}_{t\in \R_{>0}}$, its choice (among the piecewise linear functions inducing the same convex subdivision) does not affect the topology of $(K_\Delta, V_{K\Delta}(P_t))$.

A special case of Viro's method, called \textit{combinatorial patchworking}, can be distinguished.

Let $P$ be a real Laurent polynomial in $n$ variables. If $P$ has exactly $n+1$ monomials with non-zero coefficients and if its Newton polytope $\Delta (P)$ is a non-degenerate $n$-dimensional simplex, we call $P$ a \textit{simplicial} real polynomial. In particular, the monomials of $P$ and the vertices of $\Delta(P)$ are in bijection.

If the real polynomials $Q_1,\ldots,Q_s$  that are being patchworked are all simplicial, which in particular implies that the convex subdivision $\{\Delta(Q_1), \ldots, \Delta(Q_s)\}$ is a triangulation, the construction is called a combinatorial patchworking. If the subdivision $\{\Delta(Q_1), \ldots, \Delta(Q_s)\}$  is a primitive triangulation, we call the construction a \textit{primitive combinatorial patchworking}.

What makes this case special is that the real chart $Chart_{\Delta(P) \times U^n_\R}(P) $ of a real simplicial polynomial $P$ admits a very simple description, which depends only on the signs of the coefficients of $P$.

Viro's method has been generalized by B. Sturmfels in \cite{SturmfelsCompleteIntersectionsCombinatorial} to complete intersections in the combinatorial case, and then to complete intersections in the general case by Bihan in \cite{BihanCompleteIntersections}.

%% file: ConstructionMethod.tex
\section{The construction method}\label{SectionConstructionMethod}

\subsection{Preliminaries}

Let $n\geq 2$. As in the hypotheses of the Cooking Theorem \ref{TheoremConstructionsMainTheoremConstructions}, for $k=1,\ldots,n-1$, consider a family  $\{P^k_d \}_{d\in \NN}$ of completely nondegenerate real Laurent polynomials such that $\Delta(P^k_d) = S^k_d$.

As above, let $S_d^k = \left\{(x_1,\ldots,x_k) \in \R^k \;|\; x_i\geq 0 \; \text{for } i=1,\ldots,k \text{ and }  \sum_{i=1}^k x_i\leq d \right\}$ denote the simplex of side $d$ and dimension $k$.
\comment{Let $H^k_i:=\left\{(x_1,\ldots,x_k) \in \R^k | x_i=0 \right\}$ for $i=1,\ldots,k$ and $H^k_{d,0}:=\left\{(x_1,\ldots,x_k) \in \R^k | \sum_{i=1}^k x_i= d\right\}$.
For any set of indices $I\subset \left\{0,1,\ldots,k\right\}$, define $S^k_{d,I}:=S^k_d \cap (\bigcap _{i\in I} H^k_i)$ if $0\not\in I$ and $S^k_{d,I}:=S^k_d \cap H^k_{d,0} \cap (\bigcap _{i\in I\backslash \left\{0\right\}} H^k_i)$ if $0\in I$.}
For any set of indices $I\subset \left\{0,1,\ldots,k\right\}$, define $$S^k_{d,I}:=S^k_d \cap \left\{(x_1,\ldots,x_k) \in \R^k \;|\; x_i=0 \text{ for all } i\in I\right\}$$ if $0\not\in I$ and $$S^k_{d,I}:=S^k_d \cap \left\{(x_1,\ldots,x_k) \in \R^k\; |\; x_i=0 \text{ for all } i\in I\backslash\{0\}, \; \sum_{i=1}^k x_i= d\right\}$$ if $0\in I$.

We first define in Subsection \ref{SubsectionConstructionsConvexSubdivision} a specific convex triangulation of $S_d^{n-1}$, which we extend to a convex triangulation of $S_d^{n}$.

In Subsection $\ref{SubsectionConstructionsCoefficientsChoice}$, we then define a family $ \{\tilde{Q}^n_d \}_{d\in\NN}$ of real Laurent polynomials such that $\Delta(\tilde{Q}^{n}_d) = S^{n}_d$. Moreover, for any $m=0,\ldots,d-n$ and any $I\subsetneq \left\{0,1,\ldots,n-1\right\}$, the homology of the family of real projective hypersurfaces in $\PP^{n-1-|I|}$ associated to the truncation of the polynomials $\{\tilde{Q}^n_d\}_{d\in\NN}$ to the simplices $S^{n}_{d,I}\cap \{x_n=m \}$ behaves asymptotically in $d$ as the homology of the family of hypersurfaces associated to $\{P^{n-1-|I|}_d \}_{d\in \NN}$ (this will be given more precise meaning later on).

Finally, applying Viro's method to $\{\tilde{Q}^{n}_d\}_{d\in\NN}$ and the convex triangulation we devised, we obtain in Subsection \ref{SubsectionConstructionsDefinitionPnd} a family $\{Q^n_d\}_{d\in\NN}$ of real Laurent polynomials that fulfills the conditions of the Cooking Theorem \ref{TheoremConstructionsMainTheoremConstructions} (which is proved in Section \ref{SectionComputingAsymptoticBettiNumbers}).

\subsection{A convex triangulation of   \texorpdfstring{$S^{n}_d$}{Snd}}\label{SubsectionConstructionsConvexSubdivision}

Given a finite set $\Lambda \subset \Z^k$ and a function $f:\Lambda  \longrightarrow \R$, we define  $\tilde{\Phi}(f):\Conv(\Lambda) \longrightarrow \R$ as the function whose graph is the lower convex hull in $\R^{k+1}$ of the graph of $f$. Note that $\tilde{\Phi}(f)$ always defines a convex subdivision of $\Conv(\Lambda)$.

We define by induction a convex subdivision of $S^{n-1}_d$.
\begin{lemma}\label{LemmaConstructionsConvexSubdivision1}
For any $n\geq 2$ and any $d\geq n$, there exists a piecewise linear convex function $\mu^{n-1}_d:S^{n-1}_d \longrightarrow \R$ with the following properties:
\begin{itemize}
    \item It defines a convex triangulation of $S^{n-1}_d$.
    \item  The simplex \[ \left\{(x_1,\ldots,x_{n-1}) \in \R^{n-1} \text{ s.t. } \sum_{i=1}^{n-1} x_i\leq d-1  \text{ and } x_i\geq 1 \; \text{for }   i=1,    \ldots,n-1   \right\}  \] is one of the (maximal) linearity domains of $\mu^{n-1}_d$ in $S^{n-1}_d $. 
    \item More generally, let $\Gamma$ be any of the faces of dimension $k$ of $S^{n-1}_d $, and let $\Psi :\R^k \longrightarrow \R^{n-1}$ be any affine embedding that maps bijectively the vertices of $S^k_d$ to those of $\Gamma$. Then the pullback to $S^k_d$ by $\Psi$ of the restriction of $\mu^{n-1}_d$ to $\Gamma$ is such that the sub-simplex $\{(x_1,\ldots,x_k) \in \R^{k} | \sum_{i=1}^k x_i\leq d-1, \; x_i\geq 1 \; \text{for }  i=1,\ldots,k, \} \subset S^k_d$ is one of its (maximal) linearity domains.
\end{itemize}
\end{lemma}

\begin{figure}
\begin{center}
\includegraphics[scale=0.3]{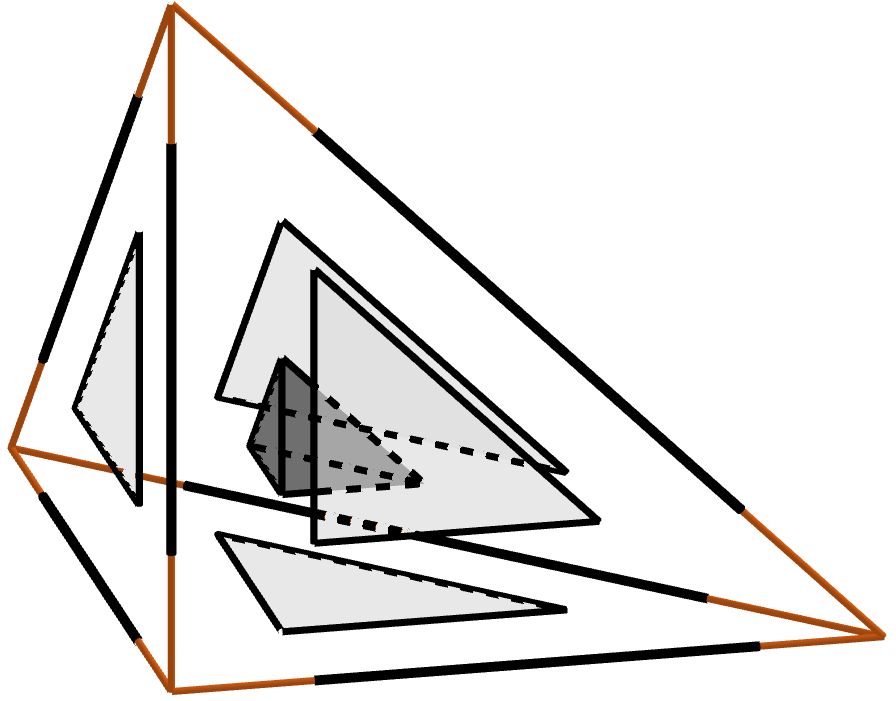}
\end{center}
\caption{For $n-1 =3$, the sub-simplices of dimension $1,2$ and $3$ of $S^3_d$ that have to be maximal linearity domains of $\mu_d^{3}$ for it to satisfy the conditions of Lemma \ref{LemmaConstructionsConvexSubdivision1} are indicated in black and grey.}
\label{FigureConstructionsConditionsLemmaMu}
\end{figure}

\begin{proof}
The sub-simplices of $S^{n-1}_d$ that have to be maximal linearity domains of $\mu^{n-1}_d$ for it to satisfy the conditions of the lemma are illustrated in Figure \ref{FigureConstructionsConditionsLemmaMu} in the case $n-1=3$.

Starting from $k=0$, we will recursively build piecewise linear functions $\mu_k$ with the following properties: $\mu_k$ is defined on the union of the faces of $S^{n-1}_d$ of dimension less than or equal to $k$,  the function $\mu_k$ is strictly positive, the restriction of $\mu_k$ to any face $\Gamma$ of dimension $i\leq k$ is convex and induces a convex triangulation of $\Gamma$, and if $\Psi :\R^i \longrightarrow \R^{n-1}$ is any affine embedding that maps bijectively the vertices of $S^i_d$ to those of $\Gamma$ (such an embedding maps integer points to integer points), then the pullback to $S^i_d$ by $\Psi$ of the restriction of $\mu_k$ to $\Gamma$ is such that the sub-simplex $\{(x_1,\ldots,x_i) \in \R^{i} | \sum_{j=1}^i x_j\leq d-1, \;  x_j\geq 1 \; \text{for }  j=1,\ldots,i \} \subset S^i_d$ is one of its (maximal) linearity domains.

Let $\mu_0$ be constant and equal to $1$ on the vertices of $S^{n-1}_d$.
Suppose that $\mu_{k-1}$ has been defined, and let us define $\mu_k$ (for $k \leq n-1$).

Let $\Gamma$ be any face of dimension $k$ of $S^{n-1}_d$, and choose an affine embedding $\Psi :\R^k \longrightarrow \R^{n-1}$  that maps bijectively the vertices of $S^k_d$ to those of $\Gamma$. The function $\mu_{k-1}$ is defined on the faces of dimension $i\leq k-1$ of $\Gamma$. Through $\Psi$ we identify $S^k_d$ with  $\Gamma$ for the remainder of this paragraph in order to simplify notations (in particular, we see $\mu_{k-1}$ as being defined on the faces of dimension $i\leq k-1$ of $S^k_d$, as in the upper left corner of Figure \ref{FigureConstructionsSubdivisionN2Image2}).
Define $\tilde{\mu}$ as taking generic, strictly positive and very small values on the $k+1$ vertices of $\{(x_1,\ldots,x_k) \in \R^{k} |\sum_{j=1}^k x_j\leq d-1,\;  x_j\geq 1 \; \text{for }  j=1,\ldots,k  \} \subset S^k_d$, and equal to $\mu_{k-1}$ on the faces of dimension $i\leq k-1$ of $S^k_d$. Define $\mu := \tilde{\Phi}(\tilde{\mu}):S^k_d \longrightarrow \R $.
The function $\mu$ coincides with $\mu_{k-1}$ on the faces of dimension $i\leq k-1$ of $S^k_d$ - in particular, it induces the same triangulation of those faces. The convex subdivision that it induces on $S^k_d$ is a triangulation, and for small enough values of $\tilde{\mu}$ on its vertices, $\{(x_1,\ldots,x_k) \in \R^{k} |\sum_{j=1}^k x_j\leq d-1, \;  x_j\geq 1 \; \text{for }  j=1,\ldots,k  \} \subset S^k_d$ is one of its (maximal) linearity domains (see Figure \ref{FigureConstructionsSubdivisionN2Image2}).
We define $\mu_k$ on the face $\Gamma$ as the pushforward of $\mu$ to $\Gamma$ (via the identification $\Gamma \cong S^k_d$ used at the beginning of the paragraph).

\begin{figure}
\begin{center}
\includegraphics[scale=0.3]{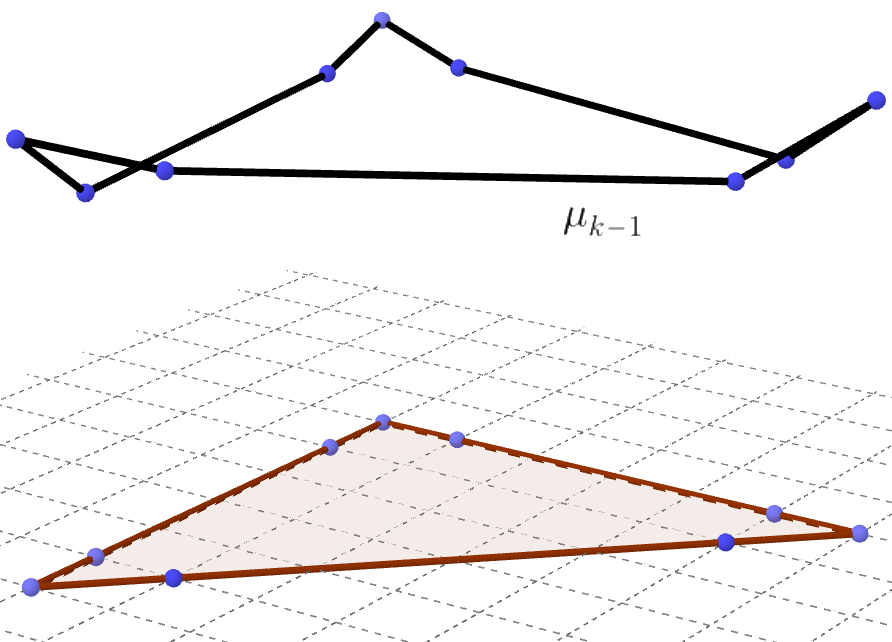}
\includegraphics[scale=0.3]{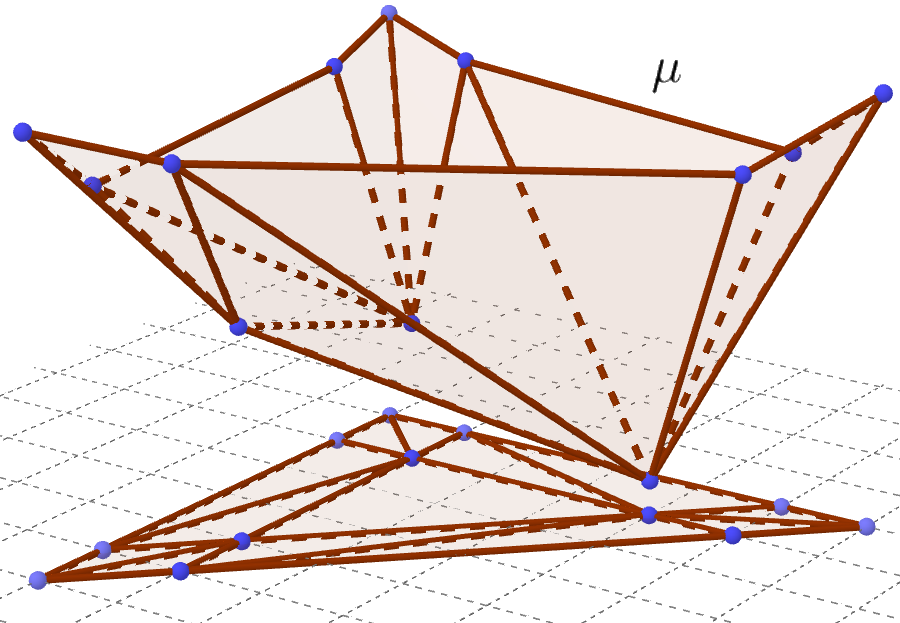}
\includegraphics[scale=0.3]{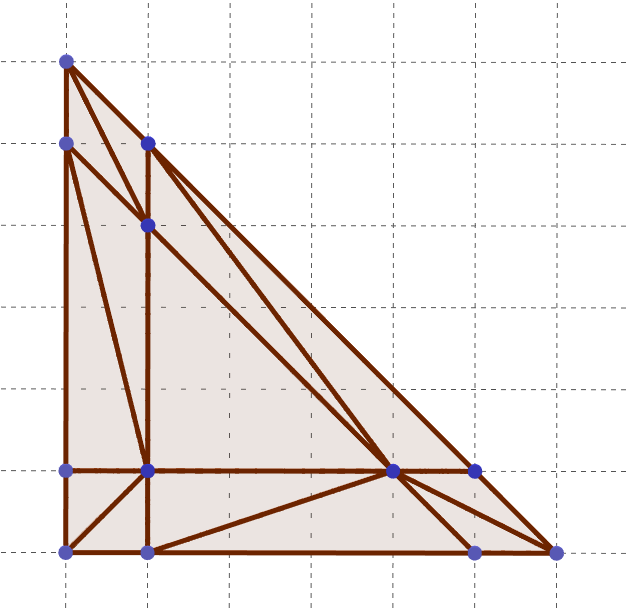}
\end{center}

\caption{
For $k=2$ and $d=6$, the graph of $\mu_{k-1}$ and the graph of $\mu$ on the face $\Gamma$ at the top, and the triangulation induced by $\mu$ on $\Gamma$ below.}
\label{FigureConstructionsSubdivisionN2Image2}
\end{figure}

We proceed similarly on all other faces of dimension $k$ of $S^{n-1}_d$; hence we have defined $\mu_k$.
We let $\mu^{n-1}_d$ be equal to $\mu_{n-1}$.

\end{proof}

We now want to extend the convex triangulation on $S^{n-1}_d$ induced by $\mu_d^{n-1}$ to $S^n_d$ in a clever way, so that it contains some subpolytopes of $S^n_d$ that will play an important role later on.

For $m=0,\ldots,d-1$, define $S^n_{d,m}:=S^n_d \cap \{x_n=m \}$ and $S^n_{d,m+}:=S^n_d \cap \{x_n\in [m,m+1] \}$.

We also need to define the simplices $R^{n}_{d,m,i} \subset S^n_{d,m}$, for $m=1,\ldots,d-n$ and $i=1,\ldots,n-1$. We want $R^{n}_{d,m,i}$ to be an $i$-dimensional simplex contained in the interior of one of the $i$-dimensional faces of the $m$-th ''floor" $S^n_{d,m} $. The face to which it belongs depends  on the parity of $m$.

More precisely, for $m=1,\ldots,d-n$ odd and $i=1,\ldots,n-2$, let $R^{n}_{d,m,i}=\{(x_1,\ldots,x_i,0,\ldots,0,m) \in \R^{n} |  x_j\geq 1 \; \text{for } j=1,\ldots,i, \allowbreak \; \sum_{j=1}^{n-1} x_j\leq d-m-1 \} \subset S^n_{d,m}$.

\sloppy For $m=0,\ldots,d-n$ even and $i=1,\ldots,n-2$, denote by $R^{n}_{d,m,i}=\{(0,\ldots,0,x_{n-1-i},\ldots,x_{n-1},m)  \in \R^{n} | \ x_j\geq 1 \; \text{for } j= n-i-1,\ldots,n-1,$ $\sum_{j=1}^{n-1} x_j= d-m \} \subset S^n_{d,m}$.

For $i=n-1$ and $m=0,\ldots,d-n-1$ (no matter whether $m$ is odd or even), define also $R^{n}_{d,m,n-1}=\{(x_1,\ldots,x_{n-1},m) \in \R^{n} |  x_j\geq 1 \; \text{for } j=1,\ldots,n-1, \; \sum_{j=1}^{n-1} x_j\leq d-m-1 \} \subset S^n_{d,m}$.

This is illustrated in Figure \ref{FigureConstructionsDefR} in the case $n=3$.
In the next lemma, we consider for each $m$ the $n$-dimensional join of $R^{n}_{d,m,i}$ and $R^{n}_{d,m+1,n-1-i}$, as well as the cones of $R^{n}_{d,m,n-1}$ with  either the points $(d-m-1)\cdot e_{n-1} + (m+1)\cdot e_n$ and $(d-m+1)\cdot e_{n-1} + (m-1)\cdot e_n$ or the points $ (m+1)\cdot e_n$ and $ (m-1)\cdot e_n$ (where $e_i$ is the $i$-th vector of the standard basis of $\R^n$), depending on the parity of $m$. The reason why we use the same notation for $R^{n}_{d,m,i}$ regardless of the parity of $m$ is that what we do later on with those joins and cones does not depend on that parity; distinguishing between the two cases would create needless complications.

\begin{figure}
\begin{center}
\includegraphics[scale=0.5]{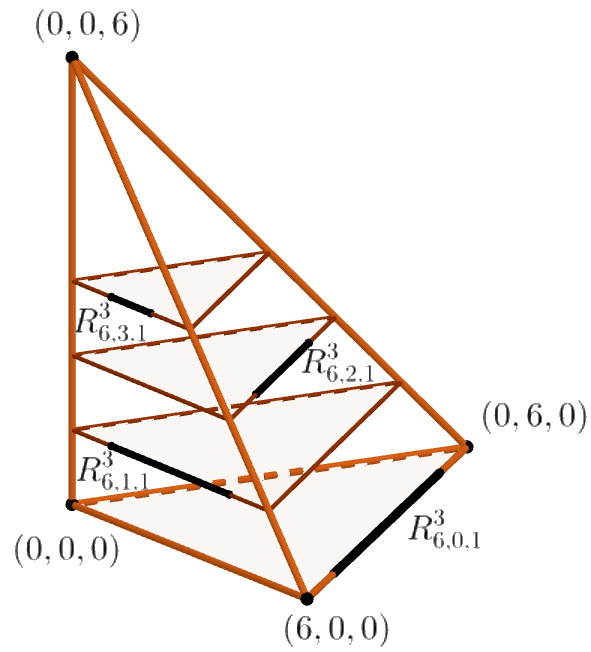}
\includegraphics[scale=0.5]{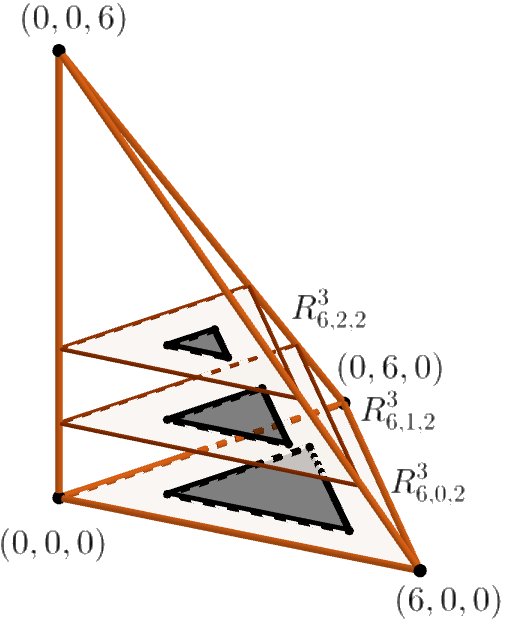}
\end{center}
\caption{In black, the subsets $R^{3}_{6,0,1}$, $R^{3}_{6,1,1}$, $R^{3}_{6,2,1}$  and $R^{3}_{6,3,1}$ of $S^3_6$ on the left and the subsets $R^{3}_{6,0,2}$, $R^{3}_{6,1,2}$ and $R^{3}_{6,2,2}$ on the right.
}
\label{FigureConstructionsDefR}
\end{figure}

\begin{lemma}\label{LemmaConstructionsConvexSubdivision2}
For $n\geq 2$, there exists a triangulation $T$ of $S^n_{d}$ that has the following properties:
\begin{itemize}
    \item For $m=1,\ldots,d-n-1$ odd, the cone of $R^{n}_{d,m,n-1}$ with the vertex $(d-m-1)\cdot e_{n-1} + (m+1)\cdot e_n$ 
    and the cone of $R^{n}_{d,m,n-1}$ with the vertex  $(d-m+1)\cdot e_{n-1} + (m-1)\cdot e_n$  appear in $T$.
    \item For $m=0,\ldots,d-n-1$ even, the cone of $R^{n}_{d,m,n-1}$ with the vertex $(m+1)\cdot e_n$ 
    and the cone of $R^{n}_{d,m,n-1}$ with the vertex $(m-1)\cdot e_n$ (if $m\geq 2$) appear in $T$.
    \item For $m=0,\ldots,d-n-1$ and $i=1,\ldots,n-2$, the join of $R^{n}_{d,m,i}$ with $R^{n}_{d,m+1,n-1-i}$ appears in $T$.
\end{itemize}

\end{lemma}

\begin{figure}
\begin{center}
\includegraphics[scale=0.4]{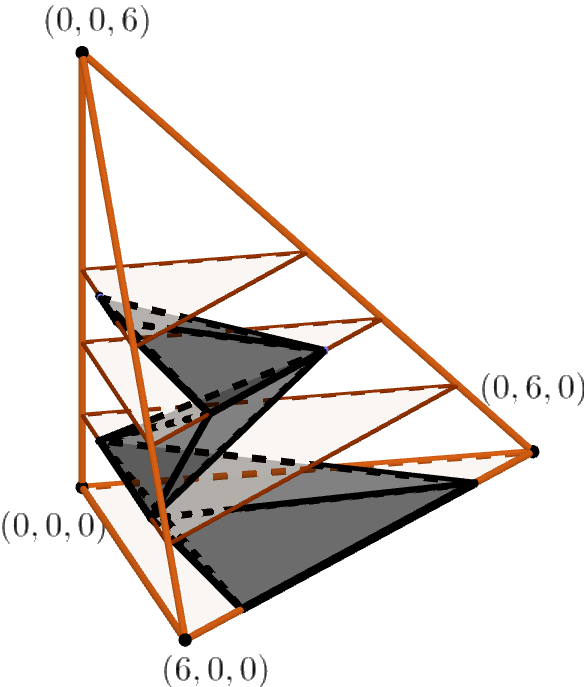}
\includegraphics[scale=0.4]{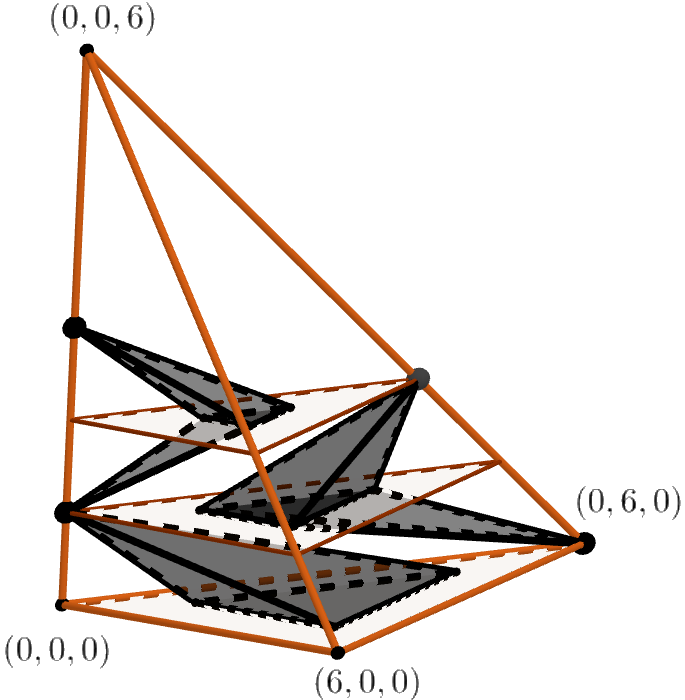}
\end{center}
\caption{The joins (respectively, the cones) that have to appear in the triangulation $T$ for it to satisfy the conditions of Lemma \ref{LemmaConstructionsConvexSubdivision2} are shown in black and grey on the left (respectively, the right).}
\label{FigureConstructionsConditionsJoinsSuspensions}
\end{figure}

\begin{proof}

The cones and joins that have to appear in the triangulation $T$ for it to satisfy the conditions of the lemma are shown in Figure \ref{FigureConstructionsConditionsJoinsSuspensions} in the case $n=3, d=6$.
The idea is simply to use the triangulation defined in Lemma \ref{LemmaConstructionsConvexSubdivision1} to refine a classic decomposition of the slices $S^n_{d,m,+}$ into cones and joins.

If $d<n+1$, choose any convex subdivision on $S^n_d$ (all conditions are automatically satisfied and it matters not, as we are only interested in asymptotic behaviors).

If $d\geq n+1$, for $m=0,\ldots,d-n$, choose functions $\mu^{n-1}_{d-m}$ satisfying the conditions of Lemma \ref{LemmaConstructionsConvexSubdivision1}, and triangulate $S^n_{d,m}$ with the convex subdivision induced by $\mu^{n-1}_{d-m}$ (via the natural identification between $S^n_{d,m}$ and $S^{n-1}_{d-m}$ given by the projection on the first $n-1$ coordinates).

Then for $m=0,\dots,d-n-1$ even, triangulate $S^n_{d,m+}$ thus: subdivide it into the $n$ simplices
\begin{equation*}
\footnotesize
\begin{gathered}
   Conv\left(\left\{(m+1)\cdot e_n \right\} \; \cup \; \left\{m\cdot e_n \right\} \; \cup \; \left\{m\cdot e_n + (d-m)\cdot e_i\right\}_{i = 1,\ldots,n-1} \right),  \\
     Conv\left(\left\{(m+1)\cdot e_n \right\} \; \cup \; \left\{(m+1)\cdot e_n + (d-m-1)\cdot e_1\right\}  \; \cup \; \left\{m\cdot e_n + (d-m)\cdot e_i\right\}_{i = 1,\ldots,n-1} \right),\\
 \ldots,Conv\left(\left\{(m+1)\cdot e_n \right\} \; \cup \; \left\{(m+1)\cdot e_n + (d-m-1)\cdot e_i\right\}_{1\leq i\leq j} \; \cup \; \{m\cdot e_n + (d-m)\cdot e_i\}_{i\geq j}\right)   , \\
\ldots ,   Conv\left(\left\{(m+1)\cdot e_n \right\} \; \cup \; \left\{(m+1)\cdot e_n + (d-m-1)\cdot e_i\right\}_{i=1,\ldots,n-1} \; \cup \; \{m\cdot e_n + (d-m)\cdot e_{n-1}\}\right) 
\end{gathered}
\end{equation*}
(this is a classical way of triangulating the topological product of a simplex and a closed interval).
Each of these simplices $S$ is obtained as the join of a $k$-dimensional face $\Gamma_k$ of $S^n_{d,m+1}$ with a $n-k-1$-dimensional face $\Gamma_{n-k-1}$ of $S^n_{d,m}$, for $k=0,\ldots,n-1$. See the left side of Figure \ref{FigureConstructionsSubdivisionN2Image3}.

\begin{figure}
\begin{center}
\includegraphics[scale=0.24]{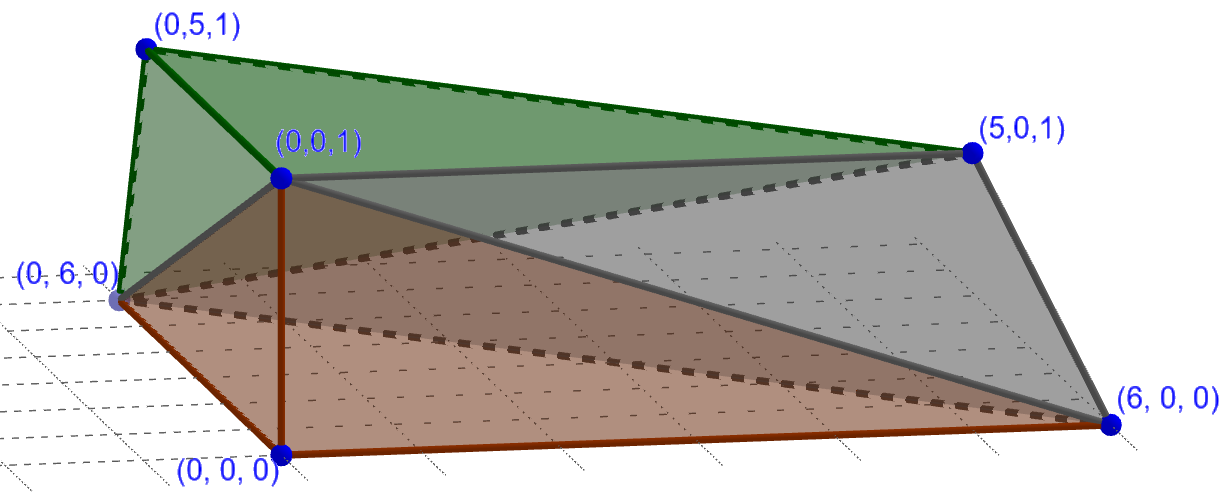}
\includegraphics[scale=0.24]{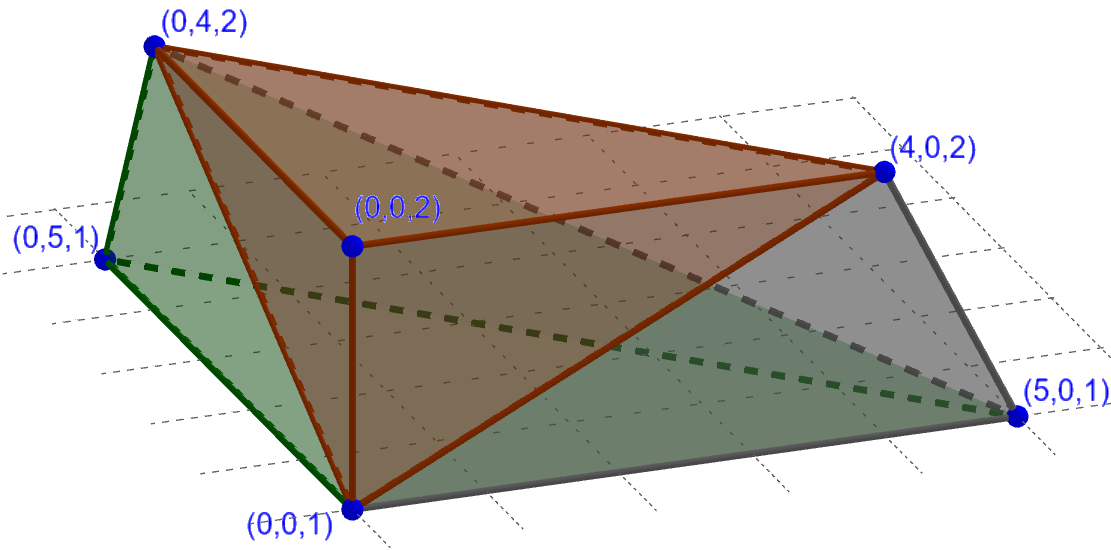}
\end{center}
\caption{
For $n=3$ and $d=6$: on the left, the subdivision $\tilde{T}$ restricted to $S^3_{6,0+}$, on the right, the subdivision $\tilde{T}$ restricted to $S^3_{6,1+}$.
}
\label{FigureConstructionsSubdivisionN2Image3}
\end{figure}

Now for $m=1,\dots,d-n-1$ odd, subdivide $S^n_{d,m+}$ as the union of the $n$ simplices

\begin{equation*}
\footnotesize
\begin{gathered}
   Conv\left(\left\{m\cdot e_n \right\} \; \cup \; \left\{(m+1)\cdot e_n \right\} \; \cup \; \left\{(m+1)\cdot e_n + (d-m-1)\cdot e_i\right\}_{i = 1,\ldots,n-1} \right),  \\
     Conv\left(\left\{m\cdot e_n \right\} \; \cup \; \left\{m\cdot e_n + (d-m)\cdot e_1\right\}  \; \cup \; \left\{(m+1)\cdot e_n + (d-m-1)\cdot e_i\right\}_{i = 1,\ldots,n-1} \right),\\
 \ldots,Conv\left(\left\{m\cdot e_n \right\} \; \cup \; \left\{m\cdot e_n + (d-m)\cdot e_i\right\}_{1\leq i\leq j} \; \cup \; \{(m+1)\cdot e_n + (d-m-1)\cdot e_i\}_{i\geq j}\right)   , \\
\ldots ,   Conv\left(\left\{m\cdot e_n \right\} \; \cup \; \left\{m\cdot e_n + (d-m)\cdot e_i\right\}_{i=1,\ldots,n-1} \; \cup \; \{(m+1)\cdot e_n + (d-m-1)\cdot e_{n-1}\}\right) 
\end{gathered}
\end{equation*}
(the roles of $m$ and $m+1$ have been reversed). See the right side of Figure \ref{FigureConstructionsSubdivisionN2Image3}.
Call the triangulation thus defined $\tilde{T}$ (see Figure \ref{FigureConstructionsGrosseTriangulation}).

We further triangulate each simplex thus obtained as the join of a $k$-dimensional face $\Gamma_k$ of $S^n_{d,m+1}$ with a $n-k-1$-dimensional face $\Gamma_{n-k-1}$ of $S^n_{d,m}$ using the join of the previously defined triangulations of $\Gamma_k$ and $\Gamma_{n-k-1}$.

Choose any triangulation of $S^n_d \cap \{x_n\in [d-n,d] \}$ which extends that on $S^n_{d,d-n}$.

\begin{figure}
\begin{center}
\includegraphics[scale=0.6]{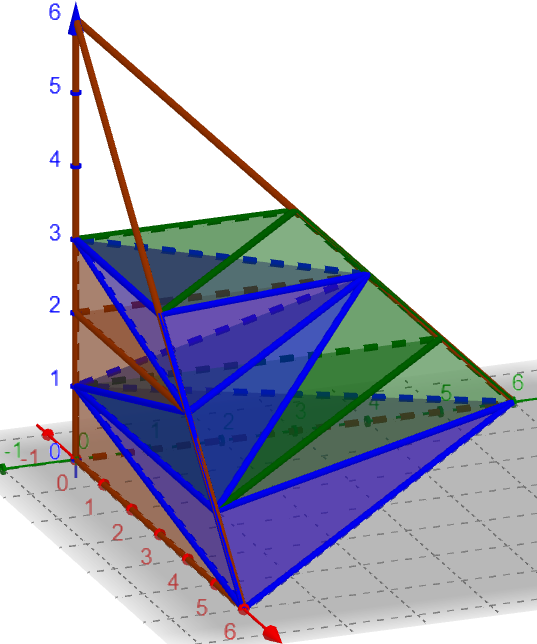}

\end{center}
\caption{
The triangulation $\tilde{T}$ on $S^3_6\cap[0,3]$.}
\label{FigureConstructionsGrosseTriangulation}
\end{figure}

Any triangulation $T$ built this way clearly satisfies the conditions of the Lemma.

\end{proof}

It remains to show that such triangulations can be required to be convex, which is the case.

\begin{lemma}\label{LemmaConstructionsConvexSubdivision3}
For any $n\geq 2$ and any $d\geq n$, there exists a piecewise linear convex function $\tilde{\mu}^{n}_d:S^{n}_d \longrightarrow \R$ such that it gives rises to a convex triangulation $T$ of $S^n_d$ which satisfies the conditions of Lemma $\ref{LemmaConstructionsConvexSubdivision2}$.

\end{lemma}

\begin{proof}

We first consider a convex subdivision of  $S^n_d \cap \{x_n\in [0,d-n] \}$ such that the domains of linearity of the associated convex function are exactly the slices $S^n_{d,m+}$ for $m=0,\ldots,d-n-1$ (a function identically equal to $m^2$ on $S^n_{d,m}$ does the trick).

Now consider on each $S^n_{d,m+}$ (for $m=0,\ldots,d-n-1$) the convex triangulation into $n$ simplices described in the proof of Lemma $\ref{LemmaConstructionsConvexSubdivision2}$. If $m$ is even, let $f$ be defined on the $2n$ vertices of  $S^n_{d,m+}$ as being identically equal to $0$ on $S^n_{d,m}\cup \{(m+1)\cdot e_n\}$, and equal to $i$ on $(m+1)\cdot e_n + (d-m-1)\cdot e_i$ (for $i=1,\ldots,n-1$). Then $\tilde{\Phi}(f):S^n_{d,m+} \longrightarrow \R$ gives the desired triangulation (and similarly for $m$ odd).

Using that and applying repeatedly the technical Lemma \ref{LemmaConstructionsTechnicalLemmaSubdivision} (in the notations of the Lemma, $\Delta$ is $S^n_d \cap \{x_n\in [0,d-n] \}$ and $\Gamma$ is $S^n_{d,m+}$, for successive $m=0,\ldots,d-n-1$), we obtain a convex triangulation $\tilde{T}$ of $S^n_d \cap \{x_n\in [0,d-n] \}$ as built in the proof of Lemma $\ref{LemmaConstructionsConvexSubdivision2}$.

For $m=0,\ldots,d-n$, let $\mu^{n-1}_{d-m}:S^n_{d,m} \longrightarrow \R$ be a function satisfying the conditions of Lemma \ref{LemmaConstructionsConvexSubdivision1} (where we have identified $S^n_{d,m}$ with $S^n_{d-m}$ via the projection on the first $n-1$ coordinates).

We need to further subdivide the simplices obtained as joins and cones (see the proof of Lemma \ref{LemmaConstructionsConvexSubdivision2}) along the triangulations induced by the functions $\mu^{n-1}_{d-m}$.

Once again, we apply repeatedly Lemma \ref{LemmaConstructionsTechnicalLemmaSubdivision}, this time to  $S^n_d \cap \{x_n\in [0,d-n] \}$ and $S^n_{d,m}$ (as $\Delta$ and $\Gamma$, respectively, in the notations of the Lemma) for $m=0,\ldots,d-n$.

We get a convex triangulation of $S^n_d \cap \{x_n\in [0,d-n] \}$  which is a refinement of $\tilde{T}$ and which coincides with the triangulation induced by the functions $\mu^{n-1}_{d-m}$ on $S^n_{d,m} $, and a piecewise linear convex function  $\mu: S^n_d \cap \{x_n\in [0,d-n] \} \longrightarrow \R$ certifying the convexity of this triangulation. Define $\tilde{\mu}^n_d :\{d\cdot e_n\} \cup (S^n_d \cap \{x_n\in [0,d-n] \}) \longrightarrow \R$ as being equal to $\mu$ on $S^n_d \cap \{x_n\in [0,d-n] \}$ and equal to some large enough $R\in \R$ on $\{d\cdot e_n\}$. Then $\tilde{\Phi}(\tilde{\mu}^n_d)$ extends the triangulation induced by $\mu$ to a convex triangulation $T$ on $S^n_d$, which is as wanted.

\end{proof}

We now state the technical result used in the proof of Lemma \ref{LemmaConstructionsConvexSubdivision3}, a useful statement that allows us to refine a convex subdivision using another convex subdivision of one of its faces while preserving its convexity.

\begin{lemma}[Technical Lemma]\label{LemmaConstructionsTechnicalLemmaSubdivision}
Let $\Delta \subset \R^n $ be a convex (bounded) polytope with integer vertices, and let $\mu:\Delta \longrightarrow \R$ be such that $\mu = \tilde{\Phi}(\mu|_{ \Delta \cap \Z^n })$. Let $\Gamma \subset \Delta$ be a (not necessarily top-dimensional) face of the convex subdivision induced by $\mu$ on $\Delta$. Let $\nu:\Gamma \longrightarrow \R$ be such that $\nu = \tilde{\Phi}(\nu|_{ \Gamma\cap \Z^n })$. Then there exists a function $\xi:\Delta \longrightarrow \R$ so that:
\begin{enumerate}
    \item $\xi = \tilde{\Phi}(\xi|_{ \Delta \cap \Z^n })$ (hence $\xi$ is piecewise linear convex and gives rise to a convex subdivision of $\Delta$).
        \item $\xi |_{(\Delta \backslash \Gamma)\cap \Z^n}=\mu |_{(\Delta \backslash \Gamma)\cap \Z^n}$. 
    \item The convex subdivision induced by $\xi$ on $\Delta$ is a refinement of the one induced by $\mu$.
    \item The convex subdivision induced by $\xi$ on $\Gamma$ is the same as the one induced by $\nu$.
\end{enumerate}
\end{lemma}

\begin{proof}
The proof is both straightforward and tedious. We omit it.
\end{proof}

\subsection{Choosing the coefficients of \texorpdfstring{$\tilde{Q}^{n}_d$}{tildeQnd} }\label{SubsectionConstructionsCoefficientsChoice}

For any Laurent polynomial $P$ in $k$ variables, we write $P (z)=\sum _{\lambda \in \Delta(P)\cap\Z^k} c_P(\lambda)z^\lambda$ (where some coefficients $c_{P}(\lambda)\in \R$ can be $0$).
We use the notations of Section \ref{SectionPatchwork}.
In particular, given two real Laurent polynomials in $k$ variables $P_1$ and $P_2$, we say that their charts are homeomorphic if there is a homeomorphism of stratified topological spaces between the pairs $(\Delta(P_1)\times U^k_\R,Chart_{\Delta(P_1)\times U^k_\R}(P_1))$ and $(\Delta(P_2)\times U^k_\R,Chart_{\Delta(P_2)\times U^k_\R}(P_2))$ - remember that the chart of $P_i$ is actually defined as the pair $(\Delta(P_i)\times U^k_\R,Chart_{\Delta(P_i)\times U^k_\R}(P_i))$.

\begin{lemma}\label{LemmaConstructionsCoefficients}
Let $n\geq 2$ and $d\geq n$. Given a convex triangulation $T$ of $S^n_d $ that satisfies the conditions of Lemma \ref{LemmaConstructionsConvexSubdivision2} and completely nondegenerate real Laurent polynomials $\{P^k_i\}_{i=0}^{d-1-k} $ such that $\Delta(P^k_i) = S^k_i$ (for $k=1,\ldots,n-1$), there exists a real Laurent polynomial $\tilde{Q}^{n}_d$ such that:
\begin{enumerate}
    \item $\Delta(\tilde{Q}^{n}_d) = S^n_d $.
    \item  For each simplex $S$ of the triangulation $T$, the truncation $\tilde{Q}^{n}_d|^S$ is completely nondegenerate.
    \item For $m=0,\ldots,d-n$ and $k=1,\ldots,n-1$ and for any lattice-respecting identification $R^{n}_{d,m,k} \cong S^k_{d-m-1-k}$ which lets us see $\tilde{Q}^{n}_d|^{R^{n}_{d,m,k}}$ as a polynomial $G$ in $k$ variables, the charts of $G$ and $P^k_{d-m-1-k}$ are homeomorphic.
    \item For $m=0,\ldots,d-n$ even,  the only coefficient of the truncation $\tilde{Q}^{n}_d|^{\{m\cdot e_n + (d-m)\cdot e_{n-1}  \}}$ is strictly positive.
    \item  For $m=1,\ldots,d-n$ odd, the only coefficient of the truncation $\tilde{Q}^{n}_d|^{ \{m\cdot e_n\}}$ is strictly positive.
\end{enumerate}
\end{lemma}
\begin{remark}
The key idea here is that we want $\tilde{Q}^n_d$ to behave ``similarly to" (in the sense of condition (3) of the lemma) the polynomials $\{P^k_{i}\}_{i=0}^{d-1-k}$ when truncated with respect to well-chosen sub-simplices of dimension $k$ of its Newton polytope. As we defined the sets $R^n_{d,m,k}$ as simplices strictly contained in the relative interior of the face of $S^n_{d,m}$ of minimal dimension that contains them, they are disjoint from each other, and we can thus ask for this condition to hold for all $k$ simultaneously, despite there being no constraint of compatibility on the polynomials $\{P^k_{i}\}_{k,i}$.
\end{remark}

\begin{proof}
We define a function $\tilde{c}:S^n_d \cap \Z^n \longrightarrow \R$.

Set $\tilde{c}(m\cdot e_n + (d-m)\cdot e_{n-1} )=1$ for $m=0,\ldots,d-n$ even and $\tilde{c}(m\cdot e_n)=1$ for $m=1,\ldots,d-n$ odd.

For $m=0,\ldots,d-n$ and $k=1,\ldots,n-1$, choose a lattice respecting identification $R^{n}_{d,m,k} \cong S^k_{d-m-1-k}$, and via this identification, set $\tilde{c}(\lambda)=c_{P^k_{d-m-1-k}}(\lambda)$ for any $\lambda\in R^{n}_{d,m,k} \cap \Z^n $.

For any other $\lambda \in S^n_d$, pick an arbitrary non-zero value for $\tilde{c}(\lambda)$.

Now all conditions, except a priori Condition 2, are satisfied by the polynomial $P (z):=\sum _{\lambda \in S^n_d\cap\Z^n} \tilde{c}(\lambda)z^\lambda$.
As observed in \cite{ViroPatchworking}, among all polynomials with a given Newton polytope, nondegenerate polynomials form an (Zariski) open set.
Moreover, as each $P^k_i$ is nondegenerate, the hypersurface $V_{K\Delta(P^k_i)}(P^k_i)$ is smooth, and a small perturbation of the coefficients of $P^k_i$ will not change the topology of its chart.

With those two observations in mind, we can define $c$ as a small generic perturbation of $\tilde{c}$ such that all conditions are fulfilled by $\sum _{\lambda \in S^n_d\cap\Z^n} c(\lambda)z^\lambda$, and set $\tilde{Q}^n_d(z):=\sum _{\lambda \in S^n_d\cap\Z^n} c(\lambda)z^\lambda$.

\end{proof}

\subsection{Defining \texorpdfstring{$Q^{n}_d$}{Qnd} using the Patchwork}\label{SubsectionConstructionsDefinitionPnd}

Making use of the results of the two previous subsections, we get the following proposition.

\begin{proposition}[Construction Proposition]\label{PropositionConstructionsMethod}
Let $n\geq 2$. For $k=1,\ldots,n-1$, consider a family  $\{P^k_d \}_{d\in \NN}$ of completely nondegenerate real Laurent polynomials such that $\Delta(P^k_i) = S^k_i$ .

Then for all $d\geq n$, there exists a completely nondegenerate real Laurent polynomial $Q^n_d$, with $\Delta(Q^n_d)=S^n_d$, such that:

\begin{enumerate}
    \item $Q^n_d$ is obtained via a patchworking of a family $\Sigma$ of completely nondegenerate real Laurent polynomials.
    
     \item For $m=0,\ldots,d-n-1$ and $k=1,\ldots,n-2$, there are polynomials  $F^k_m \in \Sigma$ such that the chart of $F^k_m$ is homeomorphic to the chart of $G$, for some polynomial
    \begin{equation*}
        G: (x_1,\ldots,x_{n-1},x_n) \mapsto  \tilde{P}^k_{d-m-1-k}(x_1,\ldots,x_k)+ x_n\tilde{P}^{n-1-k}_{d-m-n+k}(x_{k+1},\ldots,x_{n-1}),  
    \end{equation*}
    where each $\tilde{P}^k_{i}$ is itself such that $\Delta(\tilde{P}^k_{i})$ is a translate of $S^k_i$ and that its chart is homeomorphic to the chart of $P^k_{i}$.

    
    \item For $m=1,\ldots,d-n-1$, there are polynomials $G^+_m, G^-_m \in \Sigma$ such that the $n$-dimensional simplices $\Delta(G^+_m)$ and $\Delta(G^-_m)$ have a ($n-1$)-dimensional face in common, and such that the gluing of their charts
    $$((\Delta(G^+_m)\cup\Delta(G^-_m))\times U^n_\R,Chart_{\Delta(G^+_m)\times U^n_\R}(G^+_m)\cup Chart_{\Delta(G^-_m)\times U^n_\R}(G^-_m))$$
    is homeomorphic as a (stratified) pair to the gluing of charts
    $$((\Delta(\tilde{G}^+_m)\cup\Delta(\tilde{G}^-_m))\times U^n_\R,Chart_{\Delta(\tilde{G}^+_m)\times U^n_\R}(\tilde{G}^+_m)\cup Chart_{\Delta(\tilde{G}^-_m)\times U^n_\R}(\tilde{G}^-_m)),$$
    where
    $$\tilde{G}^+_m: (x_1,\ldots,x_{n-1},x_n) \mapsto \tilde{P}^{n-1}_{d-m-n}(x_1,\ldots, x_{n-1})+\gamma^+_m \cdot x_n,$$
    $$\tilde{G}^-_m: (x_1,\ldots,x_{n-1},x_n) \mapsto \tilde{P}^{n-1}_{d-m-n}(x_1,\ldots, x_{n-1})+\gamma^-_m \cdot x_n^{-1},$$
    $\gamma^+_m$ and $\gamma^- _m$ are some strictly positive constant, and each $\tilde{P}^{n-1}_{i}$ is itself such that $\Delta(\tilde{P}^{n-1}_{i})$ is a translate of $S^{n-1}_i$ and that its chart is homeomorphic to the chart of $P^{n-1}_i$.
    
    

    \item The interiors of the simplices $\Delta(F^k_m)$, $\Delta(G^+_l)$ and $\Delta(G^-_p)$ are disjoint for all $k$, $m$, $l$ and $p$.
    
\end{enumerate}

Additionally,  if each $P^k_d$ was obtained by combinatorial patchworking, there exists such a polynomial $Q^n_d$ that can also be obtained by combinatorial patchworking. 
\end{proposition}

\begin{remark}
Intuitively, the chart of the polynomials $F^k_m$ should be thought of as a join of sorts of the charts of $P^k_{d-m-1-k}$ and $P^{n-1-k}_{d-m-n+k}$, while the union of the charts of $G^+_m$ and $G^-_m$ corresponds to a suspension of sorts of the chart of $P^{n-1}_{d-m-n}$. The interior of their Newton polytopes being disjoint means that they contribute mostly independently to the homology of the hypersurface defined by $Q^n_d$. This will be made rigorous in the next section.
\end{remark}

\begin{proof}

Let $d\geq n$.
By Lemma \ref{LemmaConstructionsConvexSubdivision3}, there exists a convex triangulation $T$ that satisfies the conditions of Lemma \ref{LemmaConstructionsConvexSubdivision2} and a convex function $\tilde{\mu}^n_d:S^n_d\longrightarrow \R$ which certifies its convexity.
The Newton polytopes of the polynomials $F^k_m$ (respectively $G^+_m$ and $G^-_m$) are going to be the joins (respectively the cones) that have to appear in $T$ for it to fulfill the conditions of Lemma \ref{LemmaConstructionsConvexSubdivision2} and that were illustrated in Figure \ref{FigureConstructionsConditionsJoinsSuspensions} - the Newton polytopes of the polynomials $F^k_m$ can be seen on the left of the figure, and the Newton polytopes of the polynomials $G^+_m$ and $G^-_m$ can be seen on the right.

The triangulation $T$ and the polynomials $\{P^k_i\}_{i=0}^{d-1-k}$ satisfy the hypotheses of Lemma \ref{LemmaConstructionsCoefficients}, which yields a polynomial $\tilde{Q} ^n_d$ satisfying its conditions. We can apply Viro's Patchwork Theorem \ref{TheoremPatchworkMainPatchworkTHM} to $T$, $\tilde{\mu}^n_d$ and $\tilde{Q} ^n_d$ (which plays the role of $P$ in the notations of Section \ref{SectionPatchwork}) to get a family of polynomials $\{P_t\}_{t\in \R_{>0}}$, and let $Q^n_d$ be any $P_t$ with $t$ small enough for the conclusions of Patchwork Theorem \ref{TheoremPatchworkMainPatchworkTHM} to apply. Let us show that $Q^n_d$ satisfies all required conditions.

Condition 1 is trivially satisfied. 

For $m=0,\ldots,d-n-1$ and $k=1,\ldots,n-2$, polynomial  $F^k_m$ is defined as the truncation $\tilde{Q} ^n_d|^{R^n_{d,m,k}\star  R^n_{d,m+1,n-1-k}}$, where $\star$ denotes the join. Observe that  $F^k_m= \tilde{Q} ^n_d|^{R^n_{d,m,k}}+ \tilde{Q} ^n_d|^{  R^n_{d,m+1,n-1-k}}$.

A suitable affine isomorphism $\Z^n \longrightarrow \Z^n$, extended to $\R^n$, will map $R^n_{d,m,k}$ to
$$\left\{(x_1,\ldots,x_k,0,\ldots,0,m) \in \R^{n}\; |\;  x_j\geq 1 \; \text{for } j=1,\ldots,k, \; \sum_{j=1}^{n-1} x_j\leq d-m-1 \right\} \subset S^n_{d,m}$$ and $R^n_{d,m+1,n-1-k}$ to
\begin{align*}
\Bigg\{(0,\ldots,0,x_{k+1},\ldots,x_{n-1},m+1) \in \R^{n} \;|  \;x_j\geq 1 \; \text{for } j=k+1,\ldots,n-1,  \\ \; \sum_{j=1}^{n-1} x_j\leq d-m-2 \Bigg\} \subset S^n_{d,m+1}. 
\end{align*}
This linear transformation induces an isomorphic change of coordinates $(K^*)^n \longrightarrow (K^*)^n$. In particular, that change of coordinates maps $\tilde{Q} ^n_d|^{R^n_{d,m,k}}$  (respectively, $\tilde{Q} ^n_d|^{R^n_{d,m+1,n-1-k}}$) to $x_n^m\cdot \tilde{P}^k_{d-m-1-k}$, where $\tilde{P}^k_{d-m-1-k}$ is a polynomial in $k$ variables  (respectively, $x_n^{m+1}\cdot\tilde{P}^{n-1-k}_{d-m-n+k}$ with $\tilde{P}^{n-1-k}_{d-m-n+k}$ a polynomial in $n-1-k$ variables), and  $\tilde{Q} ^n_d|^{R^n_{d,m,k}\star  R^n_{d,m+1,n-1-k}}$ to $x_n^m\cdot\tilde{P}^k_{d-m-1-k}+x_n^{m+1}\cdot\tilde{P}^{n-1-k}_{d-m-n+k}$.

Now $\tilde{P}^k_{d-m-1-k}$ has been obtained from $\tilde{Q} ^n_d|^{R^n_{d,m,k}}$ via an isomorphic change of coordinates, and $\tilde{Q} ^n_d|^{R^n_{d,m,k}}$ itself was obtained from $P^k_{d-m-1-k} $ via an isomorphic change of coordinates (since $\tilde{Q} ^n_d$ satisfies to Condition 3 of Lemma \ref{LemmaConstructionsCoefficients}) and a small generic perturbation, so that the topology of the associated (via the change of coordinates) hypersurface  would not change. Hence the chart of $\tilde{P}^k_{d-m-1-k}$ is homeomorphic to the chart of $P^k_{d-m-1-k}$.
The same applies to $\tilde{P}^{n-1-k}_{d-m-n+k}$.

Finally, there is a trivial homeomorphism of pairs between the toric variety and hypersurface induced by $F^k_m$ and those induced by $\tilde{P}^k_{d-m-1-k}+x_n\tilde{P}^{n-1-k}_{d-m-n+k}$, hence between the corresponding charts and ambient spaces as well. This proves Condition 2.

For $m=2,\ldots,d-n-1$ even, polynomial  $G_m^+$ (respectively, $G_m^-$) is defined as the truncation $\tilde{Q} ^n_d|^{ (m+1)\cdot e_n \star  R^n_{d,m,n-1}}$ (respectively, the truncation $\tilde{Q} ^n_d|^{(m-1)\cdot e_n \star  R^n_{d,m,n-1}}$).

For $m=1,\ldots,d-n-1$ odd, polynomial  $G_m^+$ (respectively, $G_m^-$) is defined as the truncation $\tilde{Q} ^n_d|^{\left((m+1)\cdot e_n + (d-m-1)\cdot e_{n-1}\right) \star  R^n_{d,m,n-1}}$ (respectively, the truncation $\tilde{Q} ^n_d|^{\left((m-1)\cdot e_n + (d-m+1)\cdot e_{n-1}\right)  \star  R^n_{d,m,n-1}}$).

The same type of arguments as above yield Condition 3, and condition 4 is an evident consequence of the definitions of the polynomials $F^k_m$, $G^+_m$ and $G^-_m$.

If each $P^k_d$ was obtained by combinatorial patchworking, it is easy to show, using repeatedly Lemma \ref{LemmaConstructionsTechnicalLemmaSubdivision}, that the triangulation $T$ can be refined into a convex triangulation $T'$ such that its restriction to each $R^n_{d,m,k}$ corresponds to the triangulation used to define the corresponding polynomial $P^k_{d-m-1-k}$  (via the proper identifications).
Likewise, the proof of Lemma \ref{LemmaConstructionsCoefficients} only has to be adapted in that the coefficients of $\tilde{Q}^n_d$ have to be chosen so that the truncation $\tilde{Q}^n_d|^{S}$ is completely nondegenerate for each simplex of the refined triangulation $T'$, which is once again a condition generically satisfied. 

Then the Patchwork can be applied to $T'$ and $\tilde{Q}^n_d$, and the same conclusions as above stand for the resulting polynomial $Q^n_d$.
\end{proof}

%% file: ComputingAsymptoticBetti.tex
\section{Computing asymptotic Betti numbers}\label{SectionComputingAsymptoticBettiNumbers}
In this section, we are proving the Cooking Theorem \ref{TheoremConstructionsMainTheoremConstructions}; more specifically, that families of real Laurent polynomials obtained in the Construction Proposition \ref{PropositionConstructionsMethod}  using the ingredients in the statement of the Cooking Theorem do satisfy Formula \ref{FormulaConstructionsMainFormula}.
Briefly and informally summarized, we start by showing in Subsection \ref{SubsectionConstructionsPreliminaries} that  ``most" of the cycles of a real algebraic hypersurface can be represented by chains that verify useful conditions that make them easy to manipulate. In Subsections \ref{SubsectionConstructionsCyclesSuspension} and \ref{SubsectionConstructionsCyclesSuspension}, we show how such easy-to-manipulate chains in the hypersurfaces used as ingredients for our construction give rise to many new cycles in some special areas of the resulting hypersurfaces. Finally, in Subsection \ref{SubsectionConstructionsCountingCycles}, we prove by using linking numbers that enough of these cycles are independent in the homology of the resulting hypersurfaces, thence completing the proof of the Cooking Theorem.

\subsection{Preliminaries}\label{SubsectionConstructionsPreliminaries}

We first state a useful simplifying fact.


\begin{lemma}\label{LemmaConstructionsEquivalenceOuvertFerme}
For every $k\geq 1$, there is a constant $C(k)>0$ such that for all completely nondegenerate polynomials $P$ in $k$ variables and degree $d\geq 1$ such that $\Delta(P) =S^k_d$, we have the following inequalities:

\begin{align}\label{FormulaConstructionsEquivalenceOuverFerme}
    |b_i(V_{\R\PP^k}(P)) - b_i(V_{(\R*)^k}(P))|\leq C(k)d^{k-1},\\
        |b_i(V_{\R^k}(P)) - b_i(V_{(\R*)^k}(P))|\leq C(k)d^{k-1}.\nonumber
\end{align}

\end{lemma}

\begin{proof}
The proof is a straightforward combination of the Smith-Thom inequality \ref{SmithThom}, Equation (\ref{ComplexPartTotalHomology}) (which describes the dimension of the total homology of the complex part of a smooth real projective algebraic hypersurface) and repeated applications of the Mayer-Vietoris sequence.
\end{proof}

\begin{remark}\label{RemarkConstructionsEquivalenceOuvertFerme}
As explained in Section \ref{SectionPatchwork}, the projective space $\R \PP^k$ can be seen, via an isomorphism, as an appropriate quotient of $S^k_d \times U^k_\R$; the same is also  true of $\R^k$.
This immediately implies corresponding statements regarding the homology of the charts of $P$ in $S^k_d \times U^k_\R$ and its relevant quotients. We also know that the pairs $(\mathring{S^k_d} \times U^k_\R ,Chart_{\mathring{S^k_d} \times U^k_\R }(P))$ and $(S^k_d \times U^k_\R ,Chart_{S^k_d \times U^k_\R }(P))$ are homotopy equivalent.
\end{remark}

As we are only interested in the asymptotical behavior in $d$ (in the sense described in Section \ref{SectionIntroduction}) of the Betti numbers, Lemma \ref{LemmaConstructionsEquivalenceOuvertFerme} means that we can often ignore the distinction between the homology of a given hypersurface in $(\R^*)^k$ and that of the corresponding hypersurface in $\R \PP^k$.

Let $k,i\geq 0$, and $X$ be a submanifold of $\R^k$. Given homology classes $\alpha \in \tilde{H}_i(X)$ and $\beta \in \tilde{H}_{k-1-i}(\R^k \backslash X)$ (where $\tilde{H}$ indicates the reduced homology), the linking number $l(\alpha, \beta)$ is well defined as the transversal intersection (in $\R^k$) number (in $\Z_2$) of any cycle $a \in \alpha$ and any $k-i$ chain $m$ in $\R^k$, called a \textit{membrane}, such that $\partial m =b$, where $b$ is a cycle in $\beta$ (a membrane can always be found, since any cycle is a boundary in the trivial reduced homology of $\R^k$).
It can be shown that it is equivalent to consider the transversal intersection number of $b$  with a membrane of $a$.
We can adapt this operation to the non-reduced homology by taking the exact same definition when $k-1-i\neq 0$, and restricting the linking number to  $H_{k-1}(X)\times \ker (H_0 (\R^k \backslash X) \longrightarrow H_0(\R^k))$ when $k-1-i=0$, as any cycle in $\R^k \backslash X$ whose class belongs to $\ker (H_0 (\R^k \backslash X) \longrightarrow H_0(\R^k)) $ admits a membrane in $\R^k$. In fact, $\ker (H_0 (\R^k \backslash X) \longrightarrow H_0(\R^k)) $ and $\tilde{H}_0(\R^k \backslash X)$ are naturally isomorphic.
See \cite{FomenkoFuchs} for more details on linking numbers.

This definition can be easily generalized, in our particular case to pairs $(Y,X)$ where $X\subset Y$ and $Y$ is a disjoint union of convex subsets of $\R^k$.
Given such a pair $(Y,X)$ and a collection of homology classes $\alpha_1,\ldots,\alpha_r$ in $H_i(X)$, we say that classes $\beta_1,\ldots,\beta_r$ in $ H_{k-1-i}(Y\backslash X)$ (respectively, in $\ker (H_0 (Y \backslash X) \longrightarrow H_0(Y))$ if $k-1-i=0$) are \textit{axes} for the collection $\alpha_1,\ldots,\alpha_r$ if for any $i,j$ we have $l(\alpha_i,\beta_j)=\delta_{i,j}\in \Z_2$.
As the linking number is a $\Z_2$-bilinear product, this implies, in particular, that the classes $\alpha_1,\ldots,\alpha_r$ are linearly independent.

We want to prove lower bounds on the Betti numbers of the hypersurfaces constructed using the Construction Proposition \ref{PropositionConstructionsMethod} by finding enough cycles in the hypersurfaces and by showing that their homology classes are independent using suitable axes, in the spirit of \cite{IV}. To do so, we rely on the following Lemma to obtain cycles and axes in sufficient numbers and such that they can each be represented by chains contained in a connected component of $\mathring{S^k_d} \times U_\R^k$ (which makes them easy to manipulate).

\begin{lemma}\label{LemmaConstructionCyclesEtAxesBienInclus}
For all $k\geq 1$, there is a constant $D(k)>0$ with the following property:

Let $P$ be a completely nondegenerate polynomial of degree $d\geq 1$ in $k$ variables such that $\Delta(P) =S^k_d$, and let $i \in \{0,1,\ldots,k-1\}$.
Then there exists
$$r\geq b_i( V_{(\R^*)^k}(P)) - D(k) d^{k-1}$$
\sloppy such that we can find classes $\alpha_1,\ldots,\alpha_r$ in $ H_i(Chart_{\mathring{S^k_d} \times U_\R^k}(P))$ and $\beta_1,\ldots,\beta_r $ in $ H_{k-1-i}((\mathring{S^k_d} \times U_\R^k) \backslash Chart_{\mathring{S^k_d} \times U_\R^k}(P))$
(respectively, in $\ker (H_0 ((\mathring{S^k_d} \times U_\R^k) \backslash Chart_{\mathring{S^k_d} \times U_\R^k}(P)) \allowbreak \longrightarrow H_0(\mathring{S^k_d} \times U_\R^k))$ if $k-1-i=0$) whose linking numbers in $\mathring{S^k_d} \times U_\R^k$ satisfy $l(\alpha_s,\beta_t)=\delta_{s,t}\in \Z_2$ (the classes $\beta_t$ are axes for the classes $\alpha_s$).


Moreover, we can ask that there be cycles $b_1\in\beta_1, \ldots, b_r\in \beta_r$ and $a_1\in \alpha_1,\ldots,a_r\in \alpha_r$ such that the sign of $P$ is constant on each $b_j$ (when evaluated via the identification $\mathring{S^k_d} \times U_\R^k \cong (\R^*)^k $) and such that each $b_j$ and each $a_j$ is contained in a single connected component of $\mathring{S^k_d} \times U_\R^k$.

\end{lemma}

\begin{proof}
The main idea of the demonstration is to apply 
Alexander duality (see \cite{FomenkoFuchs}) to each connected component of $\mathring{S^k_d} \times U_\R^k$ and its intersection with the chart $Chart_{\mathring{S^k_d} \times U_\R^k}(P)$ by seeing them as subsets of the sphere $S^k$ - this yields cycles and axes in $S^k$ in sufficient numbers. Using a few topological tricks, it is then easy to show that enough of these are also in $\mathring{S^k_d} \times U_\R^k$ and satisfy the conditions of the Lemma.

We know that the inclusion $(\mathring{S^k_d} \times U_\R^k, Chart_{\mathring{S^k_d} \times U_\R^k}(P) )\xhookrightarrow{in} (S^k_d \times U_\R^k, Chart_{S^k_d \times U_\R^k}(P) )$ is a homotopy equivalence of pairs, and that in particular, $H_i(Chart_{\mathring{S^k_d} \times U_\R^k}(P)) \xlongrightarrow{in_*} H_i(Chart_{S^k_d \times U_\R^k}(P))$ is an isomorphism.

Let $\epsilon \in U^k_\R$ and consider the quadrant $\Delta_\epsilon:=S^k_d \times\{ \epsilon \} $, which is one of the $2^k$ connected components of $S^k_d\times U^k_\R$. Let also $X_\epsilon := Chart_{S^k_d \times U_\R^k}(P)\cap \Delta_\epsilon$.

See $\Delta_\epsilon$ as a subset of $\R^k$ by identifying it with $S^k_d \subset \R^k$, and see $\R^k$ as a subset of the sphere $S^k$ (via Alexandroff's compactification). Consider
$$U:=S^k\backslash \left(\{(x_1,\ldots,x_k) \in \R^{k} \;| \; x_j\geq \delta  \; \forall j=1,\ldots,k, \; \sum_{j=1}^{k} x_j\leq d-\delta \} \cup Chart_{S^k_d \times U_\R^k}(P)\right) $$
for some $\delta>0$ (see Figure \ref{FigureConstructionsEnsembleU}). If $\delta$ is small enough (as $X_\epsilon$ is a manifold with boundaries in $\partial \Delta_\epsilon$, which it intersects transversally), which we assume to be the case, $U$ can be retracted to $S^k\backslash S^k_d$ and is thus contractible. We can also assume that $U\cap (\mathring{\Delta_\epsilon} \backslash X_\epsilon)$ is homotopically equivalent to $\partial \Delta_\epsilon \backslash X_\epsilon$.

By considering the Mayer-Vietoris sequence of the sets $U$, $\mathring{\Delta_\epsilon} \backslash X_\epsilon$, $U\cup \mathring{\Delta_\epsilon} \backslash X_\epsilon = S^k\backslash X_\epsilon$ and  $U\cap (\mathring{\Delta_\epsilon} \backslash X_\epsilon)\cong \partial \Delta_\epsilon \backslash X_\epsilon$, we see that the morphism
\begin{equation*}
i_*:H_{k-1-i}(\mathring{\Delta_\epsilon} \backslash X_\epsilon)\longrightarrow H_{k-1-i}(S^k\backslash X_\epsilon)  
\end{equation*}
induced by the inclusion has a cokernel of dimension at most $b_*(\partial \Delta_\epsilon \backslash X_\epsilon)$.

Alexander duality (see \cite{FomenkoFuchs}) can be applied to $X_\epsilon \subset S^k$, as it is compact and locally contractible.
It states that the product
$$ l: \tilde{H}_{k-1-i}(S^k\backslash X_\epsilon)\times\tilde{H}_i(X_\epsilon) \longrightarrow\Z_2$$
where as above $l$ is the linking number (defined as in $\R^k$), is non-degenerate.
In particular, $\tilde{H}_{k-1-i}(S^k\backslash X_\epsilon)$ and  $\tilde{H}_i(X_\epsilon)$ are of the same dimension $\dim \left(\tilde{H}_i(X_\epsilon) \right)\geq b_i(X_\epsilon)-1$.

\begin{figure}
\begin{center}
\includegraphics[scale=0.65]{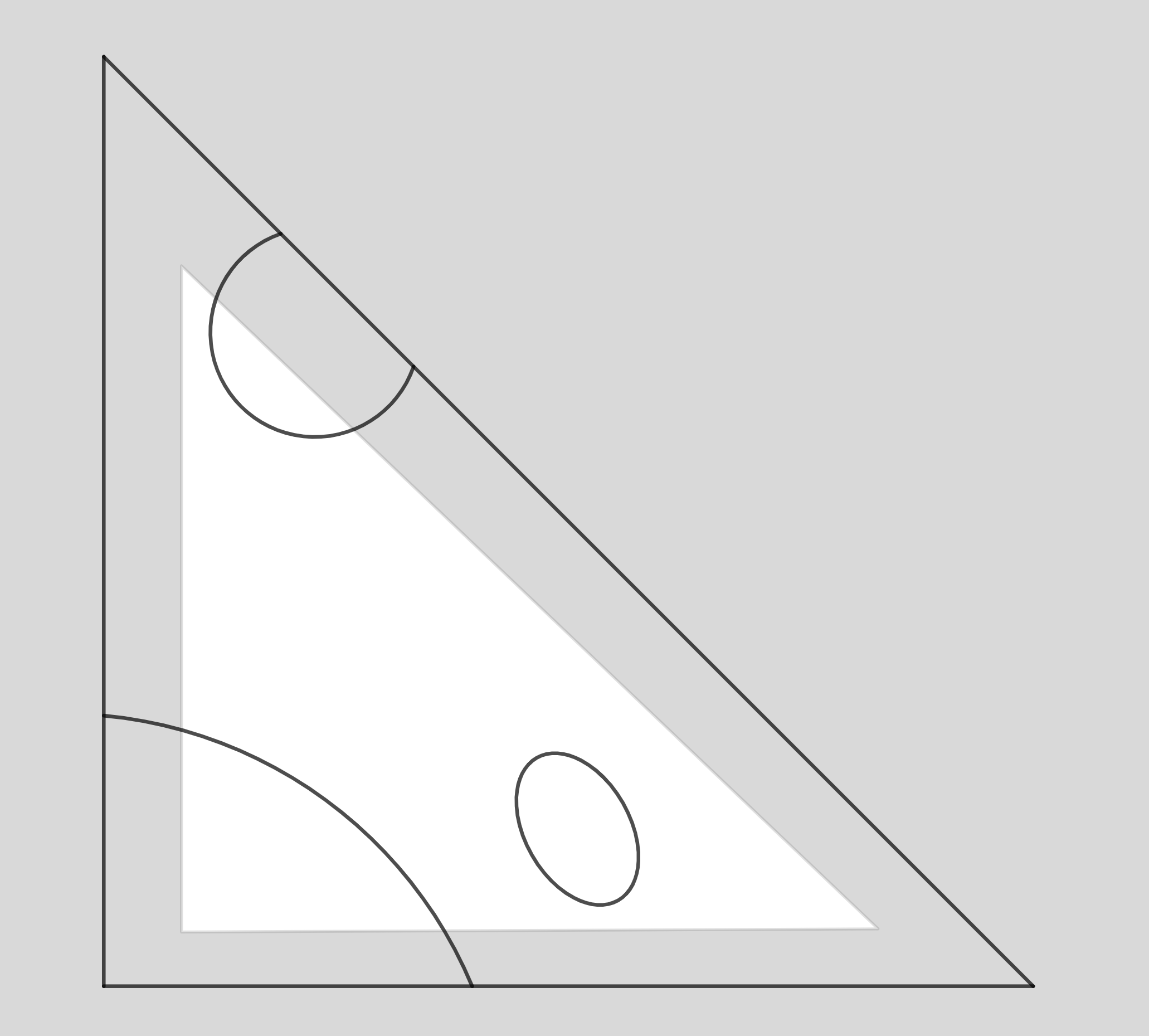}

\end{center}
\caption{
In light grey, the set $U$ illustrated.}
\label{FigureConstructionsEnsembleU}
\end{figure}

If $k-1-i>0$, the reduced and regular homology of $S^k\backslash X_\epsilon$ are equal
\begin{equation*}
\tilde{H}_{k-1-i}(S^k\backslash X_\epsilon) = H_{k-1-i}(S^k\backslash X_\epsilon),
\end{equation*}
and one can find  $s_\epsilon \geq b_i(X_\epsilon) - 1- b_*(\partial \Delta_\epsilon \backslash X_\epsilon)$ classes $\beta^\epsilon_1,\ldots,\beta^\epsilon_{s_\epsilon} \in H_{k-1-i}(\mathring{\Delta_\epsilon} \backslash X_\epsilon) $ such that the classes $i_*(\beta^\epsilon_1),\ldots,i_*(\beta^\epsilon_{s_\epsilon}) \in \tilde{H}_{k-1-i}(S^k\backslash X_\epsilon)$ are linearly independent.

Let $b_{t,1}+\ldots + b_{t,i_t}$ be a chain representing $\beta^\epsilon_t$ (for $t=1,\ldots,s_\epsilon$), where each $b_{t,j}$ is connected. As each $b_{t,j}$ is also a cycle, and since the subspace of $\tilde{H}_{k-1-i}(S^k\backslash X_\epsilon)$ generated by $\{i_*[b_{t,j}]|t=1,\ldots,s_\epsilon, \; j=1,\ldots,i_t\}$ contains the subspace generated by $\{i_*(\beta^\epsilon_1),\dots,i_*(\beta^\epsilon_{s_\epsilon})\} $, we can redefine the classes $\beta^\epsilon_t$ and assume that they can each be represented by a connected cycle $b^\epsilon_t \in \beta^\epsilon_t$. In particular, $P$ has constant sign on each cycle $b^\epsilon_t$.
As $H_{k-1-i}(\mathring{\Delta_\epsilon})$ is trivial, each $b^\epsilon_t$ also admits a membrane in $\Delta_\epsilon$.

If $k-1-i=0$, one can likewise find $s_\epsilon \geq b_i(X_\epsilon) - 1- b_*(\partial \Delta_\epsilon \backslash X_\epsilon)$ classes $\beta^\epsilon_1,\ldots,\beta^\epsilon_{s_\epsilon} \in H_{0}(\mathring{\Delta_\epsilon} \backslash X_\epsilon) $ such that the classes $i_*(\beta^\epsilon_1),\ldots,i_*(\beta^\epsilon_{s_\epsilon}) \in H_{0}(S^k\backslash X_\epsilon)$ are linearly independent (note that this time, the classes belong to the non-reduced homology group). Similarly to the previous case, we can assume that each  $\beta^\epsilon_t$ can be represented by a point $p_t$ in $ \mathring{\Delta_\epsilon} \backslash X_\epsilon$. Denote by $\delta_t \in \{+,-\}$ the sign of $P(p_t)$.
Choose $p_+,p_- \in \mathring{\Delta_\epsilon} \backslash X_\epsilon $ such that $P$ takes positive value on $p_+$ and negative value on $p_-$ (if $P$ has constant sign on $\mathring{\Delta_\epsilon} $, then $X_\epsilon=\emptyset$, $s_\epsilon=0$ and we have nothing to do).

Now consider the family of classes $[p_1 + p_{\delta_1}],\ldots,[p_{s_\epsilon} + p_{\delta_{s_\epsilon}}] \in H_0(\mathring{\Delta_\epsilon} \backslash X_\epsilon)$.
The family $i_*[p_1 + p_{\delta_1}],\ldots,i_*[p_{s_\epsilon} + p_{\delta_{s_\epsilon}}]\in \text{im} \left(\tilde{H}_{0}(S^k\backslash X_\epsilon)  \xhookrightarrow{} H_{0}(S^k\backslash X_\epsilon) \right)$
has rank at least ${s_\epsilon}-1$; by taking out one element (without loss of generality, the one numbered ${s_\epsilon}$), we can once again assume that it is independent.
Redefine $\beta^\epsilon_t:=[p_t + p_{\delta_t}]$ for $t=1,\ldots,{s_\epsilon}-1$. Now $\beta^\epsilon_t\in \ker (H_0 ((\mathring{S^k_d} \times U_\R^k) \backslash Chart_{\mathring{S^k_d} \times U_\R^k}(P)) \longrightarrow H_0(\mathring{S^k_d} \times U_\R^k)) $ (hence we can use it to compute linking numbers) and it can be represented by cycles on which $P$ has constant sign.

Applying Alexander duality to $X_\epsilon \subset S^k$, and using the fact that $\tilde{H}_i(\mathring{\Delta_\epsilon} \cap X_\epsilon) \xlongrightarrow{in_*} \tilde{H}_i(X_\epsilon) $ is an isomorphism (as is the case when considering the entire space $S^k_d\times U_\R^k = \bigcup_\epsilon \Delta_\epsilon$), we can now find classes $\alpha^\epsilon_1,\ldots \alpha^\epsilon_{{s_\epsilon}-1} \in \tilde{H}_i(\mathring{\Delta_\epsilon} \cap X_\epsilon)$ such that their linking number in $S^k$ satisfies $l_{S^k}(in_*(\alpha^\epsilon_s),i_*(\beta^\epsilon_t))=\delta_{s,t}$ for $s,t\in \{1,\ldots,{s_\epsilon}-1\}$.

Now consider the sets $B:=\bigcup _{\epsilon \in U_\R ^k} \{\beta_1^\epsilon,\ldots,\beta_{s_\epsilon -1}\} \subset  H_{k-1-i}((\mathring{S^k_d} \times U_\R^k) \backslash Chart_{\mathring{S^k_d} \times U_\R^k}(P))$ and  $A:=\bigcup _{\epsilon \in U_\R ^k} \{\alpha_1^\epsilon,\ldots,\alpha_{s_\epsilon -1}\} \subset H_i(Chart_{\mathring{S^k_d} \times U_\R^k}(P))$. Note that we identify the reduced classes $\alpha_l^\epsilon$ with the corresponding non-reduced classes.

Let us compute the linking number of $\beta \in B$ and $\alpha \in A$ in $\mathring{S^k_d}\times U^k_\R$. There exists $\epsilon_1,\epsilon_2 \in U^k_\R $ and indices $s,t$ such that $\alpha = \alpha_s^{\epsilon_1}$ and $\beta=\beta_t^{\epsilon_2}$. As explained above, $\beta$ can be represented by a chain $b$ such that it is the boundary of a membrane $m$ in $\mathring{\Delta_{\epsilon_2}}$. Let $a\in \alpha$ be a cycle in $\mathring{\Delta}_{\epsilon_1}$. The linking number in $\mathring{S^k_d}\times U^k_\R$ of $\beta$ and $\alpha$ is the (transversal) intersection number of $m$ and $a$.
If $\epsilon_1 \neq \epsilon_2$, this intersection is necessarily empty.
If $\epsilon_1=\epsilon_2$, $m$ is also a membrane for $b$ in $S^k$ (via the inclusion $\Delta_{\epsilon_1} \subset S^k$), so the linking number in $\mathring{S^k_d}\times U^k_\R$ is the same as in $S^k$ (the intersection number of $a$ and $m$), and thus equal to $\delta_{s,t}$.

We can rename the elements of $B$ (respectively, $A$) as $\beta_1,\ldots,\beta_r$ (respectively, $\alpha_1,\ldots,\alpha_r$), where $r:=\sum_{\epsilon\in U^k_\R} (s_\epsilon -1)$.
We have shown that the elements of the sets $B$ and $A$ are as required in the statement of the lemma. We only have to prove that we have enough of them.

We see that $$r=\sum_\epsilon (s_\epsilon -1)\geq \sum_\epsilon ( b_i(X_\epsilon) - (b_*(\partial \Delta_\epsilon \backslash X_\epsilon)+2))=b_i(V_{(\R^*)^k}(P)) - 2\cdot2^k -\sum_\epsilon b_*(\partial \Delta_\epsilon \backslash X_\epsilon).$$

For a given $\epsilon \in U_\R^k$, $\partial \Delta_\epsilon$ is homeomorphic to the ($k-1$)-sphere. We can once again apply Alexander duality to see that $b_*(\partial \Delta_\epsilon \backslash X_\epsilon)\leq b_*( \partial \Delta_\epsilon \cap X_\epsilon)+1$. Moreover, using arguments similar to those in the proof of Lemma \ref{LemmaConstructionsEquivalenceOuvertFerme}, one can show that there exists $D_1(k)$, depending only on $k$, such that $b_*(  \partial \Delta_\epsilon \cap X_\epsilon )\leq D_1(k) d^{k-1}$.

Hence $$r\geq b_i(V_{(\R^*)^k}(P)) - 2\cdot2^k -\sum_\epsilon b_*(\partial \Delta_\epsilon \backslash X_\epsilon) \geq b_i(V_{(\R^*)^k}(P)) -2^k(3+D_1(k)d^{k-1}).$$

By setting $D(k):=2^k(3+D_1(k))$, we can conclude.

\end{proof}
\begin{remark}
In the light of Lemma \ref{LemmaConstructionsEquivalenceOuvertFerme} and Remark \ref{RemarkConstructionsEquivalenceOuvertFerme}, the condition $r\geq b_i( V_{(\R^*)^k}(P)) - D(k) d^{k-1}$ in the statement can be indifferently replaced by $r\geq b_i( V_{\R\PP^k}(P)) - D(k) d^{k-1}$.

\end{remark}

\begin{remark}
This can easily be generalized to polynomials whose Newton polytope is not a simplex.
\end{remark}

\subsection{Finding cycles in a suspension}\label{SubsectionConstructionsCyclesSuspension}

The next two propositions are based on rather simple ideas, but the many indices and small technical details involved make for long demonstrations. We include a short summary of each proof at their beginning, and do not give the full proof of the second proposition.

We want to find a lower bound on the number of cycles and axes associated to each of the ``pieces" $G_m^+$, $G_m^-$ and $F^k_m$ from the Construction Proposition \ref{PropositionConstructionsMethod} using Lemma \ref{LemmaConstructionCyclesEtAxesBienInclus}.
We start with the case corresponding to $G_m^{\pm}$.
The case corresponding to $F^k_m$ is considered in the next subsection.

\begin{proposition}\label{PropositionConstructionsHomologySuspension}
For all $k\geq 1$, there is a constant $E(k)>0$ with the following property:
Let $P$ be a completely nondegenerate polynomial of degree  $d\geq 1$  in $k$ variables such that $\Delta(P)$ is a translate of $S^k_d$, and let $i \in \{0,1,\ldots,k\}$. Let $\lambda^{+},\lambda^{-} \in \R_{>0}$.
Write
$$G^+: (x_1,\ldots,x_{k},x_{k+1}) \mapsto P(x_1,\ldots, x_{k})+\lambda^+ \cdot x_{k+1}$$
and
$$G^-: (x_1,\ldots,x_{k},x_{k+1}) \mapsto P(x_1,\ldots, x_{k})+\lambda^- \cdot x_{k+1}^{-1}.$$

Define $X:=( Chart_{\Delta(G^+) \times U_\R^{k+1}}(G^+)  \cup (Chart_{\Delta(G^-) \times U_\R^{k+1}}(G^-) ) \cap Int(\Delta(G^+) \cup\Delta(G^-))\times U_\R^{k+1} ) $
(here, $\Delta(G^+)$ and  $\Delta(G^-)$ are seen as subsets of the same ambient space $\R^{k+1}$; they are ($k+1$)-simplices with a common $k$-face $\Delta(P)$).

Then there exists
$$r\geq b_i( V_{(\R^*)^k}(P))+b_{i-1}( V_{(\R^*)^k}(P)) - E(k) d^{k-1}$$
such that we can find classes $\alpha_1,\ldots,\alpha_r$ in $ H_i(X)$
and $\beta_1,\ldots,\beta_r $ in $ H_{k-i}(Int(\Delta(G^+) \cup\Delta(G^-))\times U_\R^{k+1}  \backslash X)$
(respectively, in $\ker (H_0 (Int(\Delta(G^+) \cup\Delta(G^-))\times U_\R^{k+1}  \backslash X) \longrightarrow H_0(Int(\Delta(G^+) \cup\Delta(G^-))\times U_\R^{k+1})$ if $k-i=0$) whose linking numbers in $Int(\Delta(G^+) \cup\Delta(G^-))\times U_\R^{k+1}$ satisfy $l(\alpha_s,\beta_t)=\delta_{s,t}\in \Z_2$ (the classes $\beta_t$ are axes for the classes $\alpha_s$).

\end{proposition}

\begin{proof}
The main idea here is that for each homology class of degree $j$ in $V_{(\R^*)^n}(P)$, there is a class of degree $j$ in the hypersurface corresponding to the patchworking of $G^+$ and $G^-$ (which comes from the inclusion of the original class), and another class of degree $j+1$ corresponding to some kind of suspension of a cycle representing the original class. The same can be said of the classes in the complement of the hypersurface that we use as axes. By proceeding carefully, we can reach that those new classes still have the right linking numbers properties.

Define $X_0 :=X\cap ( \Delta(P) \times U_\R^{k+1}) \subset Int(\Delta(G^+) \cup\Delta(G^-))\times U_\R^{k+1}  $, as well as $X_0^+:= X\cap ( \Delta(P) \times U_\R^{k} \times \{1\})$ and $X_0^-:= X\cap ( \Delta(P) \times U_\R^{k} \times \{-1\})$. Both $X_0^+$ and $X_0^-$ are copies of $Chart_{\mathring{\Delta}(P)\times U_\R^k}(P)$, and $X_0=X_0^+ \cup X_0^-$.

Observe that if $\Delta(P)$ is a translate of $S^k_d$ rather than $S^k_d$ itself, there is a monomial $x^\omega$ such that $\Delta(x^\omega P) =S^k_d$. Moreover, $x^\omega P$ and $P$ give rise to the same hypersurface in $(\R^*)^k$, hence in the toric varieties $\R\Delta(x^\omega P)$ and $\R\Delta(P)$; finally, the pairs $(\Delta(x^\omega P)\times U^k_\R,Chart_{\Delta(x^\omega P)\times U^k_\R}(x^\omega P))$ and $(\Delta(P)\times U^k_\R,Chart_{\Delta(P)\times U^k_\R}(P))$ are trivially isomorphic.
This nuance has no impact on the rest of the proof either.

Using an isomorphic change of variables, we can assume $\lambda^\pm$ to be equal to $1$.

Under the above assumption, note also that the change of variables $(x_1\ldots,x_k,x_{k+1}) \mapsto (x_1\ldots,x_k,x_{k+1}^{-1})$ (well defined on $(\R^*)^{k+1}$) induces an homeomorphism of pairs between $( \Delta(G^+) \times U_\R^{k+1}, Chart_{\Delta(G^+) \times U_\R^{k+1}}(G^+)  )$ and $ (\Delta(G^-) \times U_\R^{k+1}, Chart_{\Delta(G^-) \times U_\R^{k+1}}(G^-))$ (corresponding simply to a vertical symmetry of $\Delta(G^+)$).

Using Lemma \ref{LemmaConstructionCyclesEtAxesBienInclus}, we can produce:
\begin{itemize}
    \item classes $\tilde{\alpha}_1,\ldots,\tilde{\alpha}_{r_1} \in H_i(Chart_{\mathring{\Delta}(P) \times U_\R^{k} }(P))$ and $\tilde{\beta}_1,\ldots,\tilde{\beta}_{r_1} $ in $ H_{k-1-i}( (\mathring{\Delta}(P) \times U_\R^{k}) \backslash Chart_{\mathring{\Delta}(P) \times U_\R^{k} }(P))$
    (respectively, in $\ker (H_0 ((\mathring{\Delta}(P) \times U_\R^{k} ) \backslash Chart_{\mathring{\Delta}(P) \times U_\R^{k} }(P)) \allowbreak \longrightarrow H_0(\mathring{\Delta}(P) \times U_\R^{k} )$ if $k-1-i=0$),
    as well as cycles $\tilde{a}_1\in \tilde{\alpha}_1,\ldots,\tilde{a}_{r_1}\in \tilde{\alpha}_{r_1} $ and $\tilde{b}_1\in \tilde{\beta}_1,\ldots,\tilde{b}_{r_1}\in \tilde{\beta}_{r_1} $
    
     \item classes $\tilde{\gamma}_1,\ldots,\tilde{\gamma}_{r_2} \in H_{i-1}(Chart_{\mathring{\Delta}(P) \times U_\R^{k} }(P))$ and $\tilde{\delta}_1,\ldots,\tilde{\delta}_{r_2} $ in $ H_{k-i}( (\mathring{\Delta}(P) \times U_\R^{k} ) \backslash Chart_{\mathring{\Delta}(P) \times U_\R^{k} }(P))$
     (respectively, in $\ker (H_0 ((\mathring{\Delta}(P) \times U_\R^{k} ) \backslash Chart_{\mathring{\Delta}(P) \times U_\R^{k} }(P)) \allowbreak \longrightarrow  H_0(\mathring{\Delta}(P) \times U_\R^{k})$ if $k-i=0$),
     as well as cycles $\tilde{c}_1\in \tilde{\gamma}_1,\ldots,\tilde{c}_{r_2}\in \tilde{\gamma}_{r_2} $ and $\tilde{d}_1\in \tilde{\delta}_1,\ldots,\tilde{d}_{r_2}\in \tilde{\delta}_{r_2} $

\end{itemize}
where each pair of families of classes and associated cycles verifies the conditions of Lemma \ref{LemmaConstructionCyclesEtAxesBienInclus} (the classes $\tilde{\beta}_j$ are axes to the classes $\tilde{\alpha}_j$, $P$ has constant sign over each cycle $\tilde{b}_j$ or $\tilde{d}_j$, each cycle $\tilde{a}_j$, $\tilde{b}_j$, $\tilde{c}_j$ or $\tilde{d}_j$ is contained in a single quadrant $\mathring{\Delta}(P) \times \{\epsilon\}$, etc.), $r_1 = \max(b_{i}(V_{(\R^*)^k}(P)) -D(k)d^{k-1},0)$ and $r_2=\max(b_{i-1}(V_{(\R^*)^k}(P)) -D(k)d^{k-1},0)$.

Moreover, if $i-1>0$, observe that each $\tilde{c}_t$ is a boundary in $\mathring{\Delta}(P) \times U_\R^{k}$. If $i-1=-1$, $r_2=0$ and it is also (trivially) true.
If $i-1=0$, we can still assume this to be the case: going back to the proof of Lemma \ref{LemmaConstructionCyclesEtAxesBienInclus}, we see that the homology classes it yields (the classes $\tilde{\gamma}_t$ in this case) are all the images of reduced homology classes, and as a result can be represented by chains $\tilde{c}_t$ that are boundaries in the ambient space $\mathring{\Delta}(P) \times U_\R^{k}$ and satisfy all the conditions of the lemma.


We also define
\begin{itemize}
    
    \item classes $\alpha^\pm_1,\ldots,\alpha^\pm_{r_1} \in H_i(X_0^\pm)$ and $\beta^\pm_1,\ldots,\beta^\pm_{r_1} $ in $ H_{k-1-i}( (\mathring{\Delta}(P) \times U_\R^{k} \times \{\pm 1\}) \backslash X_0^\pm)$
    (respectively, in $\ker (H_0 ((\mathring{\Delta}(P) \times U_\R^{k} \times \{\pm 1\}) \backslash X_0^\pm) \longrightarrow H_0(\mathring{\Delta}(P) \times U_\R^{k} \times \{\pm 1\})$ if $k-1-i=0$)
    as well as cycles $a^\pm_1\in \alpha^\pm_1,\ldots,a^\pm_{r_1}\in \alpha^\pm_{r_1} $ and $b^\pm_1\in \beta^\pm_1,\ldots,b^\pm_{r_1}\in \beta^\pm_{r_1} $ as copies of $\tilde{\alpha}_t, \tilde{\beta}_t, \tilde{a}_t$ and $\tilde{b}_t$ in $X_0^\pm$ and $(\mathring{\Delta}(P) \times U_\R^{k} \times \{\pm 1\}) \backslash X_0^\pm$ via the identification of $(\mathring{\Delta}(P)\times U_\R^k\times\{ \pm 1\}, X_0^\pm)$ with $(\mathring{\Delta}(P)\times U_\R^k, Chart_{\mathring{\Delta}(P)\times U_\R^k}(P))$.
     
    \item classes $\gamma^\pm_1,\ldots,\gamma^\pm_{r_2} \in H_{i-1}(X_0^\pm)$ and $\delta^\pm_1,\ldots,\delta^\pm_{r_2} $ in $ H_{k-i}( (\mathring{\Delta}(P) \times U_\R^{k} \times \{\pm 1\}) \backslash X_0^\pm)$
    (respectively, in $\ker (H_0 ((\mathring{\Delta}(P) \times U_\R^{k} \times \{\pm 1\}) \backslash X_0^\pm) \longrightarrow H_0(\mathring{\Delta}(P) \times U_\R^{k} \times \{\pm 1\})$ if $k-i=0$)
    as well as cycles $c^\pm_1\in \gamma^\pm_1,\ldots,c^\pm_{r_2}\in \gamma^\pm_{r_2} $ and $d^\pm_1\in \delta^\pm_1,\ldots,d^\pm_{r_2}\in \delta^\pm_{r_2} $ as copies of $\tilde{\gamma}_t, \tilde{\delta}_t, \tilde{c}_t$ and $\tilde{d}_t$ in $X_0^\pm$ and $(\mathring{\Delta}(P) \times U_\R^{k} \times \{\pm 1\}) \backslash X_0^\pm$ via the identification of $(\mathring{\Delta}(P)\times U_\R^k\times\{ \pm 1\}, X_0^\pm)$ with $(\mathring{\Delta}(P)\times U_\R^k, Chart_{\mathring{\Delta}(P)\times U_\R^k}(P))$.
     
\end{itemize}

Consider the sets $\Gamma^+_{\pm}=(Int(\Delta(G^+)\cup \Delta(G^-) ) \cap \Delta(G^+))\times U_\R ^{k}\times \{\pm\}$ and  $\Gamma^-_{\pm}=(Int(\Delta(G^+)\cup \Delta(G^-) ) \cap \Delta(G^-))\times U_\R ^{k}\times \{\pm\}$ and notice that there are pair homeomorphisms (coming from the definitions of charts and toric varieties)
\begin{itemize}
    \item $\phi^+_+:(\Gamma^+_+,\Gamma^+_+ \cap X)  \longrightarrow ((\R^*)^k \times \R_{\geq 0}, V_{(\R^*)^k \times \R_{\geq 0}}(G^+))$,
    \item $\phi^+_-:(\Gamma^+_-,\Gamma^+_- \cap X) \longrightarrow ((\R^*)^k \times \R_{\leq 0}, V_{(\R^*)^k \times \R_{\leq 0}}(G^+))$,
   
   as well as

    \item $\phi^-_+:(\Gamma^-_+,\Gamma^-_+ \cap X)  \longrightarrow ((\R^*)^k \times \R_{\geq 0}, V_{(\R^*)^k \times \R_{\geq 0}}(G^+))$,
    \item $\phi^-_-:(\Gamma^-_-,\Gamma^-_- \cap X)  \longrightarrow ((\R^*)^k \times \R_{\leq 0}, V_{(\R^*)^k \times \R_{\leq 0}}(G^+))$
\end{itemize}
induced by the change of variables $(x_1\ldots,x_k,x_{k+1}) \mapsto (x_1\ldots,x_k,x_{k+1}^{-1})$ aforementioned. We use the same notation for the restriction of these homeomorphisms to one of the elements of the corresponding pair.

Each of our $r$ homology classes in $H_i(X)$ will be of one of two types: either the image in $H_i(X)$ of a class of $H_i(X_0)$ (with the suspension of the associated axis), or the suspension of a class of $H_{i-1}(X_0)$ (with the associated axis remaining the same). We proceed in that order.

Let $t\in \{1,\ldots,r_1\}$.
By definition, $P$ has constant sign $\epsilon_t \in \{+,-\}$ when evaluated over the cycles $b_t^{+}$ and $b_t^{-}$ via the proper identifications. Let $m_t$ be a ($k-i $)-membrane in $\Gamma^+_{\epsilon_t}$ whose boundary is $b_t^{\epsilon_t}$.
We define a chain $\tilde{Sb}_t^+$ in $( (\R^*)^k \times \{\epsilon_t x_{k+1} \geq 0\}) \backslash V_{(\R^*)^k \times \{\epsilon_t x_{k+1} \geq 0\}}(G^+) $ as
\begin{gather*}
    \{(x_1,\ldots,x_k,x_{k+1}) |(x_1,\ldots,x_k,0) \in \phi^+_{\epsilon_t} (b_t^{\epsilon_t}), \; \epsilon_t x_{k+1} \in [0,R] \} \\ \bigcup \; \{(x_1,\ldots,x_k,\epsilon_t R) |(x_1,\ldots,x_k,0) \in \phi^+_{\epsilon_t}(m_t) \}  ,
\end{gather*}
for $R>0$ large enough that $\tilde{Sb}_t^+$ does not intersect  $V_{(\R^*)^k \times \{\epsilon_t x_{k+1} \geq 0\}}(G^+)$ (indeed, we have that $\epsilon_t G^+(x_1,\ldots,x_k,\epsilon_t R)=\epsilon_t(P(x_1,\ldots,x_k) + \epsilon_t R)  $ is strictly positive for any $(x_1,\ldots,x_k,0) \in \phi^+_{\epsilon_t}(m_t)$ for $R$ large enough, as $m_t$ is compact).
We let the ($k-i$)-chain $Sb_t^+$ in $\Gamma^+_{\epsilon_t} \backslash X$ be $(\phi^+_{\epsilon_t})^{-1}(\tilde{Sb}_t^+) $.
We also define the ($k-i+1$)-chain
$$M_t^+:= (\phi^+_{\epsilon_t})^{-1}(\{(x_1,\ldots,x_k,x_{k+1}) |(x_1,\ldots,x_k,0) \in \phi^+_{\epsilon_t}(m_t),\; \epsilon_t x_{k+1} \in [0,R]  \}) $$
where $R$ is the same as above.

\begin{figure}
\begin{center}
\includegraphics[scale=1.05]{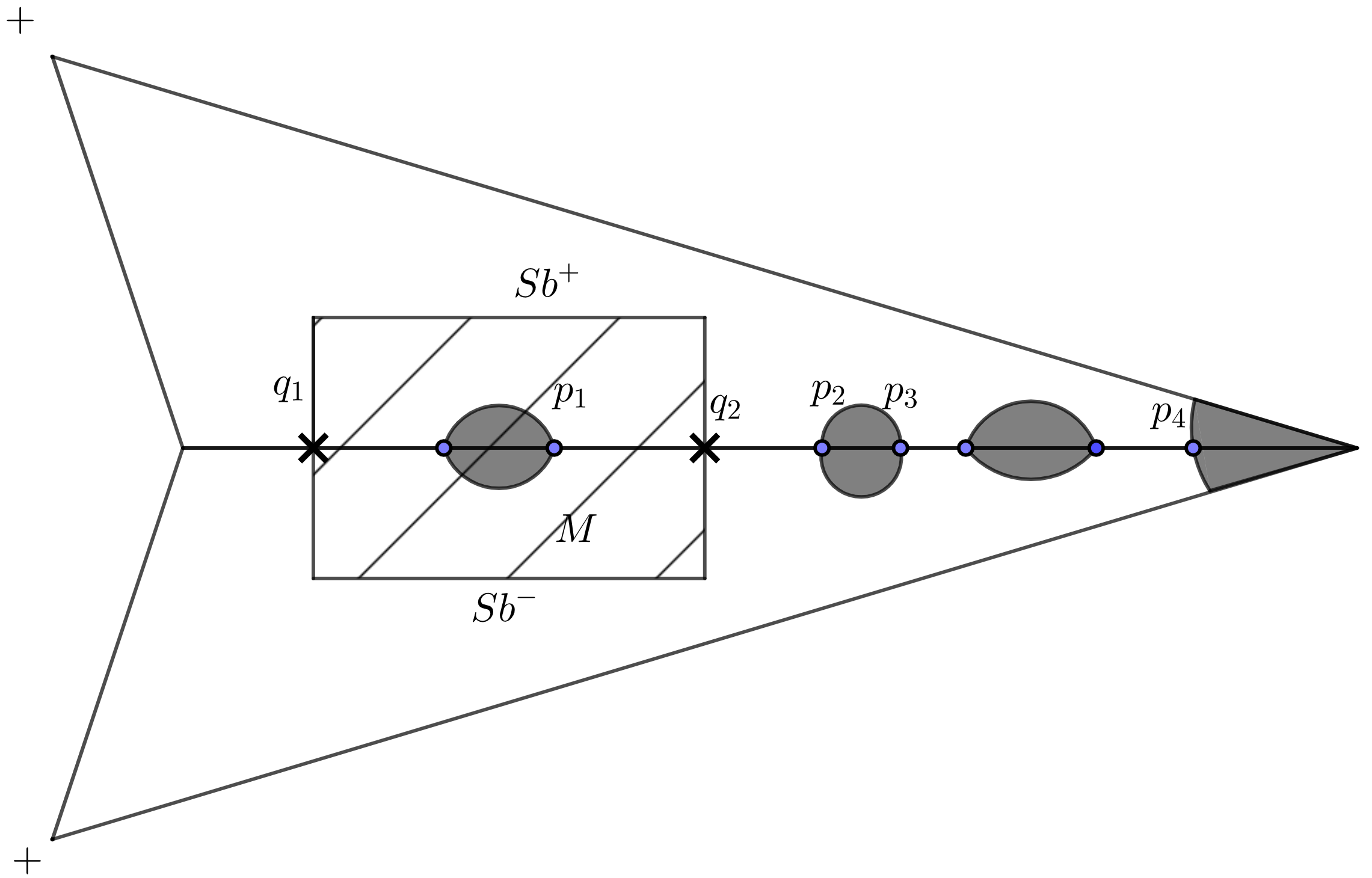}

\end{center}
\caption{
For $k=1$ and $i=0$, the suspension of the axis $\beta=[b]= [q_1 +q_2]$ in $\Delta(G^+)\cup \Delta(G^-)\times \{1\}\times\{1\}$. The cycle is $\alpha =[a]=[p_1+p_2+p_3+p_4]$.
In light grey, the preimage by $\phi ^\pm _+$ of
$(G^\pm)^{-1
} \{ y \leq 0 \}$. The hatched area corresponds to the membrane $M$.}
\label{FigureConstructionsSuspensionAxe}
\end{figure}

We apply the exact same procedure in $\Gamma ^-_{\epsilon_t}$ to get the ($k-i$)-chain $Sb_t^-$ in  $\Gamma^-_{\epsilon_t} \backslash X$ and the ($k-i+1$)-chain $M_t^-$ in  $\Gamma^-_{\epsilon_t}$.

Now we define $Sb_t:=Sb_t^+ + Sb_t^-$ (seen as a chain in $(Int(\Delta(G^+)\cup \Delta(G^-))\times U^{k+1}_\R)\backslash X$) and $M_t:=M_t^+ +M_t^-$ (seen as a chain in $Int(\Delta(G^+)\cup \Delta(G^-))\times U^{k+1}_\R$).
The chain $Sb_t$ is a ($k-i$)-cycle in $(Int(\Delta(G^+)\cup \Delta(G^-))\times U^{k+1}_\R)\backslash X$, and $\partial M_t =Sb_t$, hence $M_t$ can be used as a membrane for $Sb_t$. See Figure \ref{FigureConstructionsSuspensionAxe} for an illustration of this procedure.

We set $A:=\{[a_1^{\epsilon_1}],\ldots,[a_{r_1}^{\epsilon_{r_1}}]\} \subset H_i(X)$ (where we see the cycle $a^{\epsilon_t}_t$ as a cycle in $X$ via the inclusion $X_0^{\epsilon_t} \hookrightarrow X$) and $B:=\{[Sb_1],\ldots,[Sb_{r_1}]\} \subset H_{k-i}((Int(\Delta(G^+)\cup \Delta(G^-))\times U^{k+1}_\R)\backslash X)$ .
The elements of $B$ are axes to the elements of $A$: indeed, let $s,t\in \{1,\ldots,r_1\}$. The linking number $l([a_t^{\epsilon_t}],[Sb_s])$ is equal to the intersection number of $a_t^{\epsilon_t}$ and $M_s$. As $a_t^{\epsilon_t}$ is contained in $X_0$, this number is equal to the intersection number of $a_t^{\epsilon_t}$ and $M_s\cap X_0 = m_s$, which is by definition equal to $\delta_{s,t}$.

We now define the classes of degree $i$ obtained by suspending ($i-1$)-cycles.
Let $t\in \{1,\ldots,r_2\}$, and $\tilde{n}_t$ be a $i$-membrane in $\mathring{\Delta}(P) \times U_\R^{k}$ for $\tilde{c}_t$. Name $n_t^+$ and $n_t^-$ the copies of $\tilde{n}_t$ in $X_0^+$ and $X_0 ^-$ respectively; we have $\partial n_t^\pm = c_t^\pm$.

We define four $i$-chains $Sc_t^{\epsilon_1,\epsilon_2} \subset \Gamma^{\epsilon_1}_{\epsilon_2} \cap X$ (for $\epsilon_1,\epsilon_2 \in \{+,-\}$) as
\vspace{-22.5pt}
\begin{align*}
    Sc_t^{\epsilon_1,\epsilon_2}:=(\phi^{\epsilon_1}_{\epsilon_2})^{-1}(\{(x_1,\ldots,x_k,- P(x_1,\ldots,x_k)) \; |\; (x_1,\ldots,x_k,0) \in \phi^{\epsilon_1}_{\epsilon_2} (n_t^{\epsilon_2}) , \; \\ \epsilon_2 P(x_1,\ldots,x_k) \leq 0 \})
\end{align*}
as well as four corresponding ($i+1$)-chains $N_t^{\epsilon_1,\epsilon_2} \subset \Gamma^{\epsilon_1}_{\epsilon_2} $ as 
\begin{align*}
   N_t^{\epsilon_1,\epsilon_2}:=(\phi^{\epsilon_1}_{\epsilon_2})^{-1}(\{(x_1,\ldots,x_k,x_{k+1}) \;|\;(x_1,\ldots,x_k,0) \in \phi^{\epsilon_1}_{\epsilon_2} (n_t^{\epsilon_2}) , \\ \epsilon_2 P(x_1,\ldots,x_k) \leq -\epsilon_2 x_{k+1}\leq 0 \}). 
\end{align*}

We define $Sc_t:= Sc_t^{+,+}+Sc_t^{+,-}+Sc_t^{-,+}+Sc_t^{-,-}$ (seen as a chain in $X$) and $N_t:= N_t^{+,+}+N_t^{+,-}+N_t^{-,+}+N_t^{-,-}$ (seen as a chain in $Int(\Delta(G^+)\cup \Delta(G^-))\times U^{k+1}_\R$). Note that $Sc_t$ is a cycle and that $\partial N_t = Sc_t$, hence $N_t$ can be used as a membrane for $Sc_t$. See Figure \ref{FigureConstructionsSuspensionCycle} for an illustration of this procedure.

\begin{figure}
\begin{center}
\includegraphics[scale=1.3]{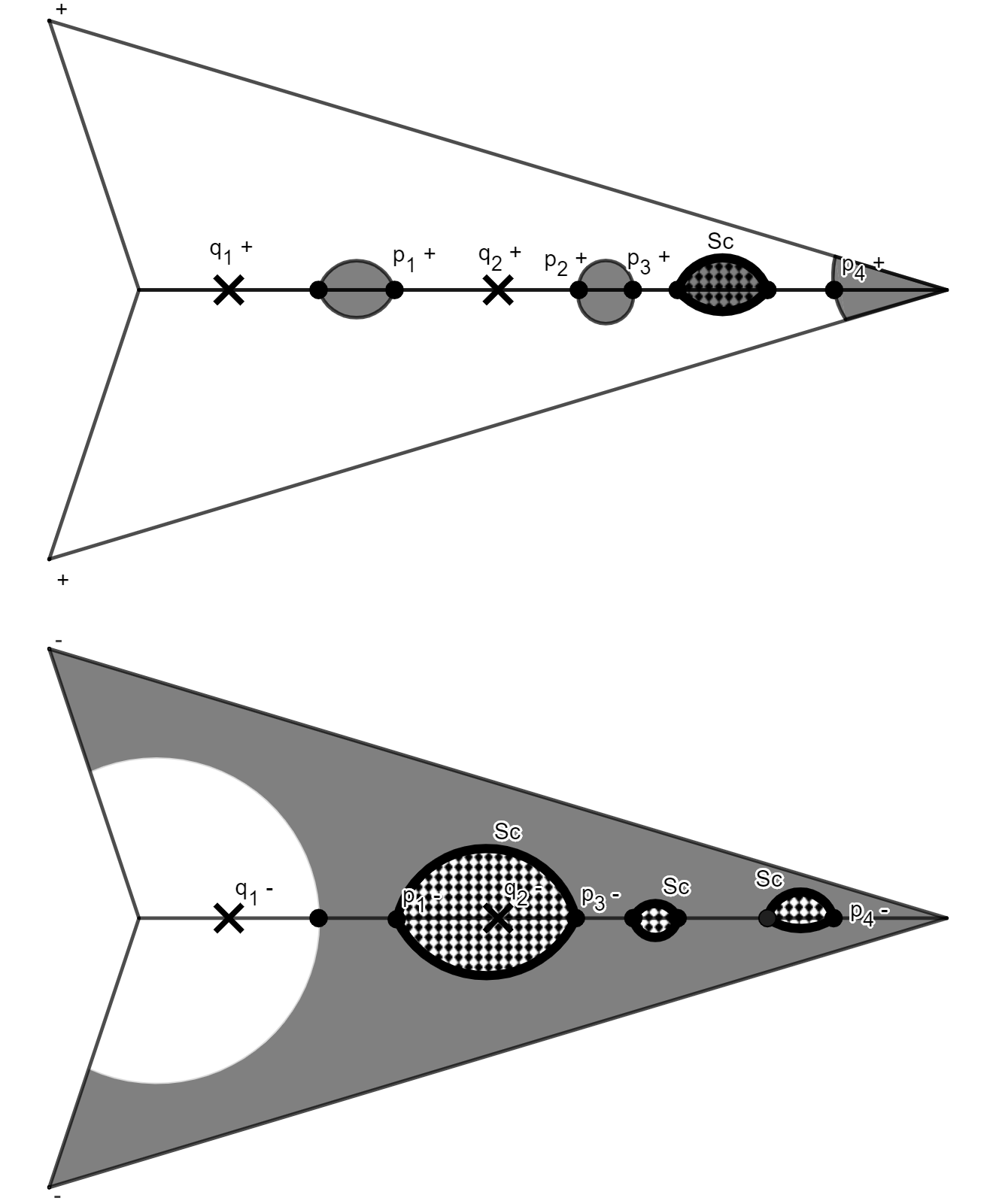}

\end{center}
\caption{
For $k=1$ and $i=1$, the thick black line is the suspension of the cycle $\gamma=[c]= [p_1 +p_2+p_3+p_4]$ in $\Delta(G^+)\cup \Delta(G^-)\times \{1\}\times\{1,-1\}$. The axis is $\delta =[d]=[q_1+q_2]$.
In light grey, the preimage by $\phi ^\pm _+$ of
$(G^\pm)^{-1
} \{ y \leq 0 \}$. The dotted areas correspond to the membrane $N$.}
\label{FigureConstructionsSuspensionCycle}
\end{figure}

By definition, $P$ has constant sign $\rho_t \in \{+,-\}$ when evaluated over the cycles $d_t^{+}$ and $d_t^{-}$.
Define $C:=\{[Sc_1],\ldots,[Sc_{r_2}]\} \subset H_i(X)$ and $D:=\{[d_1^{-\rho_1}],\ldots,[d_{r_2}^{-\rho_{r_2}}]\} \subset H_{k-i}((Int(\Delta(G^+)\cup \Delta(G^-))\times U^{k+1}_\R)\backslash X)$.
The elements of $D$ are axes to the elements of $C$: indeed, let $s,t\in \{1,\ldots,r_2\}$. The linking number $l([Sc_s],[d_t^{-\rho_t}])$ is equal to the intersection number of $N_s$ and $d_t^{-\rho_t}$.
As $d_t^{-\rho_t}$ is contained in $(\phi^{+}_{-\rho_t})^{-1}(\{(x_1,\ldots,x_k,0) | \rho_t P(x_1,\ldots,x_k) \geq 0\} )\subset X_0^{-\rho_t}$, this number is equal to the intersection number of $d_t^{-\rho_t}$ and
\begin{align*}
&N_s \cap (\phi^{+}_{-\rho_t})^{-1}(\{(x_1,\ldots,x_k,0)\;|\; \rho_t P(x_1,\ldots,x_k) \geq 0\} )= \\
&n_s^{-\rho_t}\cap (\phi^{+}_{-\rho_t})^{-1}(\{(x_1,\ldots,x_k,0) \;|\; \rho_t P(x_1,\ldots,x_k) \geq 0\} ),
\end{align*}
which is by definition equal to $\delta_{s,t}$.

We now want to show that the linking number of any element in $A$ and any element in $D$, as well as any element in $B$ and any element in $C$, is $0$.

First, let $[a_t^{\epsilon_t}] \in A$ and $[d_s^{-\rho_s}]\in D$.
Let $o$ be a membrane for $d_s^{-\rho_s}$ in
$\mathring{\Delta}(P) \times U_\R^{k} \times \{-\rho_s\}$. We can slightly rise $o$ and $d_s^{-\rho_s}$ in the following sense:
let
$$o_\lambda := (\phi^{+}_{-\rho_s})^{-1}(\{(x_1,\ldots,x_k,\lambda) \;|\;(x_1,\ldots,x_k,0) \in \phi^{+}_{-\rho_s} (o)  \})$$
and
$$(d_s^{-\rho_s})_\lambda:= (\phi^{+}_{-\rho_s})^{-1}(\{(x_1,\ldots,x_k,\lambda) \;|\;(x_1,\ldots,x_k,0) \in \phi^{+}_{-\rho_s} (d_s^{-\rho_s}).$$
We have $\partial o_\lambda = (d_s^{-\rho_s})_\lambda$ and for $\lambda>0$ small enough, we have $[(d_s^{-\rho_s})_\lambda]=[d_s^{-\rho_s}]\in H_{k-i}((Int(\Delta(G^+)\cup \Delta(G^-))\times U^{k+1}_\R)\backslash X) $ (observe that if $k-i=0$, we have $A=\emptyset$).
For such a small $\lambda$, the linking number of $[a_t^{\epsilon_t}]$ and $[d_s^{-\rho_s}]$ is equal to the intersection number of $o_\lambda$ and $a_t^{\epsilon_t}$, which is $0$ as $a_t^{\epsilon_t}$ is contained in $X_0$ and $o_\lambda$ does not intersect $X_0$.

Then, let $[Sb_t] \in B$ and $[Sc_s]\in C$.
As above, $N_s$ is a membrane for $Sc_s$. Let $\epsilon_t$ be as in the definition of $Sb_t$, and observe that $Sb_t \subset \Gamma^+_{\epsilon_t}\cup\Gamma^-_{\epsilon_t}$.
Observe moreover that
$$ Sb_t \cap \Gamma^{\pm}_{\epsilon_t} \subset (\phi^{\pm}_{\epsilon_t})^{-1}((G^{\pm})^{-1}(\{\epsilon_t y>0\})).$$
On the other hand, 
$$N_s \cap   \Gamma^{\pm}_{\epsilon_t} \subset(\phi^{\pm}_{\epsilon_t})^{-1}((G^{\pm})^{-1}(\{\epsilon_t y \leq 0\})).$$
Hence, the intersection number of $N_s$ and $Sb_t$, which is equal to the linking number of  $[Sc_s]$ and $[Sb_t]$, is $0$.

Note finally that the axes of $D$ were left untouched, and the axes of $B$ are of degree at least $1$; hence, if $k-1-i=0$, all axes in $B\cup D$ automatically belong to $\ker (H_0 (Int(\Delta(G^+) \cup\Delta(G^-))\times U_\R^{k+1}  \backslash X) \longrightarrow H_0(Int(\Delta(G^+) \cup\Delta(G^-))\times U_\R^{k+1})$.

Hence the classes of $A\cup C \subset H_i(X)$ and of $B\cup D \subset H_{k-i}(Int(\Delta(G^+) \cup\Delta(G^-))\times U_\R^{k+1}  \backslash X)$ (respectively, $ D \subset \ker (H_0 (Int(\Delta(G^+) \cup\Delta(G^-))\times U_\R^{k+1}  \backslash X) \longrightarrow H_0(Int(\Delta(G^+) \cup\Delta(G^-))\times U_\R^{k+1})$ if $k-i=0$) satisfy all the conditions of the Proposition. We only have to verify that we have enough of them.

We have found $r_1+r_2 =\max(b_{i}(V_{(\R^*)^k}(P)) -D(k)d^{k-1},0)+\max(b_{i-1}(V_{(\R^*)^k}(P)) -D(k)d^{k-1},0) \geq b_{i}(V_{(\R^*)^k}(P)) +b_{i-1}(V_{(\R^*)^k}(P)) -2D(k)d^{k-1}$ such pair of classes. By setting $E(k):=2D(k)$, we can conclude.

\end{proof}

\subsection{Finding cycles in a join}\label{SubsectionConstructionsCyclesJoin}

We state a similar result concerning the join of two polynomials (corresponding to the polynomials $F^k_m$ from the Construction Proposition \ref{PropositionConstructionsMethod}):

\begin{proposition}\label{PropositionConstructionsHomologyJoin}
For all $n\geq 3$, there is a constant $F(n)>0$ with the following property:
Let $k_1,k_2\geq 1$ be such that $k_1+k_2 =n-1$, and let $i\in \{0,\ldots,n-1\}$.

Let also $P_1$ (respectively, $P_2$) be a completely nondegenerate polynomials in $k_1$ variables (respectively, $k_2$ variables) and degree $d_1\geq 1$ (respectively, $d_2\geq1$) such that $\Delta(P_1)$ is a translate of $S^{k_1}_{d_1}$ (respectively, $\Delta(P_2)$ is a translate of $S^{k_2}_{d_2}$).

Write
$$P: (x_1,\ldots,x_{k_1},y_1,\ldots,y_{k_2},z) \mapsto P_1(x_1,\ldots, x_{k_1})+ z \cdot P_2(y_1,\ldots, y_{k_2}).$$

Define $\Delta_1:=\Delta((x_1,\ldots,x_{k_1},y_1,\ldots,y_{k_2},z) \mapsto P_1(x_1,\ldots, x_{k_1}))\subset \R^n$,

$\Delta_2:= \Delta((x_1,\ldots,x_{k_1},y_1,\ldots,y_{k_2},z) \mapsto z \cdot P_2(y_1,\ldots, y_{k_2}))\subset \R^n$ and 
$\Delta:=\Delta(P)\subset \R^n$. Observe that $\Delta= \Delta_1 \star \Delta_2$, where $\star$ is as above the join.

Define $X:=Chart_{(\mathring{\Delta}_1\star \mathring{\Delta}_2) \times U_\R^{n}}(P)$.

Then there exists
\begin{equation}\label{FormulaConstructionsCyclesJoin}
 r\geq \sum_{j=0}^{i-1} b_j( V_{(\R^*)^{k_1}}(P_1))\cdot b_{i-1-j}( V_{(\R^*)^{k_2}}(P_2)) - F(n) \max{(d_1,d_2)}^{n-2}   
\end{equation}
such that we can find classes $\alpha_1,\ldots,\alpha_r$ in $ H_i(X)$
and $\beta_1,\ldots,\beta_r $ in $ H_{n-1-i}(\mathring{\Delta} \times U_\R^{n} \backslash X)$
such that their linking numbers in $(\mathring{\Delta}_1\star \mathring{\Delta}_2) \times U_\R^{n} $ satisfy $l(\alpha_s,\beta_t)=\delta_{s,t}\in \Z_2$ (the classes $\beta_t$ are axes for the classes $\alpha_s$).
\end{proposition}

\begin{remark}
Remark that the sum in Formula (\ref{FormulaConstructionsCyclesJoin}) is trivial if $i=0,n-1$. Hence, unlike in previous statements, we do not ask that the axes belong to $\ker (H_0 ((\mathring{\Delta}_1\star \mathring{\Delta}_2) \times U_\R^{n} \backslash X) \longrightarrow H_0((\mathring{\Delta}_1\star \mathring{\Delta}_2) \times U_\R^{n} )  )$ if $n-1-i=0$.
\end{remark}
\begin{proof}
The main idea here is that for each $j$-cycle in $V_{(\R^*)^{k_1}}(P_1)$ and ($i-j-1$)-cycle in $V_{(\R^*)^{k_2}}(P_2)$, we can build a $j$-cycle in $X$ by taking the join of the two cycles. If we are cautious enough, we can proceed similarly with the cycles used as axes, and have all classes built in that fashion have the required linking numbers properties. This is illustrated in Figure \ref{FigureConstructionsJoinCyclesAxes}.
As the full proof is even more laborious than that of Proposition \ref{PropositionConstructionsHomologySuspension}, we omit it.
It can be found, with all necessary details, in the author's thesis \cite{ArnalThesis}[Proposition 5.3.7].

\begin{figure}
\begin{center}
\includegraphics[scale=0.55]{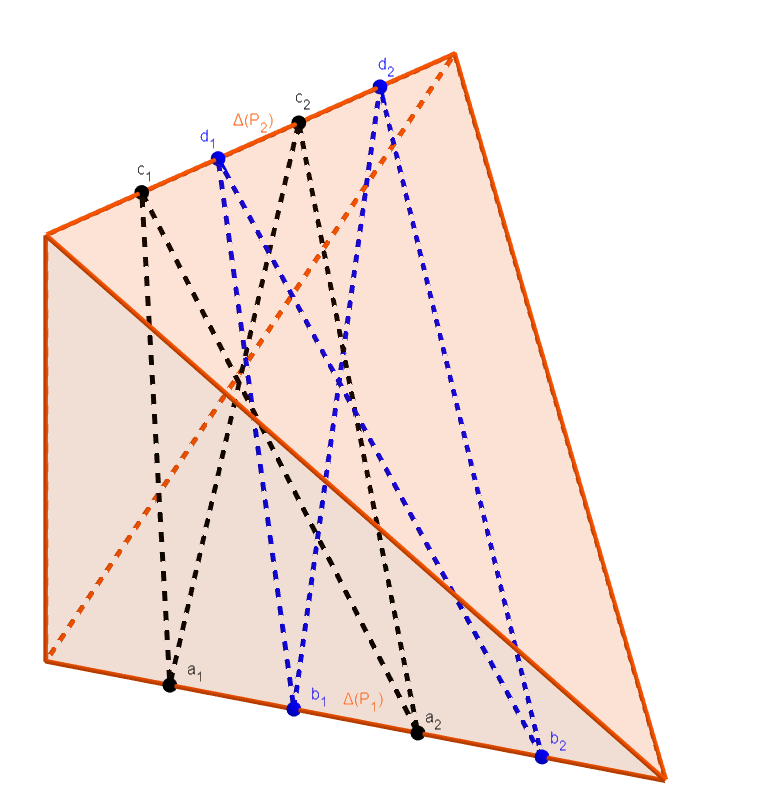}

\end{center}
\caption{
For $k_1=k_2=1$, $i=1$ and $j=0$, the join of cycles $a=a_1+a_2$ and $c=c_1+ c_2$ and the join of axes $b=b_1+b_2$ and $d=d_1+d_2$.}
\label{FigureConstructionsJoinCyclesAxes}
\end{figure}

\end{proof}

\subsection{Counting cycles}\label{SubsectionConstructionsCountingCycles}

We are now ready to prove the Cooking Theorem, which we state again.

\newtheorem*{thmTricheConstructions}{Theorem \ref{TheoremConstructionsMainTheoremConstructions}}

\begin{thmTricheConstructions}[Cooking Theorem]
Let $n\geq 2$. For $k=1,\ldots,n-1$, let $\{P^k_{d}\}_{d\in \NN}$ be a family of completely nondegenerate real Laurent polynomials in $k$ variables, such that $P^k_d$ is of degree $d$ and that the Newton polytope $\Delta(P^k_d)$ of $P^k_d$ is $S^k_d$.
Suppose additionally that for $k=1,\ldots,n-1$ and $i=0,\ldots,k-1$,
 $$b_i(V_{\R\PP ^k}(P^k_d))\overset{k}{\geq} x_i^k \cdot d^k$$
for some $x_i^k \in \R_{\geq 0}$.
Then there exists a family $\{Q^n_d\}_{d\in\NN}$ of completely nondegenerate real Laurent polynomials in $n$ variables such that $\Delta(Q^n_d)=S^n_d$ and such that for $i=0,\ldots,n-1$ 
\begin{equation*}
b_i(V_{\R\PP ^n}(Q^n_d))   \overset{n}{\geq}   \frac{1}{n}(x_i^{n-1}+x_{i-1}^{n-1} +\sum_{k=1}^{n-2}\sum_{j=0}^{i-1} x_j^k \cdot x_{i-1-j}^{n-1-k}) \cdot d^n,
\end{equation*}
where $x_j^k$ is set to be $0$ for $j\not \in \{0,\ldots,k-1\}$.

Moreover, if the families $\{P^k_{d}\}_{d\in \NN}$ were obtained using the combinatorial case of the Patchwork for all $k$, then the family $\{Q^n_{d}\}_{d\in \NN}$ can also be obtained by combinatorial patchworking.

If each family $\{P^k_{d}\}_{d\in \NN}$ (for $k=1,\ldots,n-1$) is such that the associated family of projective hypersurfaces is asymptotically maximal, then the family of projective hypersurfaces associated to $\{Q^n_{d}\}_{d\in \NN}$ is also asymptotically maximal.

\end{thmTricheConstructions}

\begin{proof}
We simply need to apply Propositions \ref{PropositionConstructionsHomologySuspension} and \ref{PropositionConstructionsHomologyJoin} to the polynomials appearing in the Construction Proposition \ref{PropositionConstructionsMethod}.
This gives us a collection of cycles and axes, and by showing that there are asymptotically enough of them, we prove the statement. The suspensions (the polynomials $G_m^+$ and $G_m^-$ in the Construction Proposition) yield the term $\frac{1}{n}(x_i^{n-1}+x_{i-1}^{n-1}) \cdot d^n$ in the statement, while the joins (the polynomials $F^k_m$ in the Construction Proposition) yield the term 
$\frac{1}{n}(\sum_{k=1}^{n-2}\sum_{j=0}^{i-1} x_j^k \cdot x_{i-1-j}^{n-1-k}) \cdot d^n$.

More precisely, for $d\leq n+1$, let $Q^n_d$ be any completely nondegenerate real Laurent polynomial obtained by combinatorial patchworking such that $\Delta(Q^n_d)=S^n_d$ (the choice of $Q^n_d$ matters not, as we are only interested in asymptotic properties), and let $\{Q^n_d\}_{d>n+1}$ be a family of completely nondegenerate real Laurent polynomials such that $\Delta(Q^n_d)=S^n_d$ and satisfying the conclusions of the Construction Proposition \ref{PropositionConstructionsMethod} with regard to the polynomials $P^k_d$ (for $k=1,\ldots,n-1$). We use the same notations as there.

As stated in the Construction Proposition, if the families $\{P^k_{d}\}_{d\in \NN}$ were obtained using combinatorial patchworking, we can assume this to also be the case for $\{Q^n_d\}_{d\in \NN}$.

We will show that $\{Q^n_d\}_{d\in\NN}$ is as wanted. Let $\tilde{C}\geq 0 $ be such that for all $j=0,\ldots,n-1$ and all $k=1,\ldots,n-1$,
we have
$$b_j(V_{\R\PP ^k}(P^k_d))\geq x_j^k \cdot d^k - \tilde{C} d^{k-1},$$
where we set $x_j^k$ to be $0$ if $j\geq k$. Let also $i\in \{0,\ldots,n-1\}$.

We know, from the Main Patchwork Theorem \ref{TheoremPatchworkMainPatchworkTHM}, that the topology of the pairs $(S^n_d\times U^n_\R,Chart_{S^n_d\times U^n_\R}(Q_d^n))$ is the same as that of $(S^n_d\times U^n_\R , v)$, where $v\subset S^n_d \times U_\R^n$ is obtained by appropriately gluing the charts of all polynomials $P_i\in \Sigma$ appearing in the patchworking.

For $m=0,\ldots,d-n-1$ and $k=1,\ldots,n-2$, we consider the polynomial $F^k_m \in \Sigma$.
Based on Condition 2 of the Construction Proposition \ref{PropositionConstructionsMethod}, we know that the chart of $F^k_m$ is homeomorphic to the chart of some polynomial $P$ which satisfies the hypotheses of Proposition \ref{PropositionConstructionsHomologyJoin}, with some polynomials  $\tilde{P}^k_{d-m-1-k}$ and $\tilde{P}^{n-1-k}_{d-m-n+k}$ (whose charts are homeomorphic to those of $P^k_{d-m-1-k}$ and $P^{n-1-k}_{d-m-n+k}$) playing the roles of $P_1$ and $P_2$ in the notations of Proposition \ref{PropositionConstructionsHomologyJoin}. In other words and loosely speaking, $F^k_m$ is the join of $\tilde{P}^k_{d-m-1-k}$ and $\tilde{P}^{n-1-k}_{d-m-n+k}$.
Then the proposition implies that there exists
\begin{align}\label{FormulaConstructionsLowerBoundProof1}
&r(F_m^k) \geq \sum_{j=0}^{i-1} b_j( V_{(\R^*)^{k}}(\tilde{P}^k_{d-m-1-k}))\cdot b_{i-1-j}( V_{(\R^*)^{n-1-k}}(\tilde{P}^{n-1-k}_{d-m-n+k}))\\
&- F(n) \max{(d-m-1-k,d-m-n+k)}^{n-2}\nonumber
\end{align}
such that we can find classes $\tilde{\alpha}_1,\ldots,\tilde{\alpha}_{r(F_m
^k)}$ in $ H_i(Chart_{\Delta(P)\times U^n_\R}(P))$
and $\tilde{\beta}_1,\ldots,\tilde{\beta}_{r(F_m
^k)} $ in $ H_{n-1-i}(\mathring{\Delta}(P)\times U^n_\R \backslash Chart_{\Delta(P)\times U^n_\R}(P))$ such that their linking numbers in $\Delta(P)\times U^n_\R $ satisfy $l(\tilde{\alpha}_s,\tilde{\beta}_t)=\delta_{s,t}\in \Z_2$.

We pull back these classes via the pair homeomorphism to get classes $\alpha^{F_m^k}_1,\ldots,\alpha^{F_m^k}_{r(F_m
^k)}$ in $ H_i(Chart_{\Delta(F_m^k)\times U^n_\R}(F_m^k))$
and $\beta^{F_m^k}_1,\ldots,\beta^{F_m^k}_{r(F_m
^k)} $ in $ H_{n-1-i}(\mathring{\Delta}(F_m^k)\times U^n_\R \backslash Chart_{\Delta(F_m^k)\times U^n_\R}(F_m^k))$ such that their linking numbers in $\Delta(F_m^k)\times U^n_\R $ satisfy $l(\alpha^{F_m^k}_s,\beta^{F_m^k}_t)=\delta_{s,t}\in \Z_2$.

\sloppy Moreover, observe that by the definition of polynomials $\tilde{P}
^k_{d-m-1-k}$ and $\tilde{P}^{n-1-k}_{d-m-n+k}$ from Condition 2 of the Construction Proposition \ref{PropositionConstructionsMethod} (whose charts were assumed to be homeomorphic to those of $P^k_{d-m-1-k}$ and $P^{n-1-k}_{d-m-n+k}$), we have
$b_j( V_{(\R^*)^{k}}(\tilde{P}^k_{d-m-1-k}))= \allowbreak b_j( V_{(\R^*)^{k}}(P^k_{d-m-1-k}))$ and $b_{i-1-j}( V_{(\R^*)^{n-1-k}}(\tilde{P}^{n-1-k}_{d-m-n+k}))=b_{i-1-j}( V_{(\R^*)^{n-1-k}}(P^{n-1-k}_{d-m-n+k}))$, which means that we can rewrite Inequality (\ref{FormulaConstructionsLowerBoundProof1}) as
\begin{align*}
&r(F_m^k) \geq
\sum_{j=0}^{i-1} b_j( V_{(\R^*)^{k}}(P^k_{d-m-1-k}))\cdot b_{i-1-j}( V_{(\R^*)^{n-1-k}}(P^{n-1-k}_{d-m-n+k}))\\
&- F(n) \max{(d-m-1-k,d-m-n+k)}^{n-2}  \geq \\
&\left[ \sum_{j=0}^{i-1} (x_j^k \cdot (d-m-1-k)^k - \tilde{C} \cdot (d-m-1-k)^{k-1})\cdot \right.\\
&\left.\cdot(x_{i-1-j}^{n-1-k} \cdot(d-m-n+k)^{n-1-k} - \tilde{C}\cdot (d-m-n+k)^{n-2-k})\right]\\
&-F(n) \max{(d-m-1-k,d-m-n+k)}^{n-2}
\end{align*}
for any $d$ large enough that $x_j^k \cdot (d-m-1-k)^k \geq \tilde{C} \cdot (d-m-1-k)^{k-1}$ and $x_{i-1-j}^{n-1-k} \cdot (d-m-n+k)^{n-1-k} \geq \tilde{C}\cdot (d-m-n+k)^{n-2-k}$ for all $j=0,\ldots,i-1$.

Define $C_1 := n(2\tilde{C}\max \{x_j^l | l=1,\ldots, n-1, j=0,\ldots,l-1\} + F(n))$.

Then we have 
\begin{align}\label{FormulaConstructionsLowerBoundProof2}
&r(F_m^k) \geq  \left(\sum_{j=0}^{i-1} x_j^k x_{i-1-j}^{n-1-k}\cdot (d-m-1-k)^k (d-m-n+k)^{n-1-k}\right) \\
&- C_1\max{(d-m-1-k,d-m-n+k)}^{n-2} \geq \nonumber\\
&\left(\sum_{j=0}^{i-1} x_j^k x_{i-1-j}^{n-1-k}\right)\cdot (d-m-n)^{n-1} - C_1(d-m-1)^{n-2}\nonumber
\end{align}
for all $d$ large enough; if we replace $C_1$ by  $\tilde{C}_1 \geq C_1$ large enough, we can assume this to be the case for all $d\geq n+1$ (and we do). 

For $m=1,\ldots, d-n-1$, we also consider the pair of polynomials $G^+_m,G^-_m \in \Sigma$ (still using the notations of the Construction Proposition \ref{PropositionConstructionsMethod}).

We know that there exist polynomials $\tilde{G}^+_m$, $\tilde{G}^-_m$ and $\tilde{P}^{n-1}_{d-m-n}$ (with the chart of $\tilde{P}^{n-1}_{d-m-n}$ homeomorphic to the chart of $P^{n-1}_{d-m-n}$) that satisfy the hypotheses of Proposition \ref{PropositionConstructionsHomologySuspension} (where $\tilde{G}^+_m$ and $\tilde{G}^-_m$ correspond to $G^+$ and $G^-$ and $\tilde{P}^{n-1}_{d-m-n}$ to $P$ in the notations of the proposition) such that the pair $((\Delta(G^+_m)\cup\Delta(G^-_m))\times U^n_\R,\allowbreak Chart_{\Delta(G^+_m)\times U^n_\R}(G^+_m)\cup Chart_{\Delta(G^-_m)\times U^n_\R}(G^-_m))$ is homeomorphic to $((\Delta(\tilde{G}^+_m)\cup\Delta(\tilde{G}^-_m))\times U^n_\R,\allowbreak Chart_{\Delta(\tilde{G}^+_m)\times U^n_\R}(\tilde{G}^+_m)\cup Chart_{\Delta(\tilde{G}^-_m)\times U^n_\R}(\tilde{G}^-_m))$.
Loosely speaking, this gluing of charts is homeomorphic to a suspension of the chart of $P^{n-1}_{d-m-n}$.


Then the proposition implies that there exists
\begin{align}\label{FormulaConstructionsLowerBoundProof3}
s(G_m)\geq b_i( V_{(\R^*)^{n-1}}(\tilde{P}^{n-1}_{d-m-n}))+b_{i-1}( V_{(\R^*)^{n-1}}(\tilde{P}^{n-1}_{d-m-n})) - E(n-1) (d-m-n)^{n-2}    
\end{align}
such that we can find classes $\tilde{\alpha}_1,\ldots,\tilde{\alpha}_{s(G_m)}$ in $$ H_i\left(Chart_{\Delta(\tilde{G}^+_m)\times U^n_\R}(\tilde{G}^+_m)\cup Chart_{\Delta(\tilde{G}^-_m)\times U^n_\R}(\tilde{G}^-_m)\right)$$
and $\tilde{\beta}_1,\ldots,\tilde{\beta}_{s(G_m)} $ in {\small$$ H_{n-1-i}\left(\left[Int(\Delta(\tilde{G}^+_m) \cup\Delta(\tilde{G}^-_m))\times U_\R^{n} \right] \backslash \left[ (Chart_{\Delta(\tilde{G}^+_m)\times U^n_\R}(\tilde{G}^+_m)\cup Chart_{\Delta(\tilde{G}^-_m)\times U^n_\R}(\tilde{G}^-_m))\right] \right)$$  }
(respectively, in the kernel of 
\begin{gather*}
H_0 \left(\left[Int(\Delta(\tilde{G}^+_m) \cup\Delta(\tilde{G}^-_m))\times U_\R^{n}\right]  \backslash \left[ (Chart_{\Delta(\tilde{G}^+_m)\times U^n_\R}(\tilde{G}^+_m)\cup Chart_{\Delta(\tilde{G}^-_m)\times U^n_\R}(\tilde{G}^-_m))\right]\right) \\ \longrightarrow \ H_0\left(Int(\Delta(\tilde{G}^+_m) \cup\Delta(\tilde{G}^-_m))\times U_\R^{n}  \right)
\end{gather*} if $n-1-i=0$) whose linking numbers in $Int(\Delta(\tilde{G}^+_m) \cup\Delta(\tilde{G}^-_m))\times U_\R^{n}$ satisfy $l(\tilde{\alpha}_s,\tilde{\beta}_t)=\delta_{s,t}\in \Z_2$.

We pull back these classes via the pair homeomorphism to get classes $\alpha^{G_m}_1,\ldots,\alpha^{G_m}_{s(G_m)}$ in $$ H_i\left(Chart_{\Delta(G^+_m)\times U^n_\R}(G^+_m)\cup Chart_{\Delta(G^-_m)\times U^n_\R}(G^-_m)\right)$$
and $\beta^{G_m}_1,\ldots,\beta^{G_m}_{s(G_m)} $ in {\small $$ H_{n-1-i}\left(\left[Int(\Delta(G^+_m)\cup\Delta(G^-_m))\times U^n_\R \right]\backslash \left[Chart_{\Delta(G^+_m)\times U^n_\R}(G^+_m)\cup Chart_{\Delta(G^-_m)\times U^n_\R}(G^-_m)\right]\right)$$ }
(respectively, in the kernel of 
\begin{gather*}
     H_0 \left(\left[Int(\Delta(G^+_m) \cup\Delta(G^-_m))\times U_\R^{n}  \right]\backslash \left[Chart_{\Delta(G^+_m)\times U^n_\R}(G^+_m)\cup Chart_{\Delta(G^-_m)\times U^n_\R}(G^-_m)\right]\right) \\ \longrightarrow H_0\left(Int(\Delta(G^+_m) \cup\Delta(G^-_m))\times U_\R^{n}  \right)
\end{gather*} if $n-1-i=0$) such that their linking numbers in $Int(\Delta(G^+_m)\cup\Delta(G^-_m))\times U^n_\R$ satisfy $l(\alpha^{G_m}_s,\beta^{G_m}_t)=\delta_{s,t}\in \Z_2$.

Moreover, observe that by the definition of polynomials $\tilde{P}^{n-1}_{d-m-n}$ (whose charts were assumed to be homeomorphic to those of the polynomials $P^{n-1}_{d-m-n}$) from Condition 3 of the Construction Proposition \ref{PropositionConstructionsMethod}, we have
$b_i( V_{(\R^*)^{n-1}}(\tilde{P}^{n-1}_{d-m-n}))=b_i( V_{(\R^*)^{n-1}}(P^{n-1}_{d-m-n}))$ and $b_{i-1}( V_{(\R^*)^{n-1}}(\tilde{P}^{n-1}_{d-m-n}))=b_{i-1}( V_{(\R^*)^{n-1}}(P^{n-1}_{d-m-n}))$, which means that we can rewrite Inequality (\ref{FormulaConstructionsLowerBoundProof3}) as
\begin{align}\label{FormulaConstructionsLowerBoundProof4}
&s(G_m)\geq b_i(V_{(\R^*)^{n-1}}(\tilde{P}^{n-1}_{d-m-n}))+b_{i-1}( V_{(\R^*)^{n-1}}(\tilde{P}^{n-1}_{d-m-n})) - E(n-1) (d-m-n)^{n-2} = \nonumber\\
&b_i( V_{(\R^*)^{n-1}}(P^{n-1}_{d-m-n}))+b_{i-1}( V_{(\R^*)^{n-1}}(P^{n-1}_{d-m-n})) - E(n-1) (d-m-n)^{n-2} \geq \nonumber\\
& x_i^{n-1} \cdot (d-m-n)^{n-1} -\tilde{C}(d-m-n)^{n-2}+ x_{i-1}^{n-1} \cdot (d-m-n)^{n-1} -\tilde{C}(d-m-n)^{n-2}\nonumber\\
&-E(n-1) (d-m-n)^{n-2}= \\
& (x_i^{n-1} + x_{i-1}^{n-1})\cdot (d-m-n)^{n-1} - (2\tilde{C} + E(n-1))\cdot(d-m-n)^{n-2}=\nonumber\\
&(x_i^{n-1} + x_{i-1}^{n-1})\cdot (d-m-n)^{n-1} - C_2\cdot(d-m-n)^{n-2}\nonumber
\end{align}
by setting $C_2:=2\tilde{C} + E(n-1)$.

Now consider the image of the classes $\alpha_s^{F_k^m} $ and $\alpha_s^{G_m}$ (for all $s,k,m$ for which they were defined) in $H_i(v) $ (via the inclusion), where $v\subset S^n_d$ is as above a gluing of the charts of the polynomials of $\Sigma$. Similarly, consider the image of the axes $\beta_t^{F_m^k}$ and $\beta_t^{G_m}$ in $H_{n-1-i}(\mathring{S}^n_d\times U^n _\R \backslash v)$ via the inclusion (respectively, in $\ker (H_0 (\mathring{S}^n_d\times U^n _\R \backslash v) \longrightarrow H_0(\mathring{S}^n_d\times U^n _\R )$ if $n-1-i=0$). We keep the same notations for the images of the classes by the inclusion.

As each axis $\beta_t^{F_m^k}$ or $\beta_t^{G_m}$ is contained in the interior of  $\Delta(F^k_m)\times U^n_\R$ or $\Delta(G^+_m)\cup\Delta(G^-_m)\times U^n_\R$ and is a boundary in that interior, we can find for each a membrane also contained in that interior.
As the interiors of these polytopes are all disjoint (see Condition 4 of the Construction Proposition \ref{PropositionConstructionsMethod}), this shows that the linking number in $S^n_d\times U^n_\R$ of $\beta_t^{F_{m_1}^{k_2}}$ with any $\alpha_s^{F_{m_2}^{k_2}}$ is $\delta_{t,s}\delta_{m_1,m_2}\delta_{k_1,k_2}$, and its linking number with any $\alpha_s^{G_m}$ is $0$.
Similarly, the linking number in $S^n_d\times U^n_\R$ of $\beta_t^{G_{m_1}}$ with any $\alpha_s^{G_{m_2}}$ is $\delta_{t,s}\delta_{m_1,m_2}$, and its linking number with any $\alpha_s^{F^{k}_{m_2}}$ is $0$.

This shows that the elements of $B:=\{\beta_t^{F_m^k}| t=1,\ldots,r(F_m^k), m=0,\ldots,d-n-1, k=1,\ldots,n-2\}\cup\{\beta_t^{G_m}|t=1,\ldots,s(G_m), m=1,\ldots,d-n-1\}$ (with $B\subset H_{n-1-i}(\mathring{S}^n_d\times U^n _\R \backslash v)$, respectively $B\subset \ker (H_0 (\mathring{S}^n_d\times U^n _\R \backslash v) \longrightarrow H_0(\mathring{S}^n_d\times U^n _\R )$ if $n-1-i=0$) are axes to the elements of $A:=\{\alpha_t^{F_m^k}| t=1,\ldots,r(F_m^k), m=0,\ldots,d-n-1, k=1,\ldots,n-2\}\cup\{\alpha_t^{G_m}|t=1,\ldots,s(G_m), m=1,\ldots,d-n-1\}\subset H_i(v)$. In particular, this implies that $b_i(v)\geq |A|$.

We know that $v$ is homeomorphic to $Chart_{S^n_d\times U^n_\R}(Q^n_d) $, which is itself homotopy equivalent to $V_{(\R^*)^n}(Q^n_d)$. Finally, we know from Lemma \ref{LemmaConstructionsEquivalenceOuvertFerme} that there is a constant $C(n)$ (dependent only on $n$) such that $b_i(V_{\R\PP^n}(Q^n_d))\geq b_i(V_{(\R^*)^n}(Q^n_d)) - C(n)d^{n-1} $.
Hence we get that
\begin{align*}
&b_i(V_{\R\PP^n}(Q^n_d))\geq |A| - C(n)d^{n-1}=\\
&\sum_{m=1}^{d-n-1} s(G_m) + \sum_{m=0}^{d-n-1}\sum_{k=1}^{n-2}r(F^k_m) - C(n)d^{n-1}.
\end{align*}

Using the fact that for all $l,p\geq 1$, we have $\sum_{q=1}^p q^l \geq \frac{p^{l+1}}{l+1} - C_3(l)p^l $ for some constant $C_3(l)>0$, and going back to Inequality (\ref{FormulaConstructionsLowerBoundProof4}), we get
{\small
\begin{align*}
&\sum_{m=1}^{d-n-1} s(G_m) \geq 
\sum_{m=1}^{d-n-1} (x_i^{n-1} + x_{i-1}^{n-1})\cdot (d-m-n)^{n-1} - C_2\cdot(d-m-n)^{n-2} \geq \\
&\frac{x_i^{n-1} + x_{i-1}^{n-1}}{n}(d-n-1)^n - (x_i^{n-1} + x_{i-1}^{n-1})C_3(n-1)\cdot(d-n-1)^{n-1} - C_2\cdot(d-n-1)^{n-1}.
\end{align*} }
Observe that each $x_j^s$ is less than, or equal to $1$. Moreover, $(d-n)^n \geq (d-n-1)^n \geq d^n - C_4(n)d^{n-1}$ for some constant $C_4(n)>0$.
We can also set $C_5(n):=\frac{2}{n}C_4(n)+2C_3(n-1)+C_2$, and have
$$ \sum_{m=1}^{d-n-1} s(G_m) \geq \frac{x_i^{n-1} + x_{i-1}^{n-1}}{n}d^n - C_5(n) d^{n-1}. $$

Going back to Inequality (\ref{FormulaConstructionsLowerBoundProof2}), we get
\begin{align*}
&\sum_{m=0}^{d-n-1}\sum_{k=1}^{n-2}r(F^k_m) \geq 
\sum_{m=0}^{d-n-1}\sum_{k=1}^{n-2}\left(\left(\sum_{j=0}^{i-1} x_j^k x_{i-1-j}^{n-1-k}\right)\cdot (d-m-n)^{n-1} - \tilde{C}_1\cdot(d-m-1)^{n-2}\right)\\& \geq 
\sum_{k=1}^{n-2} \left( \left(\sum_{j=0}^{i-1} x_j^k x_{i-1-j}^{n-1-k}\right)  \left( \frac{1}{n}(d-n)^{n} - C_3(n-1)\cdot (d-n)^{n-1}  \right) -\tilde{C}_1\cdot ( d-1)^{n-1} \right).
\end{align*}
Set $C_6(n):=nC_4(n)+n^2 C_3(n-1)+n\tilde{C}_1$. We can write
\begin{align*}
&\sum_{m=0}^{d-n-1}\sum_{k=1}^{n-2}r(F^k_m) \geq 
\left( \sum_{k=1}^{n-2} \sum_{j=0}^{i-1} x_j^k x_{i-1-j}^{n-1-k}\frac{d^n}{n} \right) - C_6(n)d^{n-1}.
\end{align*}

Hence we can conclude that 
\begin{align*}
&b_i(V_{\R\PP^n}(Q^n_d))\geq  \sum_{m=1}^{d-n-1} s(G_m) + \sum_{m=0}^{d-n-1}\sum_{k=1}^{n-2}r(F^k_m) - C(n)d^{n-1}\geq \\
&\frac{d^n}{n}(x_i^{n-1} + x_{i-1}^{n-1} +\sum_{k=1}^{n-2} \sum_{j=0}^{i-1} x_j^k x_{i-1-j}^{n-1-k} ) - (C_5(n) +C_6(n) + C(n)) d^{n-1},
\end{align*}
which is what we wanted to prove.

The statement regarding asymptotic maximality is a direct application of Lemma \ref{LemmaConstructionsAsmptoticallyMaximalPreserve} below.

\end{proof}

%% file: AsymptoticallyLargeBetti.tex
\section{Asymptotically large Betti numbers  in arbitrary dimension and index}\sectionmark{Asymptotically large Betti numbers}\label{SectionAsymptoticallyLargeBettiNumbers}
Before proceeding with the proof of Theorems \ref{TheoremConstructionsApplicationThm} and \ref{TheoremConstructionsBonneAsymptotique}, we make in Subsection \ref{SubsectionConstructionsHodgeNumbersProperties} a few observations concerning Hodge numbers and their relations to some combinatorial concepts; there are indeed many connections between Hodge numbers of algebraic varieties and interesting objects in combinatorics, some of which can be found in M. Baker's survey \cite{HodgeCombinatoricsSurvey}. We also prove results which we later use in Subsection \ref{SubsectionConstructionsFirstConstruction} to show  that some families of real projective  algebraic hypersurfaces that we define using the Cooking Theorem have appropriately large asymptotic Betti numbers, thereby proving Theorems \ref{TheoremConstructionsApplicationThm} and \ref{TheoremConstructionsBonneAsymptotique}.

\subsection{Asymptotic Hodge numbers and combinatorics}\label{SubsectionConstructionsHodgeNumbersProperties}

Let $X_d^n$ be a smooth real algebraic hypersurface of degree $d$ in $\C \PP^n$. Then we have 
\begin{equation}\label{FormulaConstructionsAsymptoticHodgeNumbers}
h^{p,n-1-p}(\C X^n_d)=\sum_{i=0}^{n+1}(-1)^i 
\binom{n+1}{i}\binom{d(p+1)-(d-1)i-1}{n}+\delta_{n-1,2p}
\end{equation}
for $p=0,\ldots,n-1$, where $h^{p,n-1-p}(\C X^n_d)$ is the $(p,n-1-p)$-th Hodge number of $\C X^n_d$ and $\binom{k_1}{k_2}=0$ if $k_1<k_2$ (see \cite{DanilovKhovansky}).

Note that $\sum_{i=0}^{n+1}(-1)^i 
\binom{n+1}{i}\binom{d(p+1)-(d-1)i-1}{n}$ is also equal to the number of ordered ($n+1$)-partitions of $d(p+1)$ such that each of the summands belongs to $\{1,\ldots,d-1\}$. Indeed, the sum can be interpreted as the number of ordered ($n+1$)-partitions of $d(p+1)$ such that each of the summands is greater than or equal to $1$, minus the number of ordered ($n+1$)-partitions of $d(p+1)$ such that each of the summand is greater than or equal to $1$ and at least one of the summand is greater than or equal to $d$. This is expressed using the inclusion-exclusion formula applied to the sets $A_I$, where $A_I$ (for $I\subset \{1,\ldots,n+1\}$) is the set of ordered ($n+1$)-partitions of $d(p+1)$ such that each of the summand is greater than or equal to $1$ and the summands $a_j$ are greater than or equal to $d$ for any $j\in I$ (in the formula, $i$ corresponds to $|I|$).

The asymptotic behavior of such expressions is an interesting topic in itself, related to lattice paths, hypergeometric functions and some probabilistic notions (see for example \cite{BoundedPartitions}).

Note that a more geometric interpretation of Formula (\ref{FormulaConstructionsAsymptoticHodgeNumbers}) can also be given,  as  $\sum_{i=0}^{n+1}(-1)^i\binom{n+1}{i}\binom{d(p+1)-(d-1)i-1}{n}$ is equal to the number of interior integer points in the section of the cube $[0,d]^{n+1}\subset \R^{n+1}$ by the hyperplane $\{\sum_{i=0}^{n+1} x_i = (p+1)d\} \subset \R^{n+1}$. 

For $n\in \NN$, $p\in\{0,\ldots,n-1\}$ and $i\in\{0,\ldots,n-1\}$ given, the expression $\binom{d(p+1)-(d-1)i-1}{n} =\binom{d(p+1-i)+i-1}{n}  $ is a polynomial in $d$ of degree $n$ whose monomial of highest degree is $d^n \frac{(p+1-i)^n}{n!}$ if $i<p+1$, and a constant (in $d$) for $d$ large enough otherwise.
Hence, as $d$ goes to infinity, $h^{p,n-1-p}(\C X^n_d)$ is a polynomial in degree $n$ whose monomial of highest degree is
\begin{equation}\label{FormulaConstructionsAsymptoticHodgeNumbers2}
\frac{d^n}{n!}\left(\sum_{i=0}^{p}(-1)^i\binom{n+1}{i} (p+1-i)^n\right).
\end{equation}
As in Section \ref{SectionIntroduction}, we define $a_p^n := \frac{1}{n!}\left(\sum_{i=0}^{p}(-1)^i\binom{n+1}{i} (p+1-i)^n\right)$ for $p\in\{0,\ldots,n-1\}$. For convenience, we also define $a_p^n := 0$ for any $p\in \Z \backslash\{0,\ldots,n-1\} $. Observe that Formula (\ref{ComplexPartTotalHomology}) implies that $\sum_{p\in\Z} a_p^n =1$ for any $n\geq 1$.

The theory of Ehrhart polynomials tells us that the number of interior integer points in the section of the cube $[0,d]^{n+1}\subset \R^{n+1}$ by the hyperplane $\{\sum_{i=0}^{n+1} x_i = (p+1)d\}$ is a polynomial in $d$ of degree $n$ whose leading coefficient is equal to the $n$-volume of the section of the cube $[0,1]^{n+1}\subset \R^{n+1}$ by the hyperplane $\{\sum_{i=0}^{n+1} x_i = p+1\}$ normalized by the lattice volume of $\{\sum_{i=0}^{n+1} x_i = p+1\}$ (which is $\sqrt{n+1}$). In other words, we have
$$a_p^n=\frac{1}{\sqrt{n+1}}Vol_n\left([0,1]^{n+1}\cap  \left\{\sum_{i=0}^{n+1} x_i = p+1\right\}\right).$$
Interesting questions can be asked about the volumes of high dimensional polytopes obtained in similar ways (see for example \cite{PolytopeVolume}).

Consider also that the $(n,p)$-th Eulerian number $E(n,p)$ (which is equal to the number of permutations of the set $\{1,\ldots,n\}$ in which exactly $p$ elements are greater than the previous element) admits the explicit expression
\begin{equation}\label{FormulaConstructionsEulerianNumber}
E(n,p)=  \sum_{i=0}^{p}(-1)^i\binom{n+1}{i} (p+1-i)^n.
\end{equation}
Hence we have $a_p^n=\frac{1}{n!}E(n,p)$ for $n\geq1$ and $p\geq 0$.

A function $p:\Z \longrightarrow \R_{\geq 0}$ is called \textit{log-concave} if $p(m)^2\geq p(m-1)p(m+1)$ for all $m\in \Z$, and if its support is a contiguous interval, \textit{i.e.} if there exist $m_1,m_2\in \Z$ such that $m_1<m_2$, $p(m) = 0 $ for all $m\leq m_1$ and all $m\geq m_2$, and $p(m)>0$ for all $m_1<m<m_2$. The second condition is sometimes omitted.
As is shown, for example, in \cite{EulerianNumbersLogConcave} (Section 6.5), the sequence of Eulerian numbers $\{E(n,p)\}_{p\in \Z}$ (for a given $n\geq 1$) is symmetric in $\frac{n-1}{2}$ and log-concave; moreover, it is $0$ for $p<0$, strictly increasing from $p=0$ to $p=\frac{n-1}{2} $ and strictly decreasing from $p=\frac{n-1}{2} $ to $p=n-1$ for $n$ odd (respectively, strictly increasing from $p=0$ to $p=\frac{n}{2}-1 $, strictly decreasing from $p=\frac{n}{2} $ to $p=n-1$, and $ E(n,\frac{n}{2}-1)=E(n,\frac{n}{2})$ for $n$ even), and $0$ for $p>n-1$. This naturally implies the corresponding statements for the sequence $\{a^n_p\}_{p\in \Z}$.

Log-concavity is an interesting notion, though we make no use of it here; a survey of the properties of log-concave functions and sequences (some of which are related to algebraic geometry) can be found in \cite{SurveyLogConcave}.

We want to consider the second order central finite differences of the coefficients $a^n_p$, \textit{i.e.} the sequence
$$ D^2a^n_p:=a^n_{p+1} -2a^n_{p} +a^n_{p-1}, $$
as it proves to be a useful notion in Subsection \ref{SubsectionConstructionsFirstConstruction}.

Eulerian numbers are known to satisfy the recursive relation
\begin{equation*}
 E(n,p)=(n-p)E(n-1,p-1) + (p+1)E(n-1,p)  
\end{equation*}
for all $n\geq1$ and $p\in \Z$.

Hence we can see that 
\begin{align*}
&n!D^2a^n_p= E(n,p+1)-2E(n,p)+E(n,p-1) = \\ &((n-p-1)E(n-1,p)+(p+2)E(n-1,p+1))\\
&-2((n-p)E(n-1,p-1)+ (p+1)E(n-1,p))\\
&+ ((n-p+1)E(n-1,p-2) + pE(n-1,p-1)) = \\
&(n-p+1)(E(n-1,p)-2E(n-1,p-1)+E(n-1,p-2)) \\
&+(p+2)(E(n-1,p+1)-2E(n-1,p)+E(n-1,p-1))=\\
&(n-p+1)(n-1)!D^2a^{n-1}_{p-1}+(p+2)(n-1)!D^2a^{n-1}_p.
\end{align*}
Therefore the finite difference $D^2a_p^n$ satisfies the following recursive relation:
\begin{align}\label{FormulaConstructionsRecursiveD2}
 D^2a^n_p  =\frac{n-p+1}{n}D^2a^{n-1}_{p-1}+\frac{p+2}{n}D^2a^{n-1}_p . 
\end{align}

Interestingly (though we make no use of that fact), we also see that the finite differences $E(n,p+1)-2E(n,p)+E(n,p-1)$ obey the same recursive relation as the Eulerian numbers themselves, up to a shift in parameter $p\rightarrow p+1$.

We already know that for $n\geq 1$, the sequence $\{D^2a^n_p\}_{p\in\Z}$ is symmetric in $\frac{n-1}{2}$. The following lemma gives us more precise information: 

\begin{lemma}\label{LemmaConstructionsSigneDifferenceFinie}
Let $n\geq 3$. If $n$ is odd (respectively, even), there exists $0\leq \tilde{p}_n< \frac{n-1}{2}$ (respectively, $0\leq \tilde{p}_n< \frac{n}{2}-1$) such that the sequence $\{D^2a^n_p\}_{p\in\Z}$ satisfies:
\begin{enumerate}
    \item $D^2a^n_p =0$ for $p\leq-2$ and $p\geq n+1$.
    \item $D^2a^n_p>0$ for $-1\leq p \leq \tilde{p}_n$ and $ n-1-\tilde{p}_n\leq p \leq n$.
    \item $D^2a^n_p\leq 0 $ for $p\in\{\tilde{p}_n+1, n-2-\tilde{p}_n\}$.
    \item $D^2a^n_p<0$ for $\tilde{p}_n+2\leq p \leq n-3-\tilde{p}_n$.
\end{enumerate}
Moreover, there cannot be two consecutive integers $n\geq 3$ such that $D^2a^n_{\tilde{p}_n+1}=D^2a^n_{n-2-\tilde{p}_n} = 0$.
\end{lemma}
\begin{proof}
We proceed by induction. The case $n=3$ can be directly computed (we have $a_0^3=a_2^3=\frac{1}{6}$ and $a_1^3=\frac{2}{3}$).

Suppose the statement true for $n-1$, and express $D^2a^n_p$ as $\frac{n-p+1}{n}D^2a^{n-1}_{p-1}+\frac{p+2}{n}D^2a^{n-1}_p $ using Formula (\ref{FormulaConstructionsRecursiveD2}).
By symmetry, we only need to consider $p\leq \frac{n-1}{2}$.
Condition 1 is clearly satisfied.

If $D^2a^{n-1}_{\tilde{p}_{n-1}+1}= 0 $, we set $\tilde{p}_n =\tilde{p}_{n-1}+1$. Then we have $D^2a^n_p>0$  for $-1\leq p \leq \tilde{p}_n$ and $ n-1-\tilde{p}_n\leq p \leq n$, and $D^2a^n_p<0$ for $\tilde{p}_n+1<p<n-2-\tilde{p}_n$. Note in particular that $D^2a^n_{\tilde{p}_n+1}=D^2a^n_{n-2-\tilde{p}_n} \neq 0$.

\begin{figure}
\begin{center}
\includegraphics[scale=1]{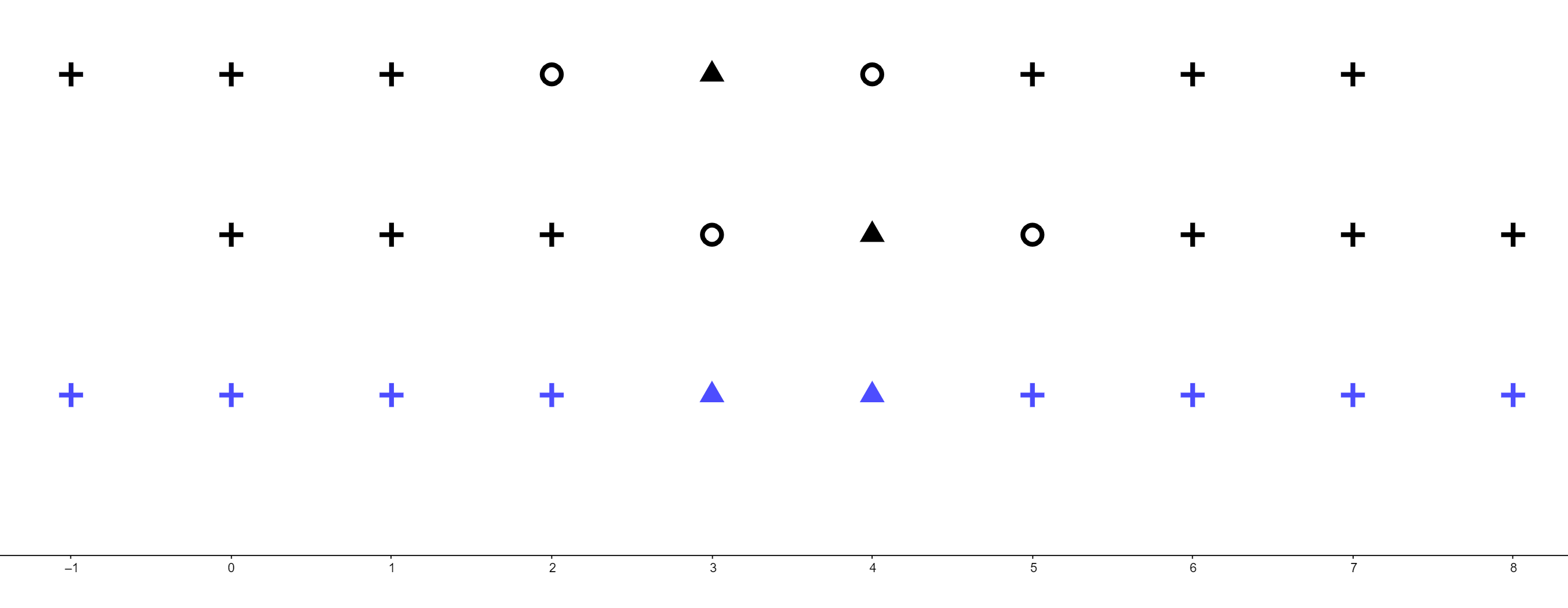}
\includegraphics[scale=1]{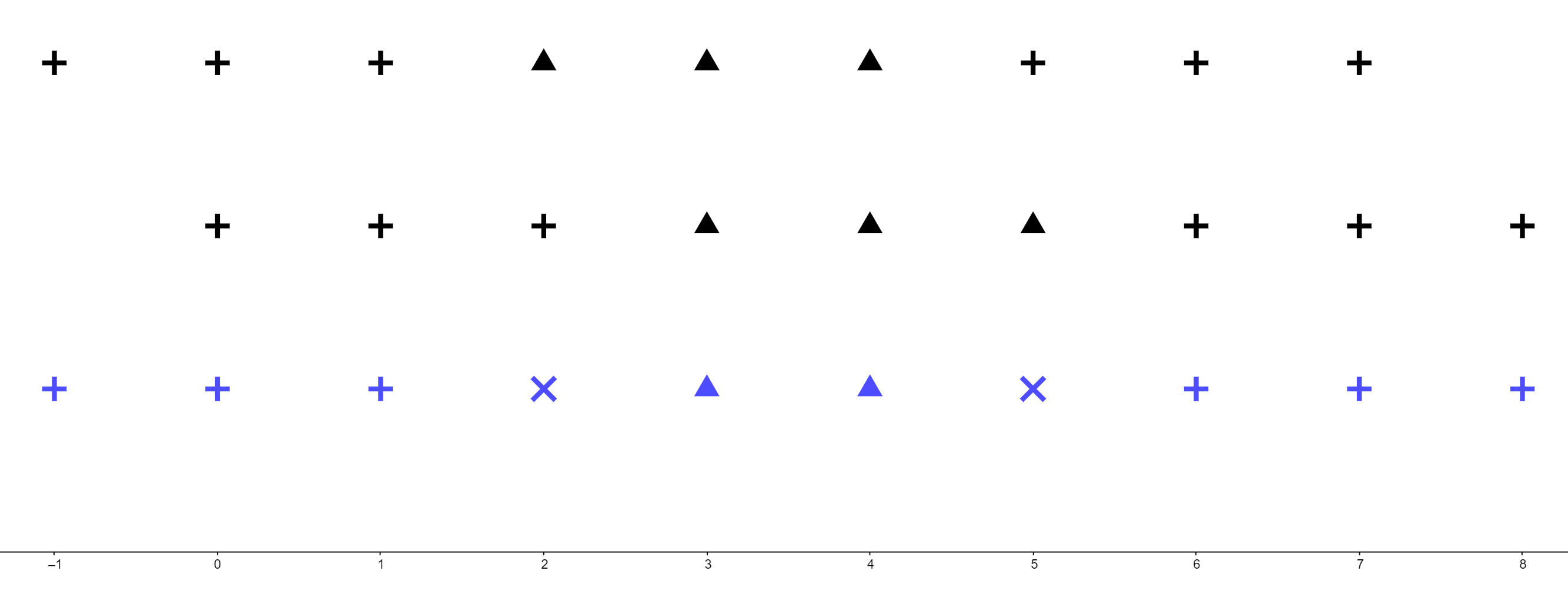}
\end{center}
\caption{In this figure, the $+$ signs correspond to a strictly positive sign, triangles to a strictly negative sign, the circles to $0$ and $\times$ to any of the three options. The first line represents the signs of the coefficients $\{\frac{n-p+1}{n}D^2a^{n-1}_p\}_{p\in\Z}$, the second line those of the coefficients $\{\frac{p+2}{n}D^2a^{n-1}_p\}_{p\in\Z}$ and the third line those of their sums, the coefficients $\{D^2a^n_p\}_{p\in\Z}$, in the case where $D^2a^{n-1}_{\tilde{p}_{n-1}+1}= 0$.
The three lines below depict the case where $D^2a^{n-1}_{\tilde{p}_{n-1}+1}< 0$.}
\label{FigureConstructionsSignesDifferentielle}
\end{figure}

If $D^2a^{n-1}_{\tilde{p}_{n-1}+1}< 0 $, set $\tilde{p}_n =\tilde{p}_{n-1}+1$ if $D^2a^{n}_{\tilde{p}_{n-1}+1} = \frac{n-\tilde{p}_{n-1}}{n}D^2a^{n-1}_{\tilde{p}_{n-1}}+\frac{\tilde{p}_{n-1}+3}{n}D^2a^{n-1}_{\tilde{p}_{n-1}+1}>0 $ and set $\tilde{p}_n =\tilde{p}_{n-1}$ if $D^2a^{n}_{\tilde{p}_{n-1}+1}\leq 0 $. Then we have $D^2a^n_p>0$  for $-1\leq p \leq \tilde{p}_n$, $D^2a^n_p\leq0$ for $p=\tilde{p}_n+1$, and $D^2a^n_p<0$ for $\tilde{p}_n+2\leq p \leq n-3-\tilde{p}_n$. This is illustrated in  Figure \ref{FigureConstructionsSignesDifferentielle}, which helps visualize how the signs of $\{D^2a^n_p\}_{p\in\Z}$ are computed from those of $\{D^2a^{n-1}\}_{p\in\Z}$ (the figure is meant to be a schematic representation, and does not necessarily correspond to any given $n$).

In both cases, we have $\tilde{p}_n\geq \tilde{p}_{n-1}\geq 0$, and $\tilde{p}_n$ is necessarily strictly smaller than $\frac{n}{2}-1$, as $D^2a^n_{\tilde{p}_n}>0$ and $D^2a^n_{\frac{n}{2}-1}<0$ if $n$ is even, as we know that $a^n_{\frac{n}{2}-2}<a^n_{\frac{n}{2}-1}=a^n_{\frac{n}{2}}>a^n_{\frac{n}{2}+1}$ (respectively, $D^2a^n_{\frac{n-1}{2}}<0$ if $n$ is odd, as we know that $a^n_{\frac{n-1}{2}-1}<a^n_{\frac{n-1}{2}}>a^n_{\frac{n-1}{2}+1}$). 

\end{proof}

\begin{remark}
Computations carried out on a computer suggest that $D^2a^n_{p}$ is in fact never $0$ between $-1$ and $n$; it admits two (symmetric) global maxima, and one global minimum in $p=\frac{n-1}{2}$ if $n$ is odd (respectively, two global minima in $p=\frac{n}{2}-1, \frac{n}{2}$ if $n$ is even).

It should be possible to prove this directly, though it does not appear to be a direct consequence of the log-concavity of the sequence.
\end{remark}


\begin{remark}\label{RemarkConstructionsAsymptoticsSecondDifferential}
We can describe more precisely the asymptotic behaviour of $D^2a^{n}_{p}$, though only its sign matters to us. Indeed, for any $s\in \R$, we have 
\begin{equation}
D^2a^{n}_{\left \lfloor {\frac{n-1}{2} + s\sqrt{n}} \right \rfloor}= \sqrt{\frac{6}{\pi}}\frac{12}{(n+1)^{\frac{3}{2}}} \exp\left(-6s^2\right)(12s^2-1) + \mathcal{O}\left({n^{-\frac{5}{2}}}\right).
\end{equation}
This can be proved by writing out  $D^2a^n_p= a^n_{p+1}-2a^n_p+a^n_{p-1}$ and by using a precise enough asymptotic estimate of the coefficients $a^n_p$. Such an estimate can for example be obtained via a small modification of the proof of Theorem 3.1 from \cite{BSplines}.
\end{remark}

Before ending this subsection, we formulate one last recursive relation related to the coefficients $a^n_p$ (for $n\geq 1$ and $p=0,\ldots,n-1 $), which we need later on. It can be immediately deduced from Itenberg's and Viro's work in \cite{IV} that we have 
\begin{equation}\label{FormulaConstructionsRelationRecursive}
a^n_p   =   \frac{1}{n}\left(a_p^{n-1}+a_{p-1}^{n-1} +\sum_{k=1}^{n-2}\sum_{j=0}^{p-1} a_j^k \cdot a_{p-1-j}^{n-1-k}\right) .
\end{equation}
Observe that this means that if every family of polynomials $\{P^k_d\}_{d\in\NN}$ in the Cooking Theorem \ref{TheoremConstructionsMainTheoremConstructions} is asymptotically standard, \textit{i.e.}  $b_i(V_{\R\PP ^k}(P^k_d))  \overset{k}{=} a_i^k\cdot d^k$, then so is $\{Q^n_d\}_{d\in\NN}$ (the asymptotic inequality in Formula (\ref{FormulaConstructionsMainFormula}) must be an equality because of the Smith-Thom inequality). This is the case considered in \cite{IV}.

In Theorem \ref{TheoremConstructionsBonneAsymptotique} below, we compute the asymptotics of various families of coefficients obeying the same recursive relation (\ref{FormulaConstructionsRelationRecursive}) as the coefficients $a^n_p$ for $n$ large enough, but with different initial parameters.

\subsection{Notations and known results}\label{SubsectionConstructionsNotationsAndKnownResults}
In this subsection, we define some notations, prove a useful lemma, and quote the results from Brugall\'e, Itenberg and Viro that will provide us with the main ingredients for our constructions.

When considering in what follows the asymptotic Betti numbers of the projective hypersurfaces associated to a family of completely nondegenerate real Laurent polynomials $\{Q^n_d\}_{d\in \NN}$, it is slightly more convenient to use the following notation: 
if $b_i(V_{\R\PP ^n}(Q^n_d))  \overset{n}{\geq} x_i^n\cdot d^n$ for some $x_i^n \geq 0$, we instead write $b_i(V_{\R\PP ^n}(Q^n_d))  \overset{n}{\geq} (a_i^n + t_i^n)\cdot d^n$, where $t_i^n=x_i^n-a_i ^n$.

If we rewrite the statement of Theorem \ref{TheoremConstructionsMainTheoremConstructions} with that convention, we get that for families $\{P^k_{d}\}_{d\in \NN}$ of completely nondegenerate real Laurent polynomials in $k$ variables (for $k=1,\ldots,n-1$), such that $P^k_d$ is of degree $d$ and that the Newton polytope $\Delta(P^k_d)$ of $P^k_d$ is $S^k_d$ and such that for $i=0,\ldots,k-1$,
 $$b_i(V_{\R\PP ^k}(P^k_d))\overset{k}{\geq} (a_i^k + t_i^k) \cdot d^k$$
for some $t_i^k \in \R$, there exists a family $\{Q^n_d\}_{d\in\NN}$ of completely nondegenerate real Laurent polynomials in $n$ variables such that $\Delta(Q^n_d)=S^n_d$ and such that for $i=0,\ldots,n-1$ 
\begin{equation}\label{FormulaConstructionsSommeResultante}
b_i(V_{\R\PP ^n}(Q^n_d))   \overset{n}{\geq}  \left( a^n_i + \frac{1}{n}\left(t_i^{n-1}+t_{i-1}^{n-1} +\sum_{k=1}^{n-2}\sum_{j=0}^{i-1} 2(t_j^k \cdot a_{i-1-j}^{n-1-k}) + t_j^k \cdot t_{i-1-j}^{n-1-k}\right)  \right) \cdot d^n,
\end{equation}
where $t_j^k$ is set to be $0$ for $j\not \in \{0,\ldots,k-1\}$.

Remember that a family of real smooth algebraic projective hypersurfaces $\{X^n_d\}_{d\in \NN}$ is aymptotically maximal if $b_*(\R X^n_d) \overset{n}{=} b_*(\C X^n_d) $. In particular, if $b_i(\R X^n_d)   \overset{n}{\geq}  \left( a^n_i + t^n_i  \right) \cdot d^n$ and $\sum_{i=0}^{n-1} t ^n_i =0$, we have $b_*(\R X^n_d)   \overset{n}{\geq} \sum_{i=0}^{n-1} \left( a^n_i + t^n_i  \right) \cdot d^n = d
^n \overset{n}{=} b_*(\C X^n_d) $, hence  $b_*(\R X^n_d) \overset{n}{=} b_*(\C X^n_d) $ (because of the Smith-Thom inequality) and $\{X^n_d\}_{d\in \NN}$ is asymptotically maximal.

We can now prove the following lemma regarding asymptotic maximality :
\begin{lemma}\label{LemmaConstructionsAsmptoticallyMaximalPreserve}
Let $n\geq 2$. For $k=1,\ldots,n-1$, let $\{P^k_{d}\}_{d\in \NN}$ be a family of completely nondegenerate real Laurent polynomials in $k$ variables such that $P^k_d$ is of degree $d$ and that the Newton polytope $\Delta(P^k_d)$ of $P^k_d$ is $S^k_d$ and such that for $i=0,\ldots,k-1$,
 \begin{equation}\label{FormulaConstructionsLemmePreservationAsymptoticite}
     b_i(V_{\R\PP ^k}(P^k_d))\overset{k}{\geq} (a_i^k + t_i^k) \cdot d^k
 \end{equation}
for some $t_i^k \in \R$, as in the Cooking Theorem \ref{TheoremConstructionsMainTheoremConstructions}.
Let $\{Q^n_d\}_{d\in \NN}$ be a family of polynomials cooked with the ingredients $\{P^k_{d}\}_{d\in \NN}$ using the Cooking Theorem.

Suppose additionally that each family $\{P^k_{d}\}_{d\in \NN}$ is such that  $\sum_{i=0}^{k-1} t ^k_i =0$, hence the associated family of projective hypersurfaces $\{V_{\PP ^k}(P^k_{d})\}_{d\in \NN}$ is asymptotically maximal. Then the sum over $i=0,\ldots,n-1$ of the coefficients
$$\frac{1}{n}\left(t_i^{n-1}+t_{i-1}^{n-1} +\sum_{k=1}^{n-2}\sum_{j=0}^{i-1} 2(t_j^k \cdot a_{i-1-j}^{n-1-k}) + t_j^k \cdot t_{i-1-j}^{n-1-k}\right) $$
from Formula \ref{FormulaConstructionsSommeResultante} is also $0$; in particular, the family of hypersurfaces $\{V_{\PP ^k}(Q^n_{d})\}_{d\in \NN}$ associated to $\{Q^n_d\}_{d\in \NN}$ is also asymptotically maximal, and the asymptotic inequality (\ref{FormulaConstructionsLemmePreservationAsymptoticite}) is an asymptotic equality.
\end{lemma}

\begin{proof}
We know that for $i=0,\ldots, n-1$,
\begin{equation*}
b_i(V_{\R\PP ^n}(Q^n_d))   \overset{n}{\geq}  \left( a^n_i + \frac{1}{n}\left(t_i^{n-1}+t_{i-1}^{n-1} +\sum_{k=1}^{n-2}\sum_{j=0}^{i-1} 2(t_j^k \cdot a_{i-1-j}^{n-1-k}) + t_j^k \cdot t_{i-1-j}^{n-1-k}\right)  \right) \cdot d^n,
\end{equation*}
where $t_j^k$ and $a_j^k$ are set to be $0$ for $j\not \in \{0,\ldots,k-1\}$.
Set
$$t^n_i := \frac{1}{n}\left(t_i^{n-1}+t_{i-1}^{n-1} +\sum_{k=1}^{n-2}\sum_{j=0}^{i-1} 2(t_j^k \cdot a_{i-1-j}^{n-1-k}) + t_j^k \cdot t_{i-1-j}^{n-1-k}\right) $$ 
for $i\in \Z$; observe that $t^n_i=0$ if $i\not \in \{0,\ldots,n-1\}$.

Showing that $\sum_{i=0}^{n-1} t ^n_i=\sum_{i\in\Z} t ^n_i =0$ is enough to conclude (using as above the Smith-Thom inequality).

Indeed, we have
\begin{align*}
&\sum_{i=0}^{n-1} t ^n_i =\sum_{i\in\Z}  \frac{1}{n}\left(t_i^{n-1}+t_{i-1}^{n-1} +\sum_{k=1}^{n-2}\sum_{j=0}^{i-1} 2(t_j^k \cdot a_{i-1-j}^{n-1-k}) + t_j^k \cdot t_{i-1-j}^{n-1-k}\right) = \\
& \frac{1}{n}\left(\left(\sum_{i\in\Z}t_i^{n-1}\right)+\left(\sum_{i\in\Z}t_{i-1}^{n-1} \right)+ \sum_{k=1}^{n-2}\sum_{j\in \Z} t_j^k \cdot (2\sum_{i\in\Z} a_{i-1-j}^{n-1-k}+\sum_{i\in\Z}t_{i-1-j}^{n-1-k})\right)=\\
& \frac{1}{n}\left(0+0 + \sum_{k=1}^{n-2}\sum_{j\in \Z} t_j^k \cdot (2+0)\right)=
\frac{2}{n} \sum_{k=1}^{n-2}\sum_{j\in \Z} t_j^k =0.
\end{align*}
\end{proof}

We now quote (using our notations) two results that attest the existence of families of polynomials which we later use as ingredients for the Cooking Theorem to prove Theorem \ref{TheoremConstructionsApplicationThm}.
The first one was proved by Itenberg and Viro in \cite{IV}, and alluded to in Section \ref{SectionIntroduction}.

\begin{thm}[Itenberg, Viro]\label{TheoremConstructionsFamilleAsymptotiqueItenbergViro}
Let $k\geq1$. There exists a family $\{I^k_{d}\}_{d\in \NN}$ of completely nondegenerate real Laurent polynomials in $k$ variables obtained by combinatorial Patchwork such that $I^k_d$ is of degree $d$ and that the Newton polytope $\Delta(I^k_d)$ of $I^k_d$ is $S^k_d$ and such that for $i=0,\ldots,k-1$,
$$b_i(V_{\R\PP ^k}(I^k_d))\overset{k}{=} a_i^k \cdot d^k.$$
In particular, the families $\{V_{\PP ^k}(I^k_d)\}_{d\in \NN}$ are asymptotically maximal.
\end{thm}
The hypersurfaces associated to the polynomials $I^k_d$ are asymptotically standard in the same sense as above. They serve as ``neutral" ingredients in what follows, in that they do not contribute to any difference from the asymptotically standard case.

The second result comes from \cite{Brugalle2006}, where Brugall{\'e} builds two families of completely nondegenerate real Laurent polynomials in $3$ variables such that the associated surfaces in $\PP ^3$ have exceptionally large asymptotic $b_0$ (respectively, $b_1$) using a method from Bihan's \cite{BihanAsymptotic}.

\begin{thm}[Brugall{\'e}]\label{TheoremConstructionsFamilleAsymptotiqueBrugalle}
There exist families $\{B^+_{d}\}_{d\in \NN}$ and $\{B^-_{d}\}_{d\in \NN}$ of completely nondegenerate real Laurent polynomials in $3$ variables such that $B^\pm_d$ is of degree $d$ and that the Newton polytope $\Delta(B^\pm_d)$ of $B^\pm_d$ is $S^3_d$ and such that
$$b_0(V_{\R\PP ^3}(B^+_d))\overset{3}{=} \frac{3}{8} \cdot d^3 =\left(\frac{1}{6}+ \frac{5}{24}\right) \cdot d^3 =\left(a^3_0+ \frac{5}{24}\right) \cdot d^3,$$
$$b_1(V_{\R\PP ^3}(B^+_d))\overset{3}{=} \frac{1}{4} \cdot d^3 =\left(\frac{2}{3}- \frac{5}{12}\right) \cdot d^3 =\left(a^3_1 -\frac{5}{12}\right) \cdot d^3,$$
$$b_0(V_{\R\PP ^3}(B^-_d))\overset{3}{=} \frac{1}{8} \cdot d^3 =\left(\frac{1}{6}- \frac{1}{24}\right) \cdot d^3 =\left(a^3_0- \frac{1}{24}\right) \cdot d^3$$
and
$$b_1(V_{\R\PP ^3}(B^-_d))\overset{3}{=} \frac{3}{4} \cdot d^3 =\left(\frac{2}{3}+ \frac{1}{12}\right) \cdot d^3 =\left(a^3_1+ \frac{1}{12}\right) \cdot d^3.$$
In particular, the families $\{V_{\PP ^3}(B^\pm _d)\}_{d\in \NN}$ are asymptotically maximal.
\end{thm}
\begin{remark}
Of course, Poincar{\'e} duality applies, as homology is considered with coefficients in $\Z_2$.
\end{remark}
\begin{remark}
As far as the author is aware, these are the largest asymptotic values for each respective Betti numbers of a smooth real projective algebraic surface to have been obtained to this day, which is why we choose to use them as ingredients in what follows.
\end{remark}
\begin{remark}\label{RemarkConstructionsBrugalleIntermediaire}
It is not particularly hard, though somewhat tedious, to show that for any $a\in [-\frac{1}{24},\frac{5}{24}]$, we can build (using the families $\{B^\pm_d\}_{d\in \NN}$) a family $\{P^3_d\}_{d\in \NN}$ of completely nondegenerate real Laurent polynomials in $3$ variables such that the Newton polytope $\Delta(P^{3}_{d})$ is $S^3_d$ and that for $i=0,1,2$, we have
$$b_i(V_{\R\PP ^n}(P^3_{d}))\overset{n}{=} x^3_i\cdot d^n$$
with $x^3_0=x^3_2 = \frac{1}{6}+a$ and $x_1^3=\frac{4}{6}-2a$.
The idea is to partition $S^3_d$ (for very large degrees $d$) into smaller, albeit still very large, simplices corresponding either to $B^+_{\tilde{d}}$ or to $B^-_{\tilde{d}}$, for some $\tilde{d}<d$ (with some interstitial space of asymptotically negligible volume).
The proportion $\lambda_d\in [0,1]$ (respectively, $1-\lambda_d \in [0,1]$) of the total volume of $S^3_d$ filled by simplices corresponding to $B^+_{\tilde{d}}$ (respectively, to $B^-_{\tilde{d}}$) must be such that $\lambda_d \frac{5}{24} - (1-\lambda_d) \frac{1}{24}$ converges to $a$ as $d \rightarrow \infty$.

\end{remark}

\subsection{The first construction}\label{SubsectionConstructionsFirstConstruction}

In this subsection, we describe the first of our two main families of constructions. It allows us to find for every dimension and index Betti numbers that are asymptotically (in $d$) strictly superior to the standard case, but not by a large margin: this enables us to prove Theorem \ref{TheoremConstructionsApplicationThm}.
The other one is described in Subsection \ref{SubsectionConstructionsSecondConstruction} and provides much better asymptotic (in $d$) lower bounds on the chosen Betti numbers, but only asymptotically (in $n$); it allows us to prove Theorem \ref{TheoremConstructionsBonneAsymptotique}.

The idea is to carefully pick families of polynomials $\{P^k_{d}\}_{d\in \NN}$ in $k$ variables (with $k$ small) as ingredients for the Cooking Theorem \ref{TheoremConstructionsMainTheoremConstructions} in order to get families of polynomials $\{Q^n_{d}\}_{d\in \NN}$ (with $n$ large) with interesting properties.

Given $n\geq 3$ and $i=0,\ldots,n-1$, the polynomials $\{P^k_{d}\}_{d\in \NN}$ (for $k=1,\ldots,n-1$) must be chosen so that the asymptotic Betti numbers of the associated families of projective hypersurfaces are such that the right-hand term of Formula (\ref{FormulaConstructionsSommeResultante}) is large.

As far as the author knows, few interesting (in that regard) families have yet been constructed in high ambient dimension. Hence, we must work our way up from dimension $3$, where we have Brugall{\'e}'s Theorem \ref{TheoremConstructionsFamilleAsymptotiqueBrugalle}, by recursively applying the Cooking Theorem: a family $\{Q^n_{d}\}_{d\in \NN}$ that we get as the result of one application of the theorem can serve as an ingredient for a construction in higher dimension.

Note that should a new family of hypersurfaces with interesting asymptotic Betti number be developed in a given dimension $n$, we could immediately use it as an ingredient to hopefully get new and interesting results in dimension $\tilde{n}>n$.

Observe also that since there is only one non-trivial Betti number in ambient dimension $1$, and two that are equal in ambient dimension $2$ (and hence both asymptotically smaller than or equal to $\frac{d^2}{2}$), nothing interesting can a priori be expected from the direct use of non-asymptotically standard families in ambient dimension $1$ and $2$.

In general, it is unclear how to choose the ingredients $\{P^k_{d}\}_{d\in \NN}$ so that a given Betti number is maximized in the resulting family of hypersurfaces, as Formula (\ref{FormulaConstructionsSommeResultante}) is fairly complicated; the trick that we use in the first construction, which is described in the proof of the following lemma, is to make it so that most terms in the formula are trivial, so that we can understand it better.
The results are most likely, in a sense, suboptimal, but they suffice for our purpose here.

\begin{lemma}\label{LemmaConstructionsMainConstruction}
For each $n\geq 8$, there exist families $\{H^{+,n}_{d}\}_{d\in \NN}$ and $\{H^{-,n}_{d}\}_{d\in \NN}$ of completely nondegenerate real Laurent polynomials in $n$ variables such that $H^{\pm,n}_{d}$ is of degree $d$ and that the Newton polytope $\Delta(H^{\pm,n}_{d})$ is $S^n_d$ and such that for $i=0,\ldots,n-1$, we have
$$b_i(V_{\R\PP ^n}(H^{+,n}_{d}))\overset{n}{=} \left( a^n_i + \frac{2}{n}\frac{5}{24}D^2a^{n-4}_{i-2} \right)\cdot d^n$$
and
$$b_i(V_{\R\PP ^n}(H^{-,n}_{d}))\overset{n}{=} \left( a^n_i - \frac{2}{n}\frac{1}{24}D^2a^{n-4}_{i-2} \right) \cdot d^n$$
Moreover, the family of hypersurfaces associated to each family is asymptotically maximal.
\end{lemma}
\begin{remark}
The coefficients $\frac{1}{24}$ and $\frac{5}{24}$ come from Brugall{\'e}'s Theorem \ref{TheoremConstructionsFamilleAsymptotiqueBrugalle}; should we get better asymptotic results than in Theorem \ref{TheoremConstructionsFamilleAsymptotiqueBrugalle}, we would be able to immediately ``plug" them in and improve the asymptotics of Lemma \ref{LemmaConstructionsMainConstruction}.

\end{remark}

\begin{proof}
We define $\{H^{+,n}_{d}\}_{d\in \NN}$ first.

We apply the Cooking Theorem to the following ingredients $\{P^k_{d}\}_{d\in \NN}$ : let $P^k_d$ be equal to $I^k_d$ from Itenberg and Viro's Theorem \ref{TheoremConstructionsFamilleAsymptotiqueItenbergViro} for $k\in \{1,2,4,5,\ldots,n-1\}$, and let $P^3_d$ be equal to $B^+_d$ from Brugall{\'e}'s Theorem \ref{TheoremConstructionsFamilleAsymptotiqueBrugalle}.
Following the notations introduced with Formula (\ref{FormulaConstructionsSommeResultante}), we get that $t_0^3=t_2^3=\frac{5}{24}$ and $t_1^3=-\frac{5}{12} $, and $t_j^k=0$ for all other $k,j$ (by definition of the polynomials  $P^k_d$). Note in particular that as $ n>7$, $t^{n-1}_j=0$ and $t_j^k \cdot t_{i-1-j}^{n-1-k}=0$ for all $k,j$.

Hence application of the Cooking Theorem \ref{TheoremConstructionsMainTheoremConstructions} yields a family $\{H^{+,n}_{d}\}_{d\in \NN}$ such that
\begin{gather*}
b_i(V_{\R\PP ^n}(H^{+,n}_{d}))   \overset{n}{\geq}  \left( a^n_i + \frac{1}{n}\left(t_i^{n-1}+t_{i-1}^{n-1} +\sum_{k=1}^{n-2}\sum_{j=0}^{i-1} 2(t_j^k \cdot a_{i-1-j}^{n-1-k}) + t_j^k \cdot t_{i-1-j}^{n-1-k}\right)  \right) \cdot d^n =\\
\left( a^n_i + \frac{1}{n}\left(\sum_{j=0}^{i-1} 2(t_j^3 \cdot a_{i-1-j}^{n-4}) \right)  \right) \cdot d^n=  \left( a^n_i + \frac{2}{n}\frac{5}{24}\left( a_{i-1}^{n-4} -2 a_{i-2}^{n-4} + a_{i-3}^{n-4} \right)  \right) \cdot d^n=\\
 \left( a^n_i + \frac{2}{n}\frac{5}{24}D^2a^{n-4}_{i-2}  \right) \cdot d^n
\end{gather*}
for $i=0,\ldots,n-1$.

As the sum $\sum _{j\in \Z} D^2a^m_j $ of second order finite differences is $0$ for any $m\geq 1$, and in particular for $m=n-4$, we have as above that the family is asymptotically maximal because of the Smith-Thom inequality, and that the asymptotic inequality $b_i(V_{\R\PP ^n}(H^{+,n}_{d}))   \overset{n}{\geq} \left( a^n_i + \frac{2}{n}\frac{5}{24}D^2a^{n-4}_{i-2}  \right) \cdot d^n$ is in fact an asymptotic equality $b_i(V_{\R\PP ^n}(H^{+,n}_{d}))   \overset{n}{=} \left( a^n_i + \frac{2}{n}\frac{5}{24}D^2a^{n-4}_{i-2}  \right) \cdot d^n$.

The exact same proof, with $B^-_d$ replacing $B^+_d$ and $-\frac{1}{24}$ replacing $\frac{5}{24}$, yields the other case.
\end{proof}

As will become apparent in the proof of Theorem \ref{TheoremConstructionsApplicationThm}, we also need the following lemma.

\begin{lemma}\label{LemmaConstructionsConstructionBords}
For any $n\geq 3$, there exists a family $\{L^n_d\}_{d\in\NN}$ of completely nondegenerate real Laurent polynomials in $n$ variables such that $\Delta(L^n_d)=S^n_d$, that the associated family of real projective hypersurfaces is asymptotically maximal and that 
\begin{equation*}
b_0(V_{\R\PP^n }(L^n_d)) = b_{n-1}(V_{\R\PP^n }(L^n_d)) \overset{n}{\geq}  \left(a^n_0 + \frac{5}{24}\frac{3!}{n!}\right) \cdot d^n.   
\end{equation*}
\end{lemma}

\begin{proof}

We proceed by induction on $n$.
The case $n=3$ is simply Brugall{\'e}'s Theorem \ref{TheoremConstructionsFamilleAsymptotiqueBrugalle}.

Now let $n\geq 4$ and suppose that $\{L^{n-1}_d\}_{d\in\NN}$ has been defined.
We apply the Cooking Theorem \ref{TheoremConstructionsMainTheoremConstructions} to the following families of polynomials $\{P^k_{d}\}_{d\in \NN}$ : let $P^k_d$ be equal to $I^k_d$ from Itenberg and Viro's Theorem \ref{TheoremConstructionsFamilleAsymptotiqueItenbergViro} for $k\in \{1,\ldots,n-2\}$, and let $P^{n-1}_d$ be equal to $L^{n-1}_d$.
Following the notations introduced with Formula (\ref{FormulaConstructionsSommeResultante}), we have $t_0^{n-1}\geq\frac{5}{24}\frac{3!}{(n-1)!} $, and $t_j^k=0$ for all other $k<n-1$ (by definition of the polynomials  $P^k_d$).

Hence application of the Theorem \ref{TheoremConstructionsMainTheoremConstructions} yields a family $\{L^n_{d}\}_{d\in \NN}$ such that
\begin{align*}
&b_0(V_{\R\PP ^n}(L^n_d))   \overset{n}{\geq}  \left( a^n_0 + \frac{1}{n}\left(t_0^{n-1}+t_{0-1}^{n-1} +\sum_{k=1}^{n-2}\sum_{j=0}^{0-1} 2(t_j^k \cdot a_{0-1-j}^{n-1-k}) + t_j^k \cdot t_{0-1-j}^{n-1-k}\right)  \right) \cdot d^n \overset{n}{\geq}\\
& \left(a^n_0 + \frac{5}{24}\frac{3!}{n!}\right) \cdot d^n,
\end{align*}
as wanted.
\end{proof}

We can finally prove Theorem \ref{TheoremConstructionsApplicationThm}, which we state again.
\newtheorem*{thmTricheConstructions3}{Theorem \ref{TheoremConstructionsApplicationThm}}

\begin{thmTricheConstructions3}
For any $n\geq 3$ and any $i=0,\ldots,n-1$, there exists  $b^n_i>a^n_i$ and a family $\{Q^n_d\}_{d\in\NN}$ of completely nondegenerate real Laurent polynomials in $n$ variables such that $\Delta(Q^n_d)=S^n_d$, that the associated family of real projective hypersurfaces is asymptotically maximal and that 
\begin{equation*}
b_i(V_{\R\PP^n }(Q^n_d))   \overset{n}{\geq}  b^n_i \cdot d^n.   
\end{equation*}
\end{thmTricheConstructions3}

\begin{proof}
Let $n\geq 3$ and $i\in \{0,\ldots,n-1\}$.

We mainly rely on the construction from Lemma \ref{LemmaConstructionsMainConstruction} and the result of Lemma \ref{LemmaConstructionsSigneDifferenceFinie}.
However, as $D^2a^{n-4}_{p-2} =0$ for $p=0,n-1$ (the construction from Lemma \ref{LemmaConstructionsMainConstruction} cannot help us get a large number of connected components), we also need Lemma \ref{LemmaConstructionsConstructionBords}; moreover, as we were not able to show in Lemma \ref{LemmaConstructionsSigneDifferenceFinie} that the finite differential $D^2a^{n}_{p}$ is never $0$ between $k=-1$ and $k=n$, we need another ad hoc trick.

We assume that $n\geq 8$; the cases $n\leq 7$ are treated explicitly and in more details in Section \ref{SectionExplicitComputations}.

If $i\in \{0,n-1\}$, we define $Q^n_d$ as $L^n_d$ from Lemma \ref{LemmaConstructionsConstructionBords}, and set $b_i^n:=a^n_i + \frac{5}{24}\frac{3!}{n!}$ (remember that Poincar{\'e} duality applies) - this suffices.

Otherwise, consider $D^2a^{n-4}_{i-2}$. If it is strictly positive, define $Q^n_d$ as $H^{+,n}_d$ from Lemma \ref{LemmaConstructionsMainConstruction} and $b_i^n :=a^n_i + \frac{2}{n}\frac{5}{24}D^2a^{n-4}_{i-2} $; if it strictly negative, define $Q^n_d$ as $H^{-,n}_d$ from the same lemma and $b_i^n:=a^n_i - \frac{2}{n}\frac{1}{24}D^2a^{n-4}_{i-2}$. In both cases, we are done (using the statement of the lemma).

If we are unlucky, and $i$ is the only index in $\{1,\ldots,n-2\} $ such that $D^2a^{n-4}_{i-2}=0$, we know from Lemma \ref{LemmaConstructionsSigneDifferenceFinie} that $D^2a^{n-5}_{p}$ is never $0$ for $p\in \{-1,\ldots,n-5\}$. Moreover, Formula \ref{FormulaConstructionsRecursiveD2} tells us that 
$$D^2a^{n-4}_{i-2} =\frac{n-i-1}{n-4}D^2a^{n-5}_{i-3}+\frac{i}{n-4}D^2a^{n-5}_{i-2}.  $$
As at least one of the two terms is nonzero, both must be for $D^2a^{n-4}_{i-2}$ to be $0$.
Observe also that $\frac{n-i-1}{n-4}\neq \frac{i}{n-4}$, as otherwise $i=\frac{n-1}{2}$, in which case $D^2a^{n-4}_{i-2}=D^2a^{n-4}_{\frac{((n-4)-1)}{2}}\neq 0$ (the middle term is never $0$, see Lemma \ref{LemmaConstructionsSigneDifferenceFinie}). Hence $D^2a^{n-5}_{i-3}+D^2a^{n-5}_{i-2} \neq 0 $.

Now apply  the Cooking Theorem \ref{TheoremConstructionsMainTheoremConstructions} to the following ingredients $\{P^k_{d}\}_{d\in \NN}$ : let $P^k_d$ be equal to $I^k_d$ from Itenberg and Viro's Theorem \ref{TheoremConstructionsFamilleAsymptotiqueItenbergViro} for $k\in \{1,\ldots,n-2\}$, and let $P^{n-1}_d$ be equal to $H^{+,n-1}_d$ from Lemma \ref{LemmaConstructionsMainConstruction} if $D^2a^{n-5}_{i-3}+D^2a^{n-5}_{i-2} > 0 $ (respectively, equal to $H^{-,n-1}_d$ from Lemma \ref{LemmaConstructionsMainConstruction} if $D^2a^{n-5}_{i-3}+D^2a^{n-5}_{i-2} < 0 $).
Following the notations introduced with Formula (\ref{FormulaConstructionsSommeResultante}), we have $t_i^{n-1}= \frac{2}{n-1}\frac{5}{24}D^2 a^{n-5}_{i-2} $ and $t_{i-1}^{n-1}= \frac{2}{n-1}\frac{5}{24}D^2 a^{n-5}_{i-3} $ (respectively, $t_i^{n-1}= -\frac{2}{n-1}\frac{1}{24}D^2 a^{n-5}_{i-2} $ and $t_{i-1}^{n-1}= -\frac{2}{n-1}\frac{1}{24}D^2 a^{n-5}_{i-3} $), and $t_j^k=0$ for all other $k<n-1$ (by definition of the polynomials  $P^k_d$).

Hence application of the Cooking Theorem yields a family $\{Q^n_{d}\}_{d\in \NN}$ such that
\begin{align*}
&b_i(V_{\R\PP ^n}(Q^n_d))   \overset{n}{\geq}  \left( a^n_i + \frac{1}{n}\left(t_i^{n-1}+t_{i-1}^{n-1} +\sum_{k=1}^{n-2}\sum_{j=0}^{i-1} 2(t_j^k \cdot a_{i-1-j}^{n-1-k}) + t_j^k \cdot t_{i-1-j}^{n-1-k}\right)  \right) \cdot d^n =\\
& \left(a^n_i + \frac{1}{n}\frac{5}{24}\frac{2}{n-1}(D^2a^{n-5}_{i-2}+D^2a^{n-5}_{i-3})\right) \cdot d^n
\end{align*}
(respectively, $b_i(V_{\R\PP ^n}(Q^n_d))   \overset{n}{\geq} \left(a^n_i - \frac{1}{n}\frac{1}{24}\frac{2}{n-1}(D^2a^{n-5}_{i-2}+D^2a^{n-5}_{i-3})\right) \cdot d^n  $).

Define $b_i^n := a^n_i + \frac{1}{n}\frac{5}{24}\frac{2}{n-1}(D^2a^{n-5}_{i-2}+D^2a^{n-5}_{i-3})$ (respectively, $ b_i^n :=a^n_i - \frac{1}{n}\frac{1}{24}\frac{2}{n-1}(D^2a^{n-5}_{i-2}+D^2a^{n-5}_{i-3})$). This suffices.

In each case above, the family of hypersurfaces associated to the family $\{Q^n_d\}_{d\in \NN}$ is asymptotically maximal, as the family $\{Q^n_d\}_{d\in \NN}$ always comes from an application of the Cooking Theorem to families of polynomials such that the associated families of hypersurfaces are asymptotically maximal.

\end{proof}

Hence we found in all dimensions $n\geq 3$ and indices $i=0,\ldots,n-1$ Betti numbers that are asymptotically (in $d$) strictly greater than the standard case - however, the asymptotics (in $n$) of that surplus are not very good.

Indeed, the asymptotic behavior of the Eulerian numbers (and hence of the coefficients $a^n_p$) is known:
it was shown by G. Polya in \cite{Polya} (see \cite{SincFormulaEuler} for a more easily accessible reference) that for any $x\in \R$, we have
\begin{equation}\label{FormulaConstructionsPolya}
a^{n}_{\left \lfloor {\frac{n-1}{2} + x\sqrt{n}} \right \rfloor}= \sqrt{\frac{6}{\pi(n+1)}} \exp\left(-6x^2\right) + \mathcal{O}\left({n^{-\frac{3}{2}}}\right)
\end{equation}
(what he considered was actually the volume of hypercube slices, which is equivalent, as seen above).

On the other hand, the surplus that we found in Lemma \ref{LemmaConstructionsMainConstruction} in dimension $n\geq 3$ and index $1\leq i \leq n-2$ relative to the standard case was of the form $C\frac{1}{n}D^2a^{n-4}_{i-2}$, for some constant $C$.
We have already mentioned in Remark \ref{RemarkConstructionsAsymptoticsSecondDifferential} that for a fixed $s\in \R$,
$D^2a^{n}_{\left \lfloor {\frac{n-1}{2} + s\sqrt{n}} \right \rfloor}$ is $\mathcal{O}\left({n^{-\frac{3}{2}}}\right)$.
Hence, the ratio $\frac{b^n_i-a^n_i}{a^n _i}$, where $b^n_i$ comes from Theorem \ref{TheoremConstructionsApplicationThm}, is at most $\mathcal{O}\left({n^{-2}}\right)$. We achieve much better results in the next subsection.

\subsection{The second construction}\label{SubsectionConstructionsSecondConstruction}

In this subsection, we recall the statements of Theorem \ref{TheoremConstructionsBonneAsymptotique} and Theorem \ref{CorollaryConstructionsBonneAsymptotique} and prove them.
They yield good "asymptotic asymptotic" results, in the sense that we find two families of families of polynomials such that the associated Betti numbers are asymptotically of magnitude $c^n_i \cdot d^n$ and $d^n_i\cdot d^n$ respectively (as the degree $d$ goes to infinity), while the Betti numbers of the standard case are of magnitude $a^n_i\cdot d^n$, with $\frac{c^n_i}{a^n_i}$ and $\frac{d^n_i}{a^n_i}$ converging (as the dimension $n$ goes to infinity) to strictly positive constants.

To do so, we apply recursively the Cooking Theorem \ref{TheoremConstructionsMainTheoremConstructions} to the ingredients provided by Itenberg and Viro's Theorem \ref{TheoremConstructionsFamilleAsymptotiqueItenbergViro} and Brugall{\'e}'s Theorem \ref{TheoremConstructionsFamilleAsymptotiqueBrugalle}.
The construction itself is simple; the main difficulty lies in understanding the asymptotic behavior of the sequences recursively defined using Formula \ref{FormulaConstructionsMainFormula} that describe the asymptotic behavior of the Betti numbers, as we cannot expect most terms to be trivial, unlike in the proof of Theorem \ref{TheoremConstructionsApplicationThm}.
We succeed by applying probabilistic methods.

\newtheorem*{thmTricheConstructions4}{Theorem \ref{TheoremConstructionsBonneAsymptotique}}
\begin{thmTricheConstructions4}

Let $N\geq 1$. For $k=1,\ldots,N$, let $\{P^k_{d}\}_{d\in \NN}$ be a family of completely nondegenerate real Laurent polynomials in $k$ variables such  that the Newton polytope $\Delta(P^k_d)$ of $P^k_d$ is $S^k_d$.
Suppose additionally that for $k=1,\ldots,N$ and $i=0,\ldots,k-1$,
 $$b_i(V_{\R\PP ^k}(P^k_d))\overset{k}{=} x_i^k \cdot d^k$$
for some $x_i^k \in \R_{\geq 0}$ such that $\sum_{i=0}^{k-1}x_i^k =1$ (in particular, the family of projective hypersurfaces associated to each family $\{P^k_d\}_{d\in \NN}$ is asymptotically maximal). Set also $x_i^k$  to be $0$ for $i\not \in \{0,\ldots,k-1\}$.

Define
$$\sigma ^2 :=\frac{2}{(N+1)(N+2)} \left(\frac{1}{4} +\sum_{k=1}^N\sum_{i=0}^{k-1}x^k_i\left(i- \frac{k-1}{2}\right)^2 \right). $$

Then for every $n\geq N+1$ and any $i\in \Z$, there exist $x_i^n\in \R_{\geq 0}$  and a family $\{Q^n_d\}_{d\in\NN}$ of completely nondegenerate real Laurent polynomials in $n$ variables such that $\Delta(Q^n_d)=S^n_d$, that for $i\in \Z$ 
\begin{equation*}
b_i(V_{\R\PP ^n}(Q^n_d))   \overset{n}{=}   x_i^n \cdot d^n
\end{equation*}
and such that for any $m\in  \Z$ we have
\begin{equation}\label{FormulaConstructionsBonneAsymptotiquePreuve1}
x^n_m = \frac{1}{\sigma\sqrt{2\pi}}\frac{1}{\sqrt{n}} \exp\left(-\frac{\left(\frac{n-1}{2} - m\right)^2}{2n\sigma^2}\right) +o\left(n^{-\frac{1}{2}}\right),
\end{equation}
where the $o(1)$ error term is uniform in $m$.
The family of projective hypersurfaces associated to each family $\{Q^n_d\}_{d\in \NN}$ is also asymptotically maximal.

\end{thmTricheConstructions4}

\begin{remark}
As the error term is uniform in $m$, Formula (\ref{FormulaConstructionsBonneAsymptotiquePreuve1}) is equivalent to 
\begin{equation*}
x^n_{\left \lfloor {\frac{n-1}{2} + x\sqrt{n}} \right \rfloor} = \frac{1}{\sigma\sqrt{2\pi}}\frac{1}{\sqrt{n}}\exp\left(\frac{-x^2}{2\sigma^2}\right) + o\left({n^{-\frac{1}{2}}}\right),
\end{equation*}
for all $x\in  \R$ (with the error term uniform in $x$).
\end{remark}

\begin{remark}
Compare with the standard case of Formula (\ref{FormulaConstructionsPolya}). If we want to get comparatively large Betti numbers for $i$ near the middle index $\frac{n-1}{2}$, we need the variance $\sigma $ to be small. If we want large Betti numbers far from the center, we need it to be large.

As mentioned earlier, finding new constructions with interesting asymptotic Betti numbers in low dimensions would automatically yield improved (either particularly large or small) parameters $\sigma^2$.
\end{remark}

\begin{proof}
We define the families $\{Q^n_{d}\}_{d\in \NN}$ recursively, starting from $n=N+1$.

Let $n>N$ and suppose that families of polynomials $\{Q^m_{d}\}_{d\in \NN}$ and coefficients $x^m_i$ have already been defined for all $N<m<n$ and all $i\in \Z$.

We apply the Cooking Theorem \ref{TheoremConstructionsMainTheoremConstructions} to the ingredients $\{P^k_{d}\}_{d\in \NN}$ (for $k=1,\ldots,N$) and $\{Q^k_{d}\}_{d\in \NN}$ (for $k=N+1,\ldots,n-1$).
It yields a family $\{Q^{n}_{d}\}_{d\in \NN}$ such that
\begin{equation}\label{FormulaConstructionsInegaliteEgalite}
b_i(V_{\R\PP ^n}(Q^n_d))   \overset{n}{\geq}   \frac{1}{n}(x_i^{n-1}+x_{i-1}^{n-1} +\sum_{k=1}^{n-2}\sum_{j=0}^{i-1} x_j^k \cdot x_{i-1-j}^{n-1-k}) \cdot d^n
\end{equation}
for all $i\in \Z$.

Define $x_i^{n}:= \frac{1}{n}(x_i^{n-1}+x_{i-1}^{n-1} +\sum_{k=1}^{n-2}\sum_{j=0}^{i-1} x_j^k \cdot x_{i-1-j}^{n-1-k})$ for $i\in \Z$ (note that it is $0$ for $i\not \in \{0,\ldots,n-1\}$).
Lemma \ref{LemmaConstructionsAsmptoticallyMaximalPreserve} tells us that the family of hypersurfaces associated to $\{Q^{n}_{d}\}_{d\in \NN}$ is asymptotically maximal, that $\sum_{i\in \Z}x_i^{n} =1 $ and that the asymptotic inequality (\ref{FormulaConstructionsInegaliteEgalite}) is in fact an asymptotic equality. 

Suppose now that $\{Q^n_{d}\}_{d\in \NN}$ and $x^n_i$ have been defined for all $n>N$ and $i\in \Z$.
What is left to show is that the coefficients $x^n_i$ satisfy Formula (\ref{FormulaConstructionsBonneAsymptotiquePreuve1}).
This is, in fact, the hardest part, and a consequence of Proposition \ref{PropositionConstructionsCalculRecurrence} at the end of this section.

\end{proof}

The result below is a direct application of Theorem \ref{TheoremConstructionsBonneAsymptotique} to the two most extreme constructions known in dimension $3$, \textit{i.e.} those from Brugall{\'e}'s Theorem \ref{TheoremConstructionsFamilleAsymptotiqueBrugalle}.

\newtheorem*{thmTricheConstructions5}{Theorem \ref{CorollaryConstructionsBonneAsymptotique}}

\begin{thmTricheConstructions5}
For any $n\geq 3$, there exists families $\{F^{+,n}_{d}\}_{d\in \NN}$ and $\{F^{-,n}_{d}\}_{d\in \NN}$ of completely nondegenerate real Laurent polynomials in $n$ variables such that the Newton polytope $\Delta(F^{\pm,n}_{d})$ is $S^n_d$, as well as reals $c^n_i,d^n_i\in \R_{\geq 0}$ for any $i\in\Z$, such that for $i=0,\ldots,n-1$, we have
$$b_i(V_{\R\PP ^n}(F^{+,n}_{d}))\overset{n}{=} c^n_i\cdot d^n$$
and
$$b_i(V_{\R\PP ^n}(F^{-,n}_{d}))\overset{n}{=} d^n_i \cdot d^n.$$
Moreover, for all $x\in \R$, we have 
$$ c^n_{\left \lfloor {\frac{n-1}{2} + x\sqrt{n}} \right \rfloor} = \frac{2}{\sqrt{\pi}}\frac{1}{\sqrt{n}}\exp\left(-4x^2\right) + o\left({n^{-\frac{1}{2}}}\right)$$
and 
$$ d^n_{\left \lfloor {\frac{n-1}{2} + x\sqrt{n}} \right \rfloor} = \frac{\sqrt{20}}{\sqrt{3\pi}}\frac{1}{\sqrt{n}}\exp\left(\frac{-20x^2}{3}\right) + o\left({n^{-\frac{1}{2}}}\right),$$
where the error terms $o\left({n^{-\frac{1}{2}}}\right)$ are uniform in $x$.
The family of projective hypersurfaces associated to each family $\{F^{\pm,n}_{d}\}_{d\in \NN}$ is also asymptotically maximal.

\end{thmTricheConstructions5}

\begin{proof}
It is a trivial application of Theorem \ref{TheoremConstructionsBonneAsymptotique} to $N=3$ and the following families of polynomials: let $\{P^k_d\}_{d\in \NN}$ be $ \{I^k_d\}_{d\in \NN}$ from Itenberg and Viro's Theorem \ref{TheoremConstructionsFamilleAsymptotiqueItenbergViro} for $k=1,2$ and let  $\{P^3_d\}_{d\in \NN}$ be $\{B^\pm_d\}_{d\in \NN}$ from Brugall{\'e}'s Theorem \ref{TheoremConstructionsFamilleAsymptotiqueBrugalle} when defining $\{F^{\pm,n}_{d}\}_{d\in \NN}$.

We can directly compute that the variance $\sigma^2$ is equal to 
$$\frac{2}{(3+1)(3+2)} \left(\frac{1}{4} +0+2\frac{1}{2}\left(\frac{1}{2}\right)^2+2\left(\frac{1}{6}+\frac{5}{24}\right)1^2 \right)= \frac{1}{12} + \frac{1}{5}\frac{5}{24} = \frac{1}{8}$$
for $\{F^{+,n}_{d}\}_{d\in \NN}$ and to
$$\frac{2}{(3+1)(3+2)} \left(\frac{1}{4} +0+2\frac{1}{2}\left(\frac{1}{2}\right)^2+2\left(\frac{1}{6}-\frac{1}{24}\right)1^2 \right)= \frac{1}{12} - \frac{1}{5}\frac{1}{24} = \frac{3}{40}$$
for $\{F^{-,n}_{d}\}_{d\in \NN}$. 
This is enough to conclude.

\end{proof}
\begin{remark}
As noted in Remark \ref{RemarkConstructionsBrugalleIntermediaire},  we can easily get for any $a\in [-\frac{1}{24},\frac{5}{24}]$ a family $\{P^3_d\}_{d\in \NN}$ of completely nondegenerate real Laurent polynomials in $3$ variables such that the Newton polytope $\Delta(P^{3}_{d})$ is $S^3_d$ and that for $i=0,1,2$, we have
$$b_i(V_{\R\PP ^n}(F^{+,n}_{d}))\overset{n}{=} x^3_i\cdot d^n$$
with $x^3_0=x^3_2 = \frac{1}{6}+a$ and $x_1^3=\frac{4}{6}-2a$.
The same reasoning as in Theorem \ref{CorollaryConstructionsBonneAsymptotique} applied to  $ \{I^1_d\}_{d\in \NN}$, $ \{I^2_d\}_{d\in \NN}$ and such a family $\{P^3_d\}_{d\in \NN}$ yields a family $\{Q^n_d\}_{d\in \NN}$ for any $n\geq 4$ such that
$$b_i(V_{\R\PP ^n}(Q^{n}_{d}))\overset{n}{=} x^n_i\cdot d^n$$
and that
$$ x^n_{\left \lfloor {\frac{n-1}{2} + x\sqrt{n}} \right \rfloor} = \frac{1}{\sqrt{2\pi}\sqrt{\frac{1}{12} + \frac{a}{5}}}\frac{1}{\sqrt{n}}\exp\left(-\frac{x^2}{2(\frac{1}{12} + \frac{a}{5})}\right) + o\left({n^{-\frac{1}{2}}}\right).$$

Observe also that 
$$\frac{x^n_{\left \lfloor {\frac{n-1}{2} + x\sqrt{n}}\right \rfloor}}{a^n_{\left \lfloor {\frac{n-1}{2} + x\sqrt{n}}\right \rfloor}} = \frac{\sqrt{\frac{1}{12}}}{\sqrt{\frac{1}{12} + \frac{a}{5}}}\exp\left(-\frac{x^2}{2}\left( \frac{1}{\frac{1}{12} + \frac{a}{5}}- \frac{1}{\frac{1}{12}} \right) \right) +o(1)$$ is asymptotically strictly greater than 1 for $|x|<\sqrt{\frac{\ln(\frac{1}{12}) - \ln(\frac{1}{12}+\frac{a}{5})}{\frac{1}{\frac{1}{12} + \frac{a}{5}} -12}}$ and $a<0$ as well as for $|x|>\sqrt{\frac{\ln(\frac{1}{12}) - \ln(\frac{1}{12}+\frac{a}{5})}{\frac{1}{\frac{1}{12} + \frac{a}{5}} -12}}$ and $a>0$. The only $x$ for which we cannot get a strict asymptotic (in $n$) inequality are $\pm \lim_{a\rightarrow 0} \sqrt{\frac{\ln(\frac{1}{12}) - \ln(\frac{1}{12}+\frac{a}{5})}{\frac{1}{\frac{1}{12} + \frac{a}{5}} -12}}=\pm \frac{1}{12}$.

\end{remark}

The rest of the section is devoted to the proof of the following proposition, which describes the asymptotic behaviour of the coefficients $x^n_i$ recursively defined in the proof of Theorem \ref{TheoremConstructionsBonneAsymptotique}.
\begin{proposition}\label{PropositionConstructionsCalculRecurrence}
Let $N\geq1$.
For $n=1,\ldots,N$ and $i\in \Z$, let $x_i^n \in \R_{\geq 0}$ be such that $x_i^n=x^n_{\frac{n-1}{2}-i}$, that $\sum_{i\in \Z} x_i^n=1$ and that $x_i^n=0$ for $i\not\in\{0,\ldots,n-1\}$.

Recursively define $x_i^n\in \R_{\geq 0}$ for all $n>N$ as
$$x_i^{n}:= \frac{1}{n}\left(x_i^{n-1}+x_{i-1}^{n-1} +\sum_{k=1}^{n-2}\sum_{j=0}^{i-1} x_j^k \cdot x_{i-1-j}^{n-1-k}\right)$$ for $i\in \Z$ .

Define
$$\sigma ^2 :=\frac{2}{(N+1)(N+2)} \left(\frac{1}{4} +\sum_{k=1}^N\sum_{i=0}^{k-1}x^k_i\left(i- \frac{k-1}{2}\right)^2 \right). $$

Then
\begin{equation}\label{FormulaConstructionsBonneAsymptotiquePreuve2}
x^n_m = \frac{1}{\sigma\sqrt{2\pi}}\frac{1}{\sqrt{n}} \exp\left(-\frac{\left(\frac{n-1}{2} - m\right)^2}{2n\sigma^2}\right) +o\left(n^{-\frac{1}{2}}\right),
\end{equation}
where the $o(1)$ error term is uniform in $m$.
\end{proposition}

\begin{proof}
The core idea is to see the functions $x^n:i\mapsto x^n_i$ as discrete distributions, in order to apply probabilistic techniques (in particular a variant of the Local Limit Theorem).

Define $x_{-1}^{0}= x_0^0=\frac{1}{2}$ and $x_i^0=0$ for all $i\in \Z \backslash \{-1,0\}$. This formal trick allows us to rewrite $x_i^{n}$ as
\begin{equation*}
x^n_i =   \frac{1}{n}\sum_{k=0}^{n-1}\sum_{j\in \Z} x_j^k \cdot x_{i-1-j}^{n-1-k}
\end{equation*}
for all $i\in \Z$ and $n\geq N+1$.
Each family of coefficients $\{x^n_i\}_{i\in \Z}$ (for $n\geq 0$) defines a discrete distribution on $\Z$; the sum $\sum_{j\in \Z} x_j^k \cdot x_{i-1-j}^{n-1-k}$ is simply the probability density of the convolution of two such distributions in $i-1$. 
It can be directly checked that each distribution $\{x^n_i\}_{i\in \Z}$ is symmetric in $\frac{n-1}{2}$.

We recursively define distributions $\{\tilde{x}^n\}_{n\geq 1}$ over $\frac{1}{2}\Z$  thus: for $k=1,\ldots,N+1$, we set
\begin{equation}\label{FormulaConstructionsChangementDeVariables}
\tilde{x}^k_i=x^{k-1}_{i+\frac{k-2}{2}}.  
\end{equation}
Note the shifts in both indices.
Assume now that $\tilde{x}^k$ has been defined for any $k\leq n$ (for some $n\geq N+1$). 
Let $X_1,\ldots,X_n$ be independent random variables such that the probability density function of $X_k$ on $\R$ is $\tilde{x}^k$. Define a random variable $X_{n+1}$ as follows:
$$X_{n+1}=X_K+X_{n+1-K},$$
where $K$ is a uniform random variable on the set $\{1,\ldots,n\}$. Define $\tilde{x}^{n+1}$ as the probability density function of $X_{n+1}$ on $\frac{1}{2}\Z$. 

Hence for any $i\in \frac{1}{2}\Z$, we have 

$$\tilde{x}^{n+1}_i = \frac{1}{n}\sum_{k=1}^{n}\sum_{j\in \Z} \tilde{x}_j^k \cdot \tilde{x}_{i-j}^{n+1-k}.  $$
It is then trivial to show by induction that Formula (\ref{FormulaConstructionsChangementDeVariables}) holds for all $k\geq 1$. In particular,  $\tilde{x}^k$ takes non-trivial values only on $\frac{1}{2}+\Z$ for odd $k$, and only on $\Z$ for even $k$.

Notice that the variance of the distribution $\tilde{x}^{k}$ is the same as that of $x^{k-1}$, which is $\sum_{i\in \Z}x^{k-1}_i\left(i- \frac{k-2}{2}\right)^2$ (this is equal to $\frac{1}{4}$ for $k=1$), and that the distribution is symmetric in $0$. 
Now let $\{\tilde{X}^r_i | r\in \NN, i=1,\ldots, N+1\}$ be a family of independent random variables such that $\tilde{X}^r_i$ admits $\tilde{x}_i$ as a probability density function. For any $i=1,\ldots,N+1$, we also recursively define random variables $\alpha_i^n$ as follows: let $\alpha_i^k:=0$ for any $k\in \{1,\ldots,N+1\} \backslash \{i\}$ and $\alpha_i^i:=1$. For any $n\geq N+1$, $\alpha_i^{n+1} := \tilde{\alpha}_i ^{K} + \tilde{\alpha}_i^{n+1-K}$, where $K$ is a uniform random variable on the set $\{1,\ldots,n\}$ and $\{\tilde{\alpha}_i ^k\}_{k=1,\ldots,n}$ is a family of independent variables such that $\tilde{\alpha}_i ^k$ follows the same distribution as $\alpha_i ^k$.

Intuitively, $\alpha^{n+1}_i$ can be understood as such: consider the set $\{1,\ldots,n+1\}$. If $n+1\leq N+1$, you are done. Otherwise, randomly split it into two sets $\{1,\ldots,K\}$ and $\{K+1,\ldots,n+1\}$, and proceed similarly with each of these. Continue this procedure until you have obtained a partition of $\{1,\ldots,n+1\}$ into sets of cardinal less than or equal to $N+1$; the random variable $\alpha^{n+1}_i$ counts the number of sets of cardinal $i$ in that partition.

Let $Z_{n+1}$ be a random variable of distribution $\tilde{x}_{n+1}$, for $n\geq N+1$. Based on what we have seen of $\tilde{x}_{n+1}$, the variable $Z_{n+1}$ can be chosen so that there are independent random variables $\tilde{X}_1, \ldots,\tilde{X}_n$ (such that $\tilde{X}_k$ is of distribution $\tilde{x}_k$) and a uniform random variable $K$ on the set $\{1,\ldots,n\}$ such that $Z_{n+1}=\tilde{X}_K+\tilde{X}_{n+1-K}$. Moreover, each random variable $\tilde{X}_k$ can be chosen so that it satisfies the same condition (relative to the appropriate indices), as long as $k>N+1$. 
Hence $Z_{n+1}$ can be chosen (since we only concern ourselves with probability density functions in the statement of the lemma, and not particular random variables) so that by repeatedly decomposing it in that manner, it can be written as
$$Z_{n+1} = \sum _{i=1}^{N+1} \sum_{r=1}^{\alpha_i^{n+1}} \tilde{X}_i^r$$
if $n\geq N+1$.

Using Lemma \ref{LemmaConstructionsVAIndices} below (notice that there is a shift $N \rightarrow N+1$ due to the wording of the lemma's statement), we know that $\mathbb{E}[\frac{\alpha_i^{n}}{n}]=\frac{2}{(N+1)(N+2)}$ and $\text{Var}(\frac{\alpha_i^{n}}{n}) = \frac{C(N+1)}{n} \xlongrightarrow {n\rightarrow \infty} 0$ for $i=1,\ldots,N+1$.
Consequently, the sequence of random variables $\{\frac{\alpha_i^{n}}{n}\}_{n\in \NN}$ converges in probability to $a(N):=\frac{2}{(N+1)(N+2)}$.

We now prove some variant of the Local Limit Theorem to get Formula (\ref{FormulaConstructionsBonneAsymptotiquePreuve2}).
The main difficulty lies in the fact that the random variables  $\alpha_1^n, \ldots, \alpha_{N+1}^n$ are not independent.

We first compute the limit as $n\rightarrow \infty$ of the characteristic function of $\frac{Z_n}{\sqrt{n}}$.
We denote $\alpha^n:=(\alpha_1^n,\ldots,\alpha_{N+1}^n)$. The random variable $\alpha^n$ takes value in $\NN^{N+1}$.
We have
\begin{align*}
 \mathbb{E}\left[\exp\left(it\frac{1}{\sqrt{n}} Z_n\right)\right] =  &\sum_{k=(k_1,\ldots,k_{N+1})\in \NN^{N+1}}\PP (\alpha^n =k)\mathbb{E}\left[\exp\left(it\frac{1}{\sqrt{n}} Z_n\right)|\alpha^n = k\right]\\
& =\sum_{k\in \NN^{N+1} }\PP (\alpha^n =k)\prod _{i=1}^{N+1} \mathbb{E}\left[\exp \left(it \frac{1}{\sqrt{n}} \sum_{r=1}^{k_i} \tilde{X}_i^r \right)\right] \\
& =\sum_{k\in \NN^{N+1} }\PP (\alpha^n =k)\prod _{i=1}^{N+1} \left(1-\frac{t^2\text{Var}(\tilde{X}_i)}{2n}+ o\left(\frac{t^2}{n}\right)\right)^{  k_i}
\end{align*}
by independence of the random variables $\tilde{X}_i^r$ and Taylor expansion (where as above, $\tilde{X}_i$ is any random variable of law $\tilde{x}^i$).

Moreover, for any $\epsilon >0$, define $\NN^{N+1}_{\epsilon,n} :=\{(k_1,\ldots,k_{N+1}) \in \NN^{N+1} |\; | \frac{k_i}{n} - a(N)|<\epsilon a(N) \text{ for } i =1,\ldots, N+1  \}$. By hypothesis, for a given $\epsilon >0$, we have $\lim_{n\rightarrow \infty} \PP (\alpha^n \in \NN^{N+1}_{\epsilon,n} ) =1$.

Let us show that for any $t\in \R$, the characteristic function $\mathbb{E}\left[\exp\left(i\frac{t}{\sqrt{n}} Z_n\right)\right]$ converges to $\exp(-\frac{t^2}{2}\sum_{i=1}^{N+1}a(N) \text{Var}(\tilde{X}_i))$ as $n\rightarrow \infty$ . Indeed, choose $\epsilon >0$ and consider
\begin{align*}
& \mathbb{E}\left[\exp\left(i\frac{t}{\sqrt{n}} Z_n\right)\right] = &\PP (\alpha^n \not \in \NN_{\epsilon,n} ^{N+1} ) \mathbb{E}\left[\exp\left(i\frac{t}{\sqrt{n}} Z_n\right)|\alpha^n \not \in \NN_{\epsilon,n} ^{N+1} \right] + \\ & & \PP (\alpha^n \in \NN_{\epsilon,n} ^{N+1} ) \mathbb{E}\left[\exp\left(i\frac{t}{\sqrt{n}} Z_n\right)|\alpha^n  \in \NN_{\epsilon,n} ^{N+1} \right].
\end{align*}
The first term converges to $0$ as $n\rightarrow \infty$. Moreover, 
{\small
\begin{align*}
&  \PP (\alpha^n \in \NN_{\epsilon,n} ^{N+1} ) \mathbb{E}\left[\exp\left(i\frac{t}{\sqrt{n}} Z_n\right)|\alpha^n  \in \NN_{\epsilon,n} ^{N+1} \right] = \sum_{k\in \NN^{N+1}_{\epsilon,n}}\PP (\alpha^n =k)\mathbb{E}\left[\exp\left(i\frac{t}{\sqrt{n}} Z_n\right)|\alpha^n = k\right] =\\
&\sum_{k\in \NN^{N+1}_{\epsilon,n} }\PP (\alpha^n =k)\prod _{i=1}^{N+1} \left(1-\frac{t^2\text{Var}(\tilde{X}_i)}{2n}+ o\left(\frac{t^2}{n}\right)\right)^{  k_i} =\\
&\sum_{k\in \NN^{N+1}_{\epsilon,n} }\PP (\alpha^n =k)\prod _{i=1}^{N+1} \left(1-\frac{t^2\text{Var}(\tilde{X}_i)a(N)}{2na(N)}+ o\left(\frac{t^2}{n}\right)\right)^{  na(N)}  \cdot \\       &\cdot\left(\left(1-\frac{t^2\text{Var}\tilde{X}_i)a(N)}{2na(N)}+ o\left(\frac{t^2}{n}\right)\right)^{ na(N)}\right)^{ \frac{k_i - na(N) }{na(N) }} =\\
&\prod _{i=1}^{N+1} \left(1-\frac{t^2\text{Var}(\tilde{X}_i)a(N)}{2na(N)}+ o\left(\frac{t^2}{n}\right)\right)^{  na(N)} \Bigg[ \PP(\alpha^n \in \NN_{\epsilon,n} ^{N+1})  + \\
&  +  \sum_{k\in \NN^{N+1}_{\epsilon,n} }\PP (\alpha^n =k) \cdot \Bigg(\prod _{i=1}^{N+1}\Bigg(\left(1-\frac{t^2\text{Var}(\tilde{X}_i)a(N)}{2na(N)}+ o\left(\frac{t^2}{n}\right)\right)^{ na(N)}\Bigg)^{ \frac{k_i - na(N) }{na(N)}} -1 \Bigg) \Bigg].
\end{align*}}
As $n\rightarrow \infty$, we have $\PP(\alpha^n \in \NN_{\epsilon,n} ^{N+1}) \longrightarrow 1$ and $\prod _{i=1}^{N+1} \left(1-\frac{t^2\text{Var}(\tilde{X}_i)a(N)}{2na(N)}+ o\left(\frac{t^2}{n}\right)\right)^{  na(N)}\longrightarrow \exp(-\frac{t^2}{2}\sum_{i=1}^{N+1} \text{Var}(\tilde{X}_i)a(N))$. Moreover, $| \frac{k_i - na(N) }{na(N) }|<\epsilon$ for any $i=1,\ldots,N+1$ and any $k\in \NN_{\epsilon,n} ^{N+1}$. Hence there exists a function $f:\R_{>0}\longrightarrow \R_{>0}$ such that for $n$ large enough,
$$ \left|\prod _{i=1}^{N+1}\left(\left(1-\frac{t^2\text{Var}(\tilde{X}_i)a(N)}{2na(N)}+ o\left(\frac{t^2}{n}\right)\right)^{ na(N)}\right)^{ \frac{k_i - na(N) }{na(N)}} -1 \right| <f(\epsilon) $$
for any $\epsilon>0$ and $k\in \NN _{\epsilon,n} ^{N+1}$, and such that $\lim_{\epsilon \rightarrow 0} f(\epsilon)=0$.

As $ \epsilon>0$ can be chosen arbitrarily small,  this shows that $\lim _{n\rightarrow \infty} \mathbb{E}[\exp(i\frac{t}{\sqrt{n}} Z_n)]= \exp(-\frac{t^2}{2}a(N)\sum_{i=1}^{N+1} \text{Var}(\tilde{X}_i))$.
As in the statement of the proposition, let us write $$\sigma^2=a(N)\sum_{i=1}^{N+1} \text{Var}(\tilde{X}_i) =\frac{2}{(N+1)(N+2)} \left(\frac{1}{4} +\sum_{k=1}^N\sum_{i=0}^{k-1}x^k_i\left(i- \frac{k-1}{2}\right)^2 \right). $$

Note that using L\'evy's continuity theorem, this already shows that the sequence of distributions $\{\tilde{x}^n\}_{n\in \NN}$ converges in distribution to a normal distribution $\mathcal{N}(0,a(N)\sum_{i=1}^N \text{Var}(\tilde{X}_i))$.
As we want a local result, some additional work is still needed.

The remainder of our proof is inspired by the presentation of the Discrete Local Limit Theorem on Terence Tao's blog (\cite{Tao}).

By definition, for any $m\in \Z$, we have
$$x^n_{m}=\tilde{x}^{n+1}_{m-\frac{n-1}{2}}=\PP\left(Z_{n+1}=m-\frac{n-1}{2}\right)=\PP\left(Z_{n+1}+\frac{n-1}{2}=m\right). $$

As $Z_{n+1}+\frac{n-1}{2}$ only takes values in $\Z$ and $m$ is also an integer, we can write
$$ 1_{Z_{n+1}+\frac{n-1}{2} =m}=\frac{1}{2\pi}\int_{-\pi}^\pi \exp\left(it\left(Z_{n+1}+\frac{n-1}{2}\right)\right) \exp\left(-itm\right)dt,$$
which implies (using Fubini's theorem) that
$$\PP\left(Z_{n+1}+\frac{n-1}{2}=m\right) = \frac{1}{2\pi}\int_{-\pi}^\pi \mathbb{E}\left[\exp\left(itZ_{n+1}\right) \right]e^{it\frac{n-1}{2}} e^{-itm}dt $$
and then
\begin{equation*}
\begin{gathered}
\sqrt{n+1}\PP\left(Z_{n+1}+\frac{n-1}{2}=m\right) = \\ \frac{1}{2\pi}\int_{-\sqrt{n+1}\pi}^{\sqrt{n+1}\pi} \mathbb{E}\left[\exp \left(i\frac{x}{\sqrt{n+1}}Z_{n+1}\right) \right]e^{i\frac{x}{\sqrt{n+1}}\frac{n-1}{2}} e^{-i\frac{x}{\sqrt{n+1}}m}dx = \\
\frac{1}{2\pi}\int_{\R} 1_{|x|\leq \sqrt{n+1}\pi}\mathbb{E}\left[\exp \left(i\frac{x}{\sqrt{n+1}}Z_{n+1}\right) \right]e^{i\frac{x}{\sqrt{n+1}}\frac{n-1}{2}} e^{-i\frac{x}{\sqrt{n+1}}m}dx
\end{gathered}
\end{equation*}
using a change of variables.

We want to show that this expression converges to 
$$ \frac{1}{2\pi}\int_{\R} e^{-\frac{x^2 \sigma ^2}{2}}e^{i\frac{x}{\sqrt{n+1}}\frac{n-1}{2}} e^{-i\frac{x}{\sqrt{n+1}}m}dx = \frac{1}{\sigma\sqrt{2\pi}} \exp\left(-\frac{\left(\frac{n-1}{2} - m\right)^2}{2(n+1)\sigma^2}\right)$$
uniformly in $m$ (\textit{i.e.} that the difference between the two expressions is an $o(1)$ error term that can be upper-bounded uniformly in $m$).
 
We only need to show that
$$ \frac{1}{2\pi}\int_{\R} \left|1_{|x|\leq \sqrt{n+1}\pi}\mathbb{E}\left[\exp \left(i\frac{x}{\sqrt{n+1}}Z_{n+1}\right) \right] - e^{-\frac{x^2 \sigma ^2}{2}}  \right| dx$$
converges to $0$.

As we have proved above that $1_{|x|\leq \sqrt{n+1}\pi}\mathbb{E}\left[\exp \left(i\frac{x}{\sqrt{n+1}}Z_{n+1}\right) \right]$ converges (for a given $x\in \R$) to $ e^{-\frac{x^2 \sigma ^2}{2}} $ , we want to show that there exists an absolutely integrable function that dominates $\left|1_{|x|\leq \sqrt{n+1}\pi}\mathbb{E}\left[\exp \left(i\frac{x}{\sqrt{n+1}}Z_{n+1}\right) \right]  \right|$ so that we can conclude using the dominated convergence theorem (as  $x\mapsto e^{-\frac{x^2 \sigma ^2}{2}}$ is clearly absolutely integrable).

As above, we write
\begin{align*}
& \mathbb{E}\left[\exp\left(i \frac{x}{\sqrt{n+1}}Z_{n+1}\right)\right] =  \sum_{k\in \NN^{N+1} }\PP (\alpha^n =k)\prod _{i=1}^{N+1} \mathbb{E}\left[\exp \left(i \frac{x}{\sqrt{n+1}} \sum_{r=1}^{k_i} \tilde{X}_i^r \right)\right].
\end{align*}

For $i=1,\ldots,N+1$, either $\text{Var}(\tilde{X}_i)=0$ and the term $\mathbb{E}\left[\exp \left(i \frac{x}{\sqrt{n+1}} \sum_{r=1}^{k_i} \tilde{X}_i^r \right)\right]$ is always equal to $1$, or $\text{Var}(\tilde{X}_i)\neq 0$ and we write it as $\left(1-\frac{x^2\text{Var}(\tilde{X}_i)}{2(n+1)}+ o\left(\frac{x^2}{n+1}\right)\right)^{  k_i}$.

There exist two constants $\delta>0$ and $C_1>0$ such that 
$$\left|1-\frac{x^2\text{Var}(\tilde{X}_i)}{2(n+1)}+ o\left(\frac{x^2}{n+1}\right)\right|<e^{-C_1\frac{x^2}{n+1}}\leq 1$$
for all $|x|<\delta\sqrt{n+1}$ and all $i$ such that $\text{Var}(\tilde{X}_i)\neq 0$.
As $\text{Var}(\tilde{X}_1)=\frac{1}{4}$, there is at least one such $i$.

Define
$$\Omega_n:=\{k\in \NN^{N+1} | k_1\geq\frac{a(N)}{2}n\}.$$
Using Chebyshev's inequality and the fact that $\mathbb{E}[\frac{\alpha^{n+1}_1}{n+1}]=a(N)$ and $\text{Var}(\frac{\alpha^{n+1}}{n+1})=\frac{C(N+1)}{n+1}$, we see that 
$\PP(\alpha^{n+1} \not\in \Omega_{n+1})\leq \frac{C_2}{n+1}$ for some constant $C_2>0$.

Hence
{\small
\begin{align}\label{FormulaConstructionsBorneUniforme1}
&\left|1_{|x|\leq \delta\sqrt{n+1}}\mathbb{E}\left[\exp \left(i\frac{x}{\sqrt{n+1}}Z_{n+1}\right) \right]  \right| \leq \nonumber\\
&1_{|x|\leq \delta\sqrt{n+1}}\PP(\alpha^{n+1} \not \in \Omega_{n+1}) +
1_{|x|\leq \delta\sqrt{n+1}}\sum_{k\in \Omega_{n+1} }\PP (\alpha^{n+1} =k)\prod _{i=1}^{N+1}\left| \mathbb{E}\left[\exp \left(i \frac{x}{\sqrt{n+1}} \sum_{r=1}^{k_i} \tilde{X}_i^r \right)\right] \right| \leq \nonumber \\
&1_{|x|\leq \delta\sqrt{n+1}}\frac{C_2}{n+1} +  \sum_{k\in \Omega_{n+1} }\PP (\alpha^{n+1} =k) e^{-C_1\frac{x^2}{n+1}k_1}\leq 1_{|x|\leq \delta\sqrt{n+1}}\frac{C_2}{n+1} +  e^{-C_1\frac{x^2}{2}a(N)}\leq \nonumber\\
& C_2 \delta^2 \min \left(\frac{1}{x^2}, \frac{1}{\delta^2} \right) +e^{-C_1\frac{x^2}{2}a(N)},
\end{align} }
assuming (as we can) that $\delta<1$.

Suppose now that  we have $\left|\mathbb{E}\left[\exp \left(it\tilde{X}_1\right) \right]\right|  =1 $ for some $\delta\leq t\leq \pi$.
It means that $\exp \left(it\tilde{X}_1\right) $ is almost surely of constant argument, hence $t\tilde{X}_1$ almost surely takes values in $a+2\pi \Z$ for some $a\in \R$.
Thus $\tilde{X}_1$ takes values in $\frac{a}{t}+ \frac{2\pi}{t}\Z$.
But $\frac{2\pi}{t}\geq 2$, and by definition $\PP(\tilde{X}_1=\frac{1}{2})=\PP(\tilde{X}_1=-\frac{1}{2})=\frac{1}{2}$ (remember that $\tilde{X}_1$ comes from the "artificial" distribution $x^0$ defined at the beginning of the proof).
This is impossible.
Hence we have $\left|\mathbb{E}\left[\exp \left(it\tilde{X}_1\right) \right]  \right|<1 $ for any $\delta\leq t\leq \pi$, and by continuity and compacity there exists a constant $0<C_3<1$ such that $\left|\mathbb{E}\left[\exp \left(it\tilde{X}_1\right) \right]  \right|<C_3<1 $ for any $\delta\leq t\leq \pi$.

Now we can write
\begin{align*}
&\left|1_{\delta\sqrt{n+1}\leq|x|\leq \pi\sqrt{n+1}}\mathbb{E}\left[\exp \left(i\frac{x}{\sqrt{n+1}}Z_{n+1}\right) \right]  \right| \leq\\
&1_{\delta\sqrt{n+1}\leq|x|\leq \pi\sqrt{n+1}}\PP(\alpha^{n+1} \not \in \Omega_{n+1})+\\
&1_{\delta\sqrt{n+1}\leq|x|\leq \pi\sqrt{n+1}}\sum_{k\in \Omega_{n+1} }\PP (\alpha^{n+1} =k)\prod _{i=1}^{N+1}\left| \mathbb{E}\left[\exp \left(i \frac{x}{\sqrt{n+1}} \sum_{r=1}^{k_i} \tilde{X}_i^r \right)\right] \right|  \\
&\leq 1_{\delta\sqrt{n+1}\leq|x|\leq \pi\sqrt{n+1}}\frac{C_2}{n+1} +  1_{\delta\sqrt{n+1}\leq|x|\leq \pi\sqrt{n+1}}\sum_{k\in \Omega_{n+1} }\PP (\alpha^{n+1} =k) C_3^{k_1}\leq\\
&1_{|x|\leq \pi\sqrt{n+1}}\frac{C_2}{n+1} +  1_{|x|\leq \pi\sqrt{n+1}}C_3^{\frac{a(N)}{2}(n+1)}\leq C_2 \pi^2 \min \left(\frac{1}{x^2}, 1 \right) +\left(C_3 ^{\frac{a(N)}{2\pi^2}}\right)^{x^2}.
\end{align*}
This, together with Formula (\ref{FormulaConstructionsBorneUniforme1}), allows us to conclude:
we have shown that 
$$x^n_m =  \frac{1}{\sigma\sqrt{2\pi}}\frac{1}{\sqrt{n+1}} \exp\left(-\frac{\left(\frac{n-1}{2} - m\right)^2}{2(n+1)\sigma^2}\right) +o(n^{-\frac{1}{2}})$$
for some error term uniform in $m$.
We can finally observe (by distinguishing the cases where $\left(\frac{n-1}{2} - m\right)^2\leq n^{\frac{3}{2}}$ from the cases where $\left(\frac{n-1}{2} - m\right)^2\geq n^{\frac{3}{2}}$) that
$$\frac{1}{\sigma\sqrt{2\pi}}\frac{1}{\sqrt{n+1}} \exp\left(-\frac{\left(\frac{n-1}{2} - m\right)^2}{2(n+1)\sigma^2}\right) = \frac{1}{\sigma\sqrt{2\pi}}\frac{1}{\sqrt{n}} \exp\left(-\frac{\left(\frac{n-1}{2} - m\right)^2}{2n\sigma^2}\right) +o(n^{-\frac{1}{2}})$$
with the error term once again uniform in $m$ .

\end{proof}

\begin{lemma}\label{LemmaConstructionsVAIndices}
Let $N\geq 1$.
For any $i=1,\ldots,N$, recursively define the random variables $\alpha_i^n$ as follows:
$\alpha_i^k:=0$ for any $k\in \{1,\ldots,N\} \backslash \{i\}$ and $\alpha_i^i:=1$. For any $n\geq N$, $\alpha_i^{n+1} := \tilde{\alpha}_i ^{K} + \tilde{\alpha}_i^{n+1-K}$, where $K$ is a uniform random variable on the set $\{1,\ldots,n\}$ and $\{\tilde{\alpha}_i ^k\}_{k=1,\ldots,n}$ is a family of independent variables such that $\tilde{\alpha}_i ^k$ follows the same distribution as $\alpha_i ^k$.

Then for any $n\geq N+1$, we have 
$$\mathbb{E}[\alpha_i^{n}] = \frac{2n}{(N+1)N} $$
and
$$\text{Var}(\alpha_i^{n}) =C(N)n ,$$
where $C(N)$ is some constant that only depends on $N$.
\end{lemma}

\begin{proof}

Let $i\in\{1,\ldots,N\} $ and observe that if $n\geq N$, then
$$\mathbb{E}[\alpha_i^{n+1}]=\frac{1}{n}\left(\sum_{k=1}^n \mathbb{E}[\tilde{\alpha}_i^{k}]+\mathbb{E}[\tilde{\alpha}_i^{n+1-k}] \right) =\frac{2}{n}\sum_{k=1}^n \mathbb{E}[\alpha_i^{k}] .$$
Hence
$$\sum_{k=1}^{n+1} \mathbb{E}[\alpha_i^{k}] = \sum_{k=1}^n \mathbb{E}[\alpha_i^{k}] + \mathbb{E}[\alpha_i^{n+1}] = \frac{n+2}{n}\sum_{k=1}^n \mathbb{E}[\alpha_i^{k}]$$
and necessarily 
$$\sum_{k=1}^{n+1} \mathbb{E}[\alpha_i^{k}] = \frac{n+2}{n}\frac{n+1}{n-1}\ldots \frac{N+3}{N+1} \frac{N+2}{N}\sum_{k=1}^{N} \mathbb{E}[\alpha_i^{k}] = \frac{(n+2)(n+1)}{(N+1)N}\sum_{k=1}^{N} \mathbb{E}[\alpha_i^{k}].$$
Moreover, we get from the definition of the random variables $\alpha^n_i$ that $\sum_{k=1}^{N} \mathbb{E}[\alpha_i^{k}]=1$: thus $\sum_{k=1}^{n+1} \mathbb{E}[\alpha_i^{k}] = \frac{(n+2)(n+1)}{(N+1)N},$ and finally
$$\mathbb{E}[\alpha_i^{n+1}] = \sum_{k=1}^{n+1} \mathbb{E}[\alpha_i^{k}] -\sum_{k=1}^{n} \mathbb{E}[\alpha_i^{k}] = \frac{(n+2)(n+1)}{(N+1)N} - \frac{(n+1)n}{(N+1)N} = \frac{2(n+1)}{(N+1)N} $$
for any $n\geq N$.

Let us now compute the variance of $\alpha_i^{n+1}$.
First observe that
\begin{align*}
&  \mathbb{E}[(\alpha_i^{n+1})^2]  =  \frac{1}{n}\left(\sum_{k=1}^n \mathbb{E}[(\tilde{\alpha}_i^{k}+\tilde{\alpha}_i^{n+1-k})^2] \right) = \frac{2}{n}\left(\sum_{k=1}^n \mathbb{E}[(\tilde{\alpha}_i^{k})^2]+\mathbb{E}[\tilde{\alpha}_i^{k}]\mathbb{E}[\tilde{\alpha}_i^{n+1-k}] \right)=\\
&\frac{2}{n}\left(\sum_{k=1}^n \mathbb{E}[(\alpha_i^{k})^2]+\sum_{k=1}^n \frac{4k(n+1-k)}{N^2(N+1)^2} \right) = \frac{2}{n}\left(\sum_{k=1}^n \mathbb{E}[(\alpha_i^{k})^2]+ \frac{2n(n+1)(n+2)}{3N^2(N+1)^2} \right)
\end{align*}
as the variables $\tilde{\alpha}_i^{k}$ are independent from each other.

Now define for any $n\geq 1$
$$g(n+1):=\sum_{k=1}^{n+1} \mathbb{E}[(\alpha_i^{k})^2] -\frac{2(n+1)(n+2)(2n-3)}{3N^2(N+1)^2},$$
as is natural to do in such a situation.

Then
\begin{align*}
&  g(n+1) = \sum_{k=1}^{n+1} \mathbb{E}[(\alpha_i^{k})^2] -\frac{2(n+1)(n+2)(2n-3)}{3N^2(N+1)^2} =\\
& \sum_{k=1}^{n} \mathbb{E}[(\alpha_i^{k})^2] +\frac{2}{n}\left(\sum_{k=1}^n \mathbb{E}[(\alpha_i^{k})^2]+ \frac{2n(n+1)(n+2)}{3N^2(N+1)^2} \right) -\frac{2(n+1)(n+2)(2n-3)}{3N^2(N+1)^2} = \\
&\frac{n+2}{n}\sum_{k=1}^n \mathbb{E}[(\alpha_i^{k})^2] + (2-(2n-3))\frac{2(n+1)(n+2)}{3N^2(N+1)^2}=\\
&\frac{n+2}{n}\sum_{k=1}^n \mathbb{E}[(\alpha_i^{k})^2] -\frac{2(n+1)(n+2)(2n-5)}{3N^2(N+1)^2}=\\
&\frac{n+2}{n}\left(\sum_{k=1}^n \mathbb{E}[(\alpha_i^{k})^2]-  \frac{2n(n+1)(2n-5)}{3N^2(N+1)^2} \right) = \frac{n+2}{n}g(n)
\end{align*}
for any $n\geq N$, hence as before,
$$g(n+1)= \frac{(n+2)(n+1)}{(N+1)N} g(N) .$$
Thus
\begin{align*}
& \mathbb{E}[(\alpha_i^{k})^2] = g(n+1)-g(n) +\frac{2(n+1)(n+2)(2n-3)}{3N^2(N+1)^2}-\frac{2n(n+1)(2n-5)}{3N^2(N+1)^2}=\\
&\frac{2(n+1)}{(N+1)N} g(N) +\frac{2(n+1)}{3N^2(N+1)^2}((n+2)(2n-3)-n(2n-5))=\\
&\frac{2(n+1)}{(N+1)N} g(N) +\frac{4(n+1)(n-1)}{N^2(N+1)^2}.
\end{align*}

From this, we see that for any $n\geq N$,
\begin{equation*}
\begin{gathered}
\text{Var}(\alpha_i^{n+1})=\mathbb{E}[(\alpha_i^{n+1})^2] -  (\mathbb{E}[\alpha_i^{n+1}])^2 = \\
\frac{2(n+1)}{N(N+1)} g(N) +\frac{4(n+1)(n-1)}{N^2(N+1)^2} - \frac{4(n+1)^2}{N^2(N+1)^2}=\\
\frac{2(n+1)}{N(N+1)} g(N) -\frac{8(n+1)}{N^2(N+1)^2} = (n+1)C(N)
\end{gathered}
\end{equation*}
for some finite constant $C(N):=\frac{2g(N)}{N(N+1)} -\frac{8}{N^2(N+1)^2}$ that only depends on $N$.

\end{proof}


%% file: SomeExplicitComputations.tex
\section{Some explicit computations}\label{SectionExplicitComputations}

\begin{figure}
\begin{center}
\includegraphics[scale=1]{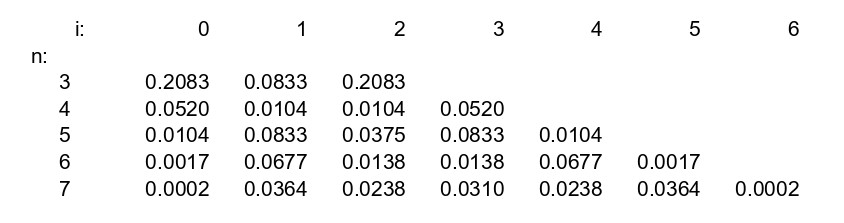}
\end{center}
\caption{The largest coefficients $t^n_i$ such that we succeeded in defining using the Cooking Theorem a family $\{Q^n_d\}_{d\in\NN}$ of completely nondegenerate real Laurent polynomials such that $\Delta(Q^n_d)=S^n_d$ and that $b_i(V_{\R\PP^n }(Q^n_d))  \protect\overset{n}{\geq}  \left(a^n_i + t^n_i \right) \cdot d^n. $}
\label{FigureConstructionsN7}
\end{figure}

Given $n\geq 3$, we can try to apply the Cooking Theorem \ref{TheoremConstructionsMainTheoremConstructions} to any combination of families of polynomials $\{P^k_d\}_{d\in\NN}$ for $k=1,\ldots, n-1$; these various combinations result in families of polynomials $\{Q^n_d\}_{d\in\NN}$ and associated hypersurfaces with a  priori distinct asymptotic Betti numbers, which we can in turn use to define new polynomials and hypersurfaces in ambient dimension $n+1$.
The total number of possibilities in dimension $n$ grows extremely fast as $n$ goes to infinity (faster than $C^{2^n}$ for some $C>1$), even with a low number of starting ingredients, \textit{i.e.} families of polynomials that are already known.
Given $n\geq 3$ and $i\in\{0,\ldots,n-1\}$, it is not yet clear how to pick the combination which will result in the largest asymptotic value for the $i$-th Betti numbers $b_i(V_{\R \PP ^n}(Q^n_d))$.

\begin{figure}
\begin{center}
\includegraphics[scale=1]{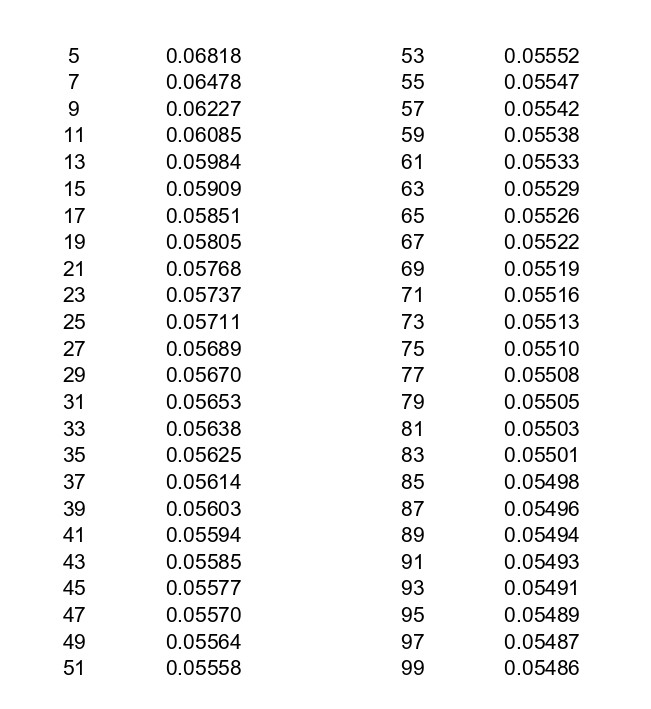}
\end{center}
\caption{For $n=5,\ldots, 99$ odd, the ratio $\frac{d_{\frac{n-1}{2}}^n-a^n_{\frac{n-1}{2}}}{a^n_{\frac{n-1}{2}}}$, where $d_{\frac{n-1}{2}}^n$ is as in Theorem \ref{CorollaryConstructionsBonneAsymptotique}. }
\label{FigureConstructionsMilieu}
\end{figure}

The author used a simple C++ program to test out each combination achievable in ambient dimension $n=4,5,6,7$ using both constructions from Theorem \ref{TheoremConstructionsFamilleAsymptotiqueBrugalle} by Brugall{\'e}  and the family of constructions from Theorem \ref{TheoremConstructionsFamilleAsymptotiqueItenbergViro} by Itenberg and Viro as building blocks.
Figure \ref{FigureConstructionsN7} shows, for $n=3,\ldots,7$ and $i\in\{0,\ldots,n-1\}$, the largest $t_i^n$ (rounded down to the $4$-th decimal) such that we were able to cook, using the Cooking Theorem \ref{TheoremConstructionsMainTheoremConstructions},  a family $\{Q^n_d\}_{d\in\NN}$ of completely nondegenerate real Laurent polynomials in $n$ variables such that $\Delta(Q^n_d)=S^n_d$, that the associated family of real projective hypersurfaces is asymptotically maximal (this is not an additional constraint, as all our ingredients are asymptotically maximal) and that 
\begin{equation*}
b_i(V_{\R\PP^n }(Q^n_d))  \overset{n}{\geq}  \left(a^n_i + t^n_i \right) \cdot d^n.   
\end{equation*}
In particular, it is enough to complete the proof of Theorem \ref{TheoremConstructionsApplicationThm} for $n\leq 7$.

In Figure \ref{FigureConstructionsMilieu}, we indicate, for $ 5\leq n\leq 99$ odd and $i\in\{0,\ldots,n-1\}$, the values of  $\frac{d_{\frac{n-1}{2}}^n-a^n_{\frac{n-1}{2}}}{a^n_{\frac{n-1}{2}}}$ (rounded down to the $5$-th decimal), where $d^n_{\frac{n-1}{2}}$ comes from Theorem \ref{CorollaryConstructionsBonneAsymptotique}.
The ratio appears to converge relatively fast to $\lim_{n\rightarrow \infty} \frac{d_{\frac{n-1}{2}}^n-a^n_{\frac{n-1}{2}}}{a^n_{\frac{n-1}{2}}}= \frac{\sqrt{10}}{3}-1 \cong 0.05409$.

%% file: Conclusion.tex
\section{Conclusion}\label{SectionConclusion}
Using  Knudsen, Mumford and Waterman's celebrated result that every lattice polytope has a dilatation that admits a primitive triangulation (see \cite{KMW}), it is easy to see that the asymptotic theorems \ref{TheoremConstructionsApplicationThm} and \ref{CorollaryConstructionsBonneAsymptotique} can be adapted to the case where the ambient space is not $\PP^n$, but a more general toric variety generated by a non-singular polytope.

More interestingly, one could try to generalize the method to obtain complete intersections of hypersurfaces with large asymptotic Betti numbers.

One could also want to find combinatorial analogs to Brugall{\'e}'s Theorem \ref{TheoremConstructionsFamilleAsymptotiqueBrugalle}, \textit{i.e.} families $\{P^3_d\}_{d\in \NN}$ of completely nondegenerate real Laurent polynomials in $3$ variables obtained using the combinatorial Patchwork and such that the Newton polytope $\Delta(P^{3}_{d})$ is $S^3_d$ and that for $i=0,\ldots,2$, we have
$$b_i(V_{\R\PP ^n}(P^3_{d}))\overset{n}{=} x^3_i\cdot d^n$$
with $x^3_0=x^3_2 = \frac{1}{6}+a$ and $x_1^3=\frac{4}{6}-2a$ for non-zero $a$.
In particular, those families would have to be asymptotically maximal.

As the Cooking Theorem  respects the combinatorial nature of the ingredients it uses, this would automatically yield combinatorial versions of Theorems \ref{TheoremConstructionsApplicationThm} and \ref{TheoremConstructionsBonneAsymptotique}.

More generally, any new interesting asymptotic constructions (not necessarily combinatorial) might potentially yield new results when used (in a clever enough way) as ingredients for the Cooking Theorem.